\documentclass[review]{elsarticle}

\usepackage{lineno,hyperref}
\modulolinenumbers[5]

\usepackage{csquotes}

\usepackage{geometry}
\usepackage{float}
\newcommand{\iu}{{i\mkern1mu}}
\usepackage{slashbox}
\usepackage{url}
\usepackage{pifont}
\usepackage{empheq}
\usepackage{bm}
\hypersetup{
     colorlinks   = true,
     citecolor    = blue
}
\usepackage[english]{babel} 
\addto\captionsenglish{}
\usepackage{amsmath}
\usepackage{amsmath,amssymb}
\usepackage{epstopdf}
\usepackage{relsize}
\usepackage{stmaryrd} 
\usepackage{subcaption}
\usepackage{multirow}
\usepackage{array,makecell}
\usepackage[squaren, Gray, cdot]{SIunits}

\usepackage{scalerel}
\def\stretchint#1{\vcenter{\hbox{\stretchto[220]{\displaystyle\int}{#1}}}}

\def\bs{\!\!}

\addto\captionsenglish{}

\usepackage{amsthm}
\renewenvironment{proof}{{\bfseries Proof.\hspace{0.3cm}}}{\qed}

\newtheorem{thm}{Theorem}[section]
\newtheorem{lem}{Lemma}[section]
\newtheorem{cor}{Corollary}[section]
\newdefinition{rmk}{Remark}[section]
\newproof{pf}{Proof}
\newproof{pot}{Proof of Theorem \ref{thm2}}
\newdefinition{definition}{Definition}
\newtheorem{post}{Postulate}[section]

\makeatletter
\def\ps@pprintTitle{%
 \let\@oddhead\@empty
 \let\@evenhead\@empty
 \def\@oddfoot{}%
 \let\@evenfoot\@oddfoot}

\begin{document}

\begin{frontmatter}

\title{Stability of Energy Stable Flux Reconstruction for the Diffusion Problem using the Interior Penalty and Bassi and Rebay II Numerical Fluxes for Linear Triangular Elements}


\author[mymainaddress]{Samuel Quaegebeur
\corref{mycorrespondingauthor}}
\cortext[mycorrespondingauthor]{Corresponding author}
\ead{samuel.quaegebeur@mail.mcgill.ca}

\author[mymainaddress]{Siva Nadarajah}

\address[mymainaddress]{Department of Mechanical Engineering, McGill University, Montreal, QC, H3A 0C3, Canada}

\begin{abstract}
The flux reconstruction (FR) method has gained popularity within the research community. The approach has been demonstrated to recover high-order methods such as the discontinuous Galerkin (DG) method. Stability analyses have been conducted for a linear advection problem leading to the energy stable flux reconstruction (ESFR) methods also named Vincent-Castonguay-Jameson-Huynh (VCJH) methods. ESFR schemes can be viewed as DG schemes with modally filtered correction fields. Using this class of methods, the linear advection diffusion problem has been shown to be stable using the local discontinuous Galerkin scheme (LDG) to compute the viscous numerical flux. This stability proof has been extended for linear triangular and tetrahedra elements. Although the LDG scheme is commonly used, it requires, on particular meshes, a wide stencil, which raises the computational cost.

As a consequence, many prefer the compact interior penalty (IP) or the Bassi and Rebay II (BR2) numerical fluxes. This article, for the first time, derives, for both schemes, a condition on the penalty term to ensure stability. Moreover the article establishes that for both the IP and BR2 numerical fluxes, the stability of the ESFR scheme is independent of the auxiliary correction field. A von Neumann analysis is conducted to study the maximal time step of various ESFR methods.
\end{abstract}


\end{frontmatter}

\newpage
\section{Introduction}

Second-order finite-volume and finite-element methods form the current heart of most, if not all, commercial computational fluid dynamics packages, in addition to in-house codes at most aerospace manufacturers. In the past two decades, discontinuous Galerkin (DG) approaches have been advanced. They offer a combination of the strengths of both the finite-volume and finite-element approaches, where the concept of a numerical flux function, to provide stability, is combined with the finite-element approach of employing high-order shape functions to represent the solution~\cite{NDG}.
A numerical flux ensures conservation across control volume faces and a number of these  numerical fluxes for the discontinuous Galerkin (DG) approach have been developed for the diffusion equation such as the Bassi-Rebay schemes (BR1~\cite{bassi_high-order_1997} and BR2~\cite{bassi_high_2000}), interior penalty (IP)~\cite{arnold_interior_1982}, local discontinuous Galerkin (LDG~\cite{cockburn_local_1998}), compact discontinuous Galerkin (CDG)~\cite{peraire_compact_2008}. For each one of these schemes, the DG method has been well-documented in terms of stability.
These methods contain a penalty term, which controls the jump of the solution and/or gradient of the solution between the cells or control volumes. A judicious choice of this parameter ensures both stability and the correct order of accuracy of the scheme.

The FR framework developed by Huynh \cite{huynh_flux_2007}, \cite{huynh_reconstruction_2009} recovers, through the use of correction functions, many high-order methods including the DG and the spectral difference methods~\cite{liu_spectral_2006}. Stability analysis for the linear advection problem have been conducted for one~\cite{vincent_new_2011}, two~\cite{castonguay_new_2012} and three~\cite{williams_energy_2014} dimensional problems, leading to the stable class of correction functions named Vincent Castonguay Jameson Huynh (VCJH) schemes. The extension from one dimension to higher dimensions has been achieved for linear simplices (triangles and tetrahedra). The mathematical proof is based on the energy of the solution and hence this class has taken the name Energy Stable Flux Reconstruction (ESFR) schemes. These stability proofs were then extended to the linear advection-diffusion problem using the LDG numerical fluxes,~\cite{castonguay_energy_2013},~\cite{williams_energy_2013},~\cite{williams_energy_2014}. They obtained the stability of the scheme by taking the penalty term to be greater than 0. However the LDG numerical fluxes may require a wide stencil~\cite{williams_energy_2013} and hence other compact methods are typically preferred. An extension of the stability of the ESFR scheme for one-dimensional problems using the compact IP and BR2 numerical fluxes has been conducted by Quaegebeur et al.~\cite{sam_article1D}. The purpose of this article is to extend the stability proof of the IP and BR2 fluxes to two dimensional problems using linear triangular elements.

This article is composed as follows: Section~\ref{sec:independent kappa} provides a proof to show that the problem, employing the IP or the BR2 numerical schemes, is independent of the auxiliary correction field; Section~\ref{sec:IP section} contains the theoretical proof of stability for the IP numerical fluxes as well as numerical verifications; Section~\ref{sec:BR2 section} presents the proof of stability for the BR2 scheme; Section~\ref{sec:VN section} demonstrates a von Neumann analysis showing the maximum time step. We conclude this paper with Section~\ref{sec:L2 errors section} where we present the $L_{2}$-errors for a given problem.
Therefore this article establishes the necessary bounds for the IP and BR2 penalty terms to ensure stability for ESFR schemes.
As the theoretical result is based on the work of Williams et al.~\cite{williams_energy_2013}, we strongly advise the reader to review~\cite{williams_energy_2013}. The current article will use similar notations in an attempt to be as comprehensible as possible.
\section{Preliminaries}\label{sec:presentation FR simplex}

In this section, we present the flux reconstruction approach for two dimensional problems using triangular elements, first introduced by Castonguay et al.\cite{castonguay_new_2012}.

Let us consider the diffusion equation
\begin{equation}\label{eq:diffusion 2D}
\dfrac{\partial u}{\partial t}=b\Delta u,\,\,\,\left(x,y\right)\in\Omega,\,\,t\in\left[0,T\right],
\end{equation}
where $x$ and $y$ are the spatial coordinates of the physical domain $\Omega$, $b$ is the diffusion parameter, $T$ is the final time and $\Delta=\left(\frac{\partial^{2} }{\partial x^{2}}+\frac{\partial^{2} }{\partial y^{2}}\right)$ is the Laplacian operator. We write the partial differential equation~\eqref{eq:diffusion 2D} as a system of two first-order equations. We introduce the operator $\mathbf{\nabla}=\begin{pmatrix}
\dfrac{\partial}{\partial x}\vspace{0.2cm}\\
\dfrac{\partial}{\partial y}
\end{pmatrix}$,
\begin{subequations}
\begin{align}
\dfrac{\partial u}{\partial t}&=\mathbf{\nabla}\cdot\mathbf{f}\left(\mathbf{q}\right),\label{eq:main eq 2D}\\
\mathbf{q}&=\mathbf{\nabla} u,\label{eq:aux eq 2D}
\end{align}
\end{subequations}
where $\mathbf{f}$ is the flux, and $\mathbf{q}$ the auxiliary variable of the problem. For a pure diffusion problem, $\mathbf{f}\left(\mathbf{q}\right)=b\mathbf{q}$.  In this article, a bold lowercase letter, $\mathbf{a}$, denotes a vector and a bold capital letter, $\mathbf{A}$ indicates a matrix.

The physical domain, $\Omega$, can be decomposed into $N_{K}$ non-overlapping linear triangular elements and thus we apply the tessalation, $\mathcal{T}_{h}=\sum_{n=1}^{N_{K}}\Omega_{n}$, where $\Omega_{n}$ denotes the element $n$. Before proceeding further on the FR procedure, we first define the computational domain and the reference element $\Omega_{s}$. We arbitrarily, chose the equilateral element as our reference element. Choosing any other reference element such as the rectangular triangle does not impact our results.

\begin{figure}[H]
\centering
\includegraphics[width=12cm]{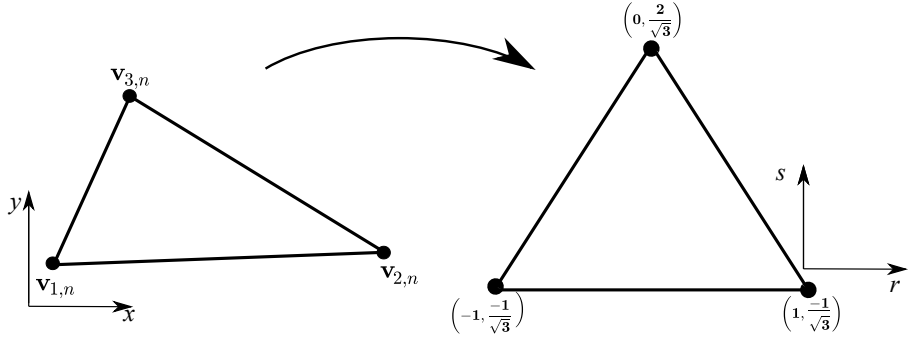}
\caption{Transformation, for triangles, from the physical space to the computational space.}
\label{fig:transformation 2D}
\end{figure}

Referring to Figure~\ref{fig:transformation 2D}, the affine mapping can be defined as
\begin{equation}\label{eq:affine mapping}
\begin{array}{lccl}
\mathcal{M}_{n}\colon &\Omega_{s}&\to& \Omega_{n}\\
&\left(r,s\right)&\mapsto& \dfrac{\left(-3r+2-\sqrt{3}s\right)}{6}\mathbf{v}_{1,n}+\dfrac{\left(2+3r-\sqrt{3}s\right)}{6}\mathbf{v}_{2,n}+\dfrac{\left(2+2\sqrt{3}s\right)}{6}\mathbf{v}_{3,n},
\end{array}
\end{equation}
where $\Omega_{s}=\left\lbrace\left(r,s\right)\mid -1\leq r\leq 1\, ,\, \frac{-1}{\sqrt{3}}\leq s\frac{2}{\sqrt{3}}\, ,\, \left|\frac{3}{\sqrt{3}}r+s\right|\leq \frac{2}{\sqrt{3}}\right\rbrace$, $\mathbf{v}_{1,n}$, $\mathbf{v}_{2,n}$ and $\mathbf{v}_{3,n}$ are the three vertices of the triangle $\Omega_{n}$.

With this mapping, we define $\hat{u}$, $\mathbf{\hat{f}}$ and $\mathbf{\hat{q}}$ as the computational quantities of $u$, $\mathbf{f}$ and $\mathbf{q}$. We then define the Jacobian of the triangle $\Omega_{n}$ and its determinant by $\mathbf{J}_{n}=\begin{pmatrix}
x_{r}&x_{s}\\
y_{r}&y_{s}
\end{pmatrix}$ and $\left|J_{n}\right|=Det\left(\mathbf{J_{n}}\right)$. We use the notation $a_{b}$ to indicate the partial derivative of $a$ with respect to $b$. The elements are linear triangles, therefore $\mathbf{J_{n}}$ is constant within $\Omega_{n}$. By convention, the computational quantites are calculated as
\begin{eqnarray}
\mathbf{\hat{\nabla}}&=&\mathbf{J}_{n}^{T}\mathbf{\nabla},\label{eq:nabla2D}\\
\hat{u}_{n}&=&\left|J_{n}\right| u_{n},\label{eq:u2D}\\
\mathbf{\hat{f}}_{n}&=&\left|J_{n}\right|\mathbf{J}_{n}^{-1}\mathbf{f}_{n},\label{eq:f2D}\\
\mathbf{\hat{q}}_{n}&=&\mathbf{\hat{\nabla}}\hat{u}=\left|J_{n}\right|\mathbf{J}_{n}^{T}\mathbf{q}_{n},\label{eq:q2D}\\
\mathbf{\hat{n}}^{e}&=&\dfrac{1}{\left|J_{n}\right|}\left|J^{e}\right|\mathbf{J}^{T}\mathbf{n}^{e},\label{eq:n2D}
\end{eqnarray}
where $\mathbf{n}^{e}$ is the outward normal to the edge $e$, $\left|J^{e}\right|$ is the length of the edge, and $\hat{\nabla}=\begin{pmatrix}
\frac{\partial}{\partial r}\\
\frac{\partial}{\partial s}
\end{pmatrix}$ is the divergence operator in the computational domain. We also note $\mathbf{r}=\left(r,s\right)$.

With these conventional transformations, we obtain
\begin{subequations}
\begin{align}
\dfrac{\partial \hat{u}}{\partial t}&=\hat{\mathbf{\nabla}}\cdot\hat{\mathbf{f}}\left(\hat{\mathbf{q}}\right),\label{eq:main compu eq 2D}\\
\hat{\mathbf{q}}&=\hat{\mathbf{\nabla}}u,\label{eq:aux compu eq 2D}
\end{align}
\end{subequations}
where, $\hat{\mathbf{f}}_{n}=\mathbf{J}_{n}^{-1}\mathbf{J}_{n}^{-T}\hat{\mathbf{q}}_{n}$.

The solution $u_{n}$ is a two-dimensional polynomial of degree $p$ $\left(u_{n}\in P_{p}\left(\Omega_{s}\right)\right)$. It is calculated through $N_{p}=\frac{\left(p+2\right)\left(p+1\right)}{2}$ solution points (SP), represented in Figure~\ref{fig:equilateral element},
\begin{equation}
u_{n}\left(\mathbf{r}\right)=\sum_{i=1}^{N_{p}}\tilde{u}_{i}l_{i}\left(\mathbf{r}\right),
\end{equation}
where $l_{i}$ is a two-dimensional Lagrange polynomial of degree $p$, $l_{i}\in P_{p}\left(\Omega_{s}\right)$ (note that the dimension of $P_{p}\left(\Omega_{s}\right)$ is $N_{p}$). $\left(\tilde{u}_{i}\right)_{i\in\llbracket 1,N_{p}\rrbracket}$ are the nodal expansion coefficients associated to the solution points $\left(r_{i}\right)_{i\in\llbracket 1,N_{p}\rrbracket}$ (interval of the type $\llbracket \cdot ,\cdot\rrbracket$ denotes an integer interval). Conversely to the 1D problem, there is no analytical expression for the Lagrange polynomial except for certain nodal distributions (e.g. equi-spaced). This issue is resolved by considering the solution in modal form
\begin{equation}\label{eq:modal form 2D}
u_{n}\left(\mathbf{r}\right)=\sum_{i=1}^{N_{p}}\tilde{u}^{m}_{i}L_{i}\left(\mathbf{r}\right),
\end{equation}
where $\tilde{u}^{m}_{i}$ are the modal coefficients and $\left(L_{i}\right)_{i\in\llbracket 1,N_{p}\rrbracket}$ forms an orthonormal basis on the equilateral reference element. Initially, at time $t_{0}$, we have the solution on element $\Omega_{n}$, $u_{n}^{t_{0}}$ at its $N_{p}$ solution points (the $N_{p}$ nodal coefficients). If we want to calculate the solution elsewhere, we evaluate the solution under its modal form. To do this, we compute the Vandermonde matrix~\cite{NDG}, $\mathcal{V}$, and perform the operation $\tilde{u}^{m}=\mathcal{V}^{-1}\tilde{u}$. Having the modal coefficients $\tilde{u}^{m}$ and the analytical expression for the orthonormal basis $L_{i}$, we can compute $u_{n}$ at any point with~\eqref{eq:modal form 2D}. The orthonormal basis $\left(L_{i}\right)$ is of utmost importance as it enables the evaluation of all quantities of the problem.

\begin{figure}[H]
\centering
\includegraphics[width=7cm]{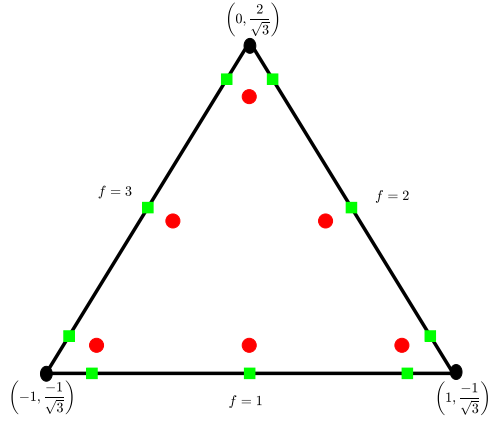}
\caption{Reference element for $p=2$, the green squares represent the Gauss-Legendre Flux Points (FP), the red circles represent the solution points (SP) and the black circles are the vertices.}
\label{fig:equilateral element}
\end{figure}

We now consider an element $\Omega_{n}$ of the domain and apply the FR procedure on equations \eqref{eq:main compu eq 2D} and \eqref{eq:aux compu eq 2D}. This results in a correction of the solution, within element $\Omega_{n}$, on each of its faces. For a given face, this correction is applied on $N_{fp}=p+1$ flux points, represented in Figure~\ref{fig:equilateral element}.

\begin{subequations}
\begin{align}
\dfrac{\partial \hat{u}_{n}}{\partial t}&=\hat{\mathbf{\nabla}}\cdot\hat{\mathbf{f}}_{n}\left(\hat{\mathbf{q}}_{n}\right)+\hat{\mathbf{\nabla}}\cdot\sum_{f=1}^{3}\sum_{j=1}^{N{fp}}\left[\left(\hat{\mathbf{f}}_{n,fj}^{*}-\hat{\mathbf{f}}_{n,fj}\right)\cdot\hat{\mathbf{n}}_{fj}\right]\mathbf{h}_{fj}\left(\mathbf{r}\right),\label{eq:main compu corrected eq 2D}\\
\hat{\mathbf{q}}_{n}&=\hat{\mathbf{\nabla}}\hat{u}_{n}+\sum_{f=1}^{3}\sum_{j=1}^{N_{fp}}\left(\hat{u}^{*}_{n,fj}-\hat{u}_{n,fj}\right)\psi_{fj}\left(\mathbf{r}\right)\hat{\mathbf{n}}_{fj},\label{eq:aux compu corrected eq 2D}
\end{align}
\end{subequations}
where $f$ denotes the faces of the element, $j$ is the index over the flux points, and lastly $u^{*}$ and $\mathbf{f}^{*}$ are the numerical fluxes. $\mathbf{h}_{fj}$ corresponds to the correction function vector of the primary equation associated to the Flux Point (FP) $\left(f,j\right)$. In addition $\psi_{fj}$ is the correction field associated to the correction function $\mathbf{g}_{fj}$ of the FP $\left(f,j\right)$. We have, for the primary equation, $\hat{\mathbf{\nabla}}\cdot \mathbf{h}_{fj}=\phi_{fj}$ and $\hat{\mathbf{\nabla}}\cdot \mathbf{g}_{fj}=\psi_{fj}$ for the auxiliary equation. The correction vectors are introduced to create continuous quantities across the edges and hence satisfy
\begin{equation}\label{eq:h dot n}
\mathbf{h}_{fj}\left(\mathbf{r}_{k}^{l}\right)\cdot \mathbf{\hat{n}}_{kl}=\delta_{fk}\delta_{jl},
\end{equation}
where $\delta$ is the Kronecker delta and $\mathbf{r}_{k}^{l}$ is the coordinates of the FP $\left(k,l\right)$. Hence, on the boundary of the reference element, the outward component of the correction function $\mathbf{h}_{fj}$ is equal to 1 on the FP $\left(f,j\right)$ and 0 on the others.

Moreover the correction functions and their divergence must satisfy additional properties to ensure stability of the advection problem. Full details of these properties can be found in the article of Castonguay et al.~\cite{castonguay_new_2012}. For brevity, we present just a few of them. We consider that the vector correction function $\mathbf{h}_{fj}$ associated to the FP $\left(f,j\right)$ lies in the Raviart-Thomas (RT) space of order $p$, $RT_{p}\left(\Omega_{s}\right)=\left(P_{p}\left(\Omega_{s}\right)\right)^{2}+\begin{pmatrix}
r\\
s
\end{pmatrix}P_{p}\left(\Omega_{s}\right)$.
As a result, two properties of $\mathbf{h}_{fj}$ have been demonstrated~\cite{castonguay_new_2012}, 
\begin{eqnarray}
\mathbf{h}_{fj}\cdot \mathbf{\hat{n}}_{\Gamma_{s}}&\in& R_{p}\left(\Gamma_{s}\right)\,\forall f\in\llbracket 1,3\rrbracket \, ,\forall j\in\llbracket 1,N_{fp}\rrbracket, \label{eq:prop1 h along edge polynomial p}\\
\hat{\mathbf{\nabla}}\cdot \mathbf{h}_{fj}&\in& P_{p}\left(\Omega_{s}\right)\,\forall f\in\llbracket 1,3\rrbracket \, ,\forall j\in\llbracket 1,N_{fp}\rrbracket,\label{eq:prop2 div h polynomial p}
\end{eqnarray}
where $\Gamma_{s}$ is the boundary of the reference element $\Omega_{s}$, $\mathbf{\hat{n}}_{\Gamma_{s}}$ is the normal of one of the faces of $\Gamma_{s}$ and $R_{p}\left(\Gamma_{s}\right)$ is defined as
\begin{equation}
R_{p}\left(\Gamma_{s}\right)=\left\lbrace\phi\mid\phi\in L^{2}\left(\Gamma_{s}\right),\,\left.\phi\right|_{\Gamma_{s,f}}\in P_{p}\left(\Gamma_{s,f}\right)\forall\, \Gamma_{s,f} \right\rbrace,
\end{equation}
where $\Gamma_{s,f}$ is one of the edges of $\Gamma_{s}$.

Equation~\eqref{eq:prop2 div h polynomial p} signifies that the divergence of the correction function associated to the FP $\left(f,j\right)$ is a two-dimensional polynomial of degree $p$ on the element while equation~\eqref{eq:prop1 h along edge polynomial p} expresses that the outward component of the correction function associated to FP $\left(f,j\right)$ along the edge of any face of the element is a 1D polynomial of degree $p$.

The following additional property is required to obtain a stable scheme: Castonguay et al.~\cite{castonguay_new_2012}, defined a class of correction functions for triangular elements, called ESFR, which is stable for the linear advection problem. This property was proposed such that additional terms, arising from the ESFR norm (defined later in equation~\eqref{eq:ESFR norm for U}), are removed. In the proof~\cite{castonguay_new_2012}, using this new norm, lower bounds are found for all the remaining terms and stability is obtained. The correction functions must satisfy, 
\begin{equation}\label{eq:ESFR property}
c\mathlarger{\sum}_{m=1}^{p+1}\binom{p}{m-1}\left(D^{\left(m,p\right)}L_{i}\right)\left(D^{\left(m,p\right)}\phi_{fj}\right)=\stretchint{5ex}_{\bs\bs\Omega_{s}}\mathbf{h}_{fj}\cdot\hat{\nabla}L_{i}\,\mbox{d}\Omega_{s},
\end{equation}
where $c$ is the correction function parameter and $D^{\left(m,p\right)}=\dfrac{\partial^{p}}{\partial r^{p-m+1}\partial s^{m-1}}$ is the derivative operator. We then define
\begin{equation}\label{eq: correction fonction decomposition}
\hat{\mathbf{\nabla}}\cdot \mathbf{h}_{fj}=\phi_{fj}\left(\mathbf{r}\right)=\sum_{i=1}^{N_{p}}\sigma_{fj,i}^{c}L_{i}\left(\mathbf{r}\right),
\end{equation}
where $\left(\sigma_{fj,i}^{c}\right)_{i\in\llbracket 1,N_{p}\rrbracket}$ are the coefficients of the correction field $\phi_{fj}$. If these coefficients respect the following equation, then $\phi_{fj}$ is an ESFR correction field and the linear advection problem is stable for triangles using the Lax-Friedrichs numerical flux.
Replacing equation~\eqref{eq: correction fonction decomposition} in equation~\eqref{eq:ESFR property} yields,
\begin{equation}\label{eq:ESFR formula}
c\sum_{k=1}^{N_{p}}\sigma_{fj,k}^{c}\sum_{m=1}^{p+1}\binom{p}{m-1}\left(D^{\left(m,p\right)}L_{i}\right)\left(D^{\left(m,p\right)}L_{k}\right)=-\sigma_{fj,i}^{c}+\stretchint{5ex}_{\bs\bs \Gamma_{s}}\left(\mathbf{h}_{fj}\cdot\hat{\mathbf{n}}\right)L_{i}\,\mbox{d}\Gamma_{s}\,\,\,\forall i\in\llbracket 1,N_{p}\rrbracket.
\end{equation}
Besides ensuring the stability of the advection problem, the above equation enables the correction fields to have mirror and rotational symmetry as explained in \cite{castonguay_new_2012}. In the following sections, both the primary ($\mathbf{h}_{fj}$) and the auxiliary ($\mathbf{g}_{fj}$) correction functions are taken such that they satisfy the above properties. The auxiliary correction field is parametrized by $\kappa$ and is defined through the coefficients $\left(\sigma_{fj,i}^{\kappa}\right)_{i\in\llbracket 1,N_{p}\rrbracket}$,

\begin{equation}\label{eq:ESFR formula kappa}
\kappa\sum_{k=1}^{N_{p}}\sigma_{fj,k}^{\kappa}\sum_{m=1}^{p+1}\binom{p}{m-1}\left(D^{\left(m,p\right)}L_{i}\right)\left(D^{\left(m,p\right)}L_{k}\right)=-\sigma_{fj,i}^{\kappa}+\stretchint{5ex}_{\bs\bs \Gamma_{s}}\left(\mathbf{g}_{fj}\cdot\hat{\mathbf{n}}\right)L_{i}\,\mbox{d}\Gamma_{s}\,\,\,\forall i\in\llbracket 1,N_{p}\rrbracket.
\end{equation}

Let us now enumerate the different choices of the ESFR method for the diffusion problem:
\begin{enumerate}
\item The correction field $\psi$ for the auxiliary equation is parametrized by the parameter $\kappa$ (Equation~\eqref{eq:ESFR formula kappa}). In the following section, we will show that when employing the IP or BR2 numerical fluxes, the method does not depend on $\kappa$. To verify this result, we will run numerical simulations for two values, $\kappa_{DG}$ and $\kappa_{+}$. These values will be further elaborated in the following section.
\item The correction field $\phi$ for the primary equation. This correction field is parametrized by $c$. Castonguay et al.~\cite{castonguay_new_2012} studied the influence of $c$ for the advection problem. He defined a particular value of $c$ noted as $c_{+}$, which yields the maximal stable time step.

\begin{table}[H]
\centering
\begin{tabular}{|c||c|c|}
\hline
\bf \backslashbox{$p$}{$\gamma$}&$60\degree$&$90\degree$\\
\hline
2&4.3e-02&4.3e-02\\
\hline
3&6.4e-04&6.0e-04\\
\hline
\end{tabular}
\caption{Numerical values for $c_{+}$ for different regular meshes and orders of polynomial $p$ obtained through a von Neumann analysis.}
\label{tab:c+ triangle}
\end{table}

The parameter $c_{+}$ given in Table~\ref{tab:c+ triangle} was obtained through a von Neumann analysis for $p=2$ and $p=3$ for an advection problem. $c_{+}$ also depends on the shape of the elements, and hence on $\gamma$, as explained in Section~\ref{sec:VN section}.

\item For each element, the solution $u$ is computed at $N_{p}$ solution points and $N_{fp}$ flux points. The stability proof requires the use of the Gauss-Legendre flux points in order to evaluate exactly the integrals over the edges. However there is no constraint on the solution points. For our numerical simulations, we chose, arbitrarily, the $\alpha$-optimized solution points proposed in the book of Hesthaven and Warburton \cite{NDG}.
\item Two numerical fluxes will be studied: the IP and BR2 schemes

\begin{equation}\label{eq:numerical fluxes}
\left\lbrace
\begin{aligned}
&\text{IP:}\\
&u^{*}=\left\lbrace\left\lbrace u\right\rbrace\right\rbrace  \\
&\mathbf{q}^{*}=\left\lbrace\left\lbrace\mathbf{\nabla} u\right\rbrace\right\rbrace -\tau \mathbf{\llbracket u\rrbracket}
\end{aligned}
\right.
\hspace{2cm}
\left\lbrace
\begin{aligned}
&\text{BR2:}\\
&u^{*}=\left\lbrace\left\lbrace u\right\rbrace\right\rbrace  \\
&\mathbf{q}^{*}=\left\lbrace\left\lbrace\mathbf{\nabla} u\right\rbrace\right\rbrace +s  \left\lbrace\left\lbrace \mathbf{r^{e}}\left(\mathbf{\llbracket u\rrbracket}\right)\right\rbrace\right\rbrace
\end{aligned}
\right.
\end{equation}
where $\tau$ is the penalty term for the IP scheme and $s$ is the penalty term for the BR2 scheme; while $\mathbf{r^{e}}$ is a lifting operator and is defined as
\begin{equation}\label{eq:re int def}
\displaystyle{
\int_{\Omega}\mathbf{r^{e}}\left(\mathbf{\llbracket u\rrbracket}\right)\cdot \bm{\Phi}\,\mbox{d}\Omega=-\int_{\Gamma_{e}}\mathbf{\llbracket u\rrbracket} \cdot \left\lbrace\left\lbrace \bm{\Phi}\right\rbrace\right\rbrace\,\mbox{d}\Gamma_{e},
}
\end{equation}
where $\bm{\Phi}$ is a vector test function and $\Gamma_{e}$ denotes the edge $e$. The notations $\mathbf{\llbracket\,\,\rrbracket}$ denotes the jump and $\{\{\,\,\}\}$ the mean value. While the former is the difference of the solution across an edge $e$, the latter is the average,
\begin{eqnarray}
\mathbf{\llbracket u\rrbracket}_{e} &=& u_{e,-}\mathbf{n}^{e}_{-}+u_{e,+}\mathbf{n}^{e}_{+},\\
\{\{ u \}\}_{e}&=&\dfrac{1}{2}\left(u_{e,+}+u_{e,-}\right),
\end{eqnarray} 
where $u_{e,-}$ and $\mathbf{n}_{-}^{e}$ denotes the interior solution and the outward interior normal vector of edge $e$, while $u_{e,+}$ and $\mathbf{n}_{+}^{e}$ signify the exterior solution and the outward exterior edge normal. Figure~\ref{fig:Part domain} provides the geometric interpretation of these quantities. Notice, while $\mathbf{\llbracket u \rrbracket}$ is a vector, $\llbracket q \rrbracket$ is a scalar value. Conversely, $\{\{u\}\}$ is a scalar value but $\mathbf{\{\{ \mathbf{q}\}\}}$ is a vector.
\end{enumerate}
\section{The diffusion equation, independent of $\kappa$}\label{sec:independent kappa}
The purpose of this section is to prove that using the ESFR schemes with the IP and BR2 numerical fluxes results in an independency of the problem from the parameter $\kappa$, associated with the correction function for the auxiliary equation.

From equation~\eqref{eq:aux compu corrected eq 2D}, the transformed solution correction at each flux point for both the IP and BR2 schemes as shown in equation~\eqref{eq:numerical fluxes}  yields,
\begin{equation}\label{eq:jump indep}
\begin{array}{lll}
\left(\hat{u}_{n,fj}^{*}-\hat{u}_{n,fj}\right)&=&\left\lbrace\left\lbrace \hat{u} \right\rbrace\right\rbrace_{fj}-\hat{u}_{fj}\\
&=&-\dfrac{\llbracket \hat{u} \rrbracket_{fj}}{2}.
\end{array}
\end{equation}
Note that $\llbracket \hat{u} \rrbracket_{fj}$ is a scalar value of the jump on face $f$ at the flux point $j$. 
However the IP and BR2 schemes differ only for the numerical flux function, $\hat{\mathbf{f}}^{*}$ for the primary equation as clearly stated in equation~\eqref{eq:numerical fluxes}. The schemes only depend on the mean of the solution gradient and the solution jump but do not depend on the auxiliary solution, $\hat{\mathbf{q}}$ and hence do not depend on the parameter $\kappa$.

By including the numerical flux into the transformed solution correction as shown in equation~\eqref{eq:jump indep} into equations~\eqref{eq:main compu corrected eq 2D} and~\eqref{eq:aux compu corrected eq 2D}, we can represent the ESFR scheme as,
 
\begin{subequations}
\begin{align}
\dfrac{\partial \hat{u}_{n}}{\partial t}&=\hat{\mathbf{\nabla}}\cdot\hat{\mathbf{f}}_{n}\left(\hat{\mathbf{q}}_{n}\right)+\hat{\mathbf{\nabla}}\cdot\mathlarger{\sum}_{f=1}^{3}\mathlarger{\sum}_{j=1}^{N_{fp}}\left[\left(\hat{\mathbf{f}}_{n,fj}^{*}-\hat{\mathbf{f}}_{n,fj}\right)\cdot\hat{\mathbf{n}}_{fj}\right]\mathbf{h}_{fj}\left(\mathbf{r}\right),\label{eq:main compu corrected eq 2D bis}\\
\hat{\mathbf{q}}_{n}&=\hat{\mathbf{\nabla}}\hat{u}_{n}-\mathlarger{\sum}_{e=1}^{3}\mathlarger{\sum}_{i=1}^{N_{fp}}\dfrac{\llbracket \hat{u} \rrbracket_{ei}}{2}\psi_{ei}\left(\mathbf{r}\right)\hat{\mathbf{n}}_{ei}.\label{eq:aux compu corrected eq 2D bis}
\end{align}
\end{subequations}
Note that we have renamed the indices $f$ and $j$ in equation~\eqref{eq:aux compu corrected eq 2D bis} by $e$ and $i$.

We now introduce the auxiliary equation~\eqref{eq:aux compu corrected eq 2D bis} into the primary~\eqref{eq:main compu corrected eq 2D bis},
\begin{equation}
\begin{array}{lll}
\dfrac{\partial \hat{u}_{n}}{\partial t}&=&\hat{\mathbf{\nabla}}\cdot\left(\mathbf{J}_{n}^{-1}\mathbf{J}_{n}^{-T}\hat{\nabla}\hat{u}_{n}\right)+\mathlarger{\sum}_{f=1}^{3}\mathlarger{\sum}_{j=1}^{N_{fp}}\left[\left(\hat{\mathbf{f}}_{n,fj}^{*}-\mathbf{J}_{n}^{-1}\mathbf{J}_{n}^{-T}\hat{\nabla}\hat{u}_{n,fj}\right)\cdot\hat{\mathbf{n}}_{fj}\right]\mathbf{\phi}_{fj}\left(\mathbf{r}\right)\\
&&-\hat{\mathbf{\nabla}}\cdot\left(\mathbf{J}_{n}^{-1}\mathbf{J}_{n}^{-T}\mathlarger{\sum}_{e=1}^{3}\mathlarger{\sum}_{i=1}^{N_{fp}}\dfrac{\llbracket \hat{u} \rrbracket_{ei}}{2}\psi_{ei}\left(\mathbf{r}\right)\hat{\mathbf{n}}_{ei}\right)\\
&&+\mathlarger{\sum}_{f=1}^{3}\mathlarger{\sum}_{j=1}^{N_{fp}}\left[\mathbf{J}_{n}^{-1}\mathbf{J}_{n}^{-T}\left(\mathlarger{\sum}_{e=1}^{3}\mathlarger{\sum}_{i=1}^{N_{fp}}\dfrac{\llbracket \hat{u} \rrbracket_{ei}}{2}\psi_{ei}\left(\mathbf{r}_{f}^{j}\right)\hat{\mathbf{n}}_{ei}\right)\cdot\mathbf{\hat{n}}_{fj}\right]\phi_{fj}\left(\mathbf{r}\right),
\end{array}
\end{equation}
where $\psi_{ei}\left(\mathbf{r}_{f}^{j}\right)$ is the correction function of the edge $e$ computed at the FP $\left(f,j\right)$. On further simplification, the primary equation can be represented as, 

\begin{equation}\label{eq:ini demo}
\begin{array}{lll}
\dfrac{\partial \hat{u}_{n}}{\partial t}&=&\hat{\mathbf{\nabla}}\cdot\left(\mathbf{J}_{n}^{-1}\mathbf{J}_{n}^{-T}\hat{\nabla}\hat{u}_{n}\right)+\mathlarger{\sum}_{f=1}^{3}\mathlarger{\sum}_{j=1}^{N_{fp}}\left[\left(\hat{\mathbf{f}}_{n,fj}^{*}-\mathbf{J}_{n}^{-1}\mathbf{J}_{n}^{-T}\hat{\nabla}\hat{u}_{n,fj}\right)\cdot\hat{\mathbf{n}}_{fj}\right]\mathbf{\phi}_{fj}\left(\mathbf{r}\right)\\
&&+\mathbf{J}_{n}^{-1}\mathbf{J}_{n}^{-T}\mathlarger{\sum}_{e=1}^{3}\mathlarger{\sum}_{i=1}^{N_{fp}}\left[\dfrac{\llbracket \hat{u} \rrbracket_{ei}}{2}\left(-\hat{\nabla}\psi_{ei}\left(\mathbf{r}\right)\cdot\hat{\mathbf{n}}_{ei}+\mathlarger{\sum}_{f=1}^{3}\mathlarger{\sum}_{j=1}^{N_{fp}}\psi_{ei}\left(\mathbf{r}_{f}^{j}\right)\left(\hat{\mathbf{n}}_{ei}\cdot\mathbf{\hat{n}}_{fj}\right)\phi_{fj}\left(\mathbf{r}\right)\right)\right].
\end{array}
\end{equation}

\begin{post}\label{pos:independency of kappa}
Let the solution of the diffusion equation be approximated by a polynomial of degree $p$ on the reference triangle (see Figure~\ref{fig:equilateral element} for an example where $p = 2$). Let $\psi_{ei}$ be the correction field associated to face $e$ at the flux point $i$ parametrized by $\kappa$ and $\mathbf{\phi}_{fj}$ the correction field, associated to face $f$ at the flux point $j$, parametrized by $c$. Let $\left(\mathbf{r}_{f}^{j}\right)_{j\in\llbracket1,N_{fp}\rrbracket}$ be the Gauss-Legendre flux points on face $f$. Then $\forall c\in\left[0,\infty\right[,$
\begin{equation}
R_{ei}\left(\mathbf{r}\right)=\left(-\hat{\nabla}\psi_{ei}\left(\mathbf{r}\right)\cdot\hat{\mathbf{n}}_{ei}+\mathlarger{\sum}_{f=1}^{3}\mathlarger{\sum}_{j=1}^{N_{fp}}\psi_{ei}\left(\mathbf{r}_{f}^{j}\right)\left(\hat{\mathbf{n}}_{ei}\cdot\mathbf{\hat{n}}_{fj}\right)\phi_{fj}\left(\mathbf{r}\right)\right),
\end{equation}
is independent of the parameter $\kappa$.
\end{post}

This two-dimensional postulate was proposed based on the one-dimensional theoretical work~\cite{sam_article1D}. An analytical proof is presented in Appendix~\ref{sec:appendix independency kappa} for $p=1$; however, a general proof, for higher $p$, currently eludes the authors since specific properties of the orthonormal basis $\left(L_{i}\right)_{i\in\llbracket 1,N_{p}\rrbracket}$ are required. This Postulate~\ref{pos:independency of kappa} has been, nonetheless, verified numerically for every face $e$ and flux point $i$, for several values of $c$ and for an order up to $p=6$. For brevity, we will present only the result for the function,
\begin{equation}
R_{11}=\left(-\hat{\nabla}\psi_{11}\left(\mathbf{r}\right)\cdot\hat{\mathbf{n}}_{11}+\mathlarger{\sum}_{f=1}^{3}\mathlarger{\sum}_{j=1}^{N_{fp}}\psi_{11}\left(\mathbf{r}_{f}^{j}\right)\left(\hat{\mathbf{n}}_{11}\cdot\mathbf{\hat{n}}_{fj}\right)\phi_{fj}\left(\mathbf{r}\right)\right).
\end{equation}

Two values of $p$ will be taken, $p=2$ and $p=3$. The correction function $\mathbf{h}_{fj}$ will be computed for two values of $c$: $c_{DG}$ and $c_{+}$. In each of the following graphs, $R_{11}\left(\mathbf{r}\right)$ will be evaluated for all $N_{p}$ solution points.

\begin{figure}[H] 
\centering

\begin{subfigure}{0.49\textwidth}
\centering
\includegraphics[width=\linewidth]{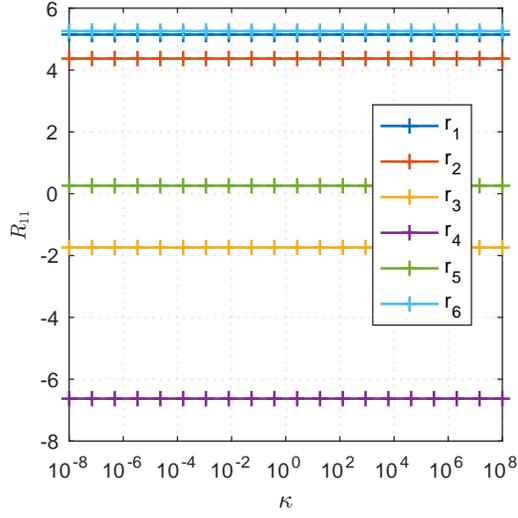}
\caption{$R_{11}$ for $c=c_{DG}$ and $p=2$.} \label{fig:dtmax DG p2}
\end{subfigure}\hspace*{\fill}
\begin{subfigure}{0.49\textwidth}
\centering
\includegraphics[width=\linewidth]{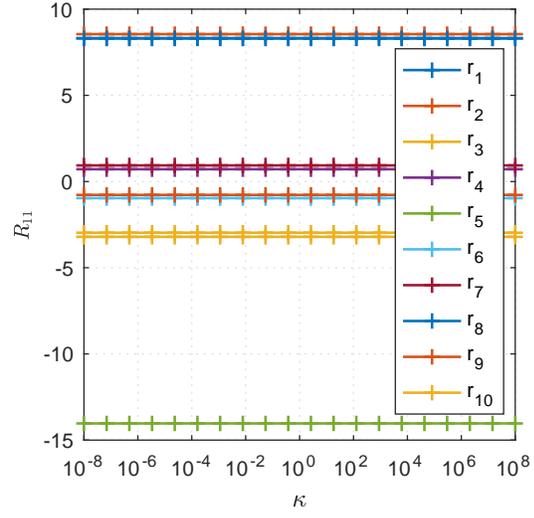}
\caption{$R_{11}$ for $c=c_{DG}$ $p=3$.} \label{fig:dtmax DG p3}
\end{subfigure}

\medskip
\begin{subfigure}{0.49\textwidth}
\centering
\includegraphics[width=\linewidth]{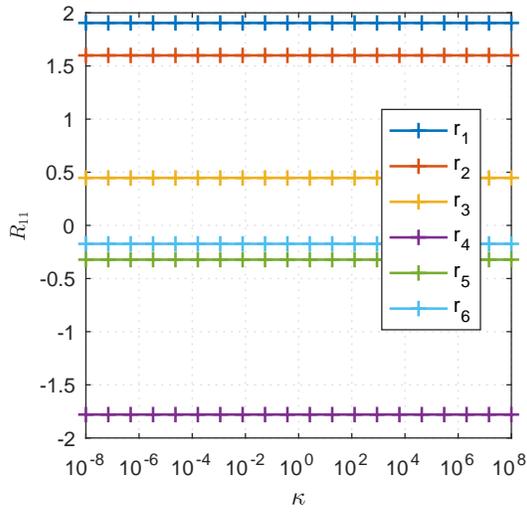}
\caption{$R_{11}$ for $c=c_{+}$ $p=2$.} \label{fig:dtmax c+ p2}
\end{subfigure}\hspace*{\fill}
\begin{subfigure}{0.49\textwidth}
\centering
\includegraphics[width=\linewidth]{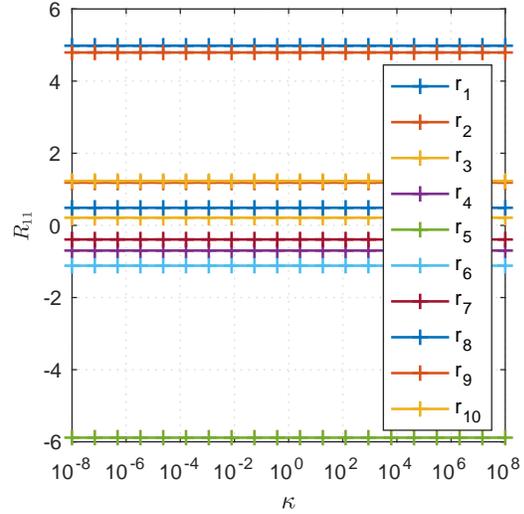}
\caption{$R_{11}$ for $c=c_{+}$ $p=3$.} \label{fig:dtmax c+ p3}
\end{subfigure}
\caption{Influence of parameter $\kappa$ on function $R_{11}$ for various values of parameter $c$ and polynomial order $p$.}
\label{fig:plot dt max}
\end{figure}

\begin{post}\label{pos:independent}
The diffusion equation is independent of $\kappa$ when employing the IP or BR2 numerical fluxes.
\end{post}
\begin{proof}
The diffusion equation with the ESFR schemes and employing the IP or BR2 numerical fluxes is written as equation~\eqref{eq:ini demo}. Only the last line of this discretization contains the parameter $\kappa$. However from Postulate~\ref{pos:independency of kappa}, it does not depend on $\kappa$.
\end{proof}
\section{IP stability condition}\label{sec:IP section}

\subsection{Theoretical result}\label{sec:theoretical proof IP triangle}

Having introduced the 2D FR concept, we can now derive the stability condition for the IP scheme using the ESFR class functions. This proof has been derived by Williams et al.~\cite{williams_energy_2013} for the LDG scheme using linear triangles. For brevity, we will not re-derive all the steps but instead begin from Equation (81) of the aforementioned article. Upon removing the advective term, we obtain:
\begin{equation}\label{eq:equation to get ESFR stability}
\dfrac{1}{2}\dfrac{\mbox{d}}{\mbox{d}t}\Vert U\Vert_{p,c}^{2}=-b\Vert \mathbf{Q}\Vert_{p,\kappa}^{2}+b\Theta_{dif},
\end{equation}
where,
\begin{eqnarray}
&\Vert U\Vert_{p,c}&=\left\lbrace\sum_{n=1}^{N_{K}}\stretchint{5ex}_{\bs\bs \Omega_{n}}\left[\left(u_{n}\right)^{2}+\dfrac{1}{A_{s}}\sum_{m=1}^{p+1}c_{m}\left(D^{\left(m,p\right)}u_{n}\right)^{2}\right]\mbox{d}\Omega_{n}\right\rbrace^{1/2},\label{eq:ESFR norm for U}\\
&\Vert \mathbf{Q}\Vert_{p,\kappa}&=\left\lbrace\sum_{n=1}^{N_{K}}\stretchint{5ex}_{\bs\bs \Omega_{n}}\left[\left(\mathbf{q}_{n}\right)^{2}+\dfrac{1}{A_{s}}\sum_{m=1}^{p+1}\kappa_{m}\left(D^{\left(m,p\right)}\mathbf{q}_{n}\right)^{2}\right]\mbox{d}\Omega_{n}\right\rbrace^{1/2},
\end{eqnarray}
are broken Sobolev-type norms for the solution and the auxiliary variables, and
\begin{equation}
\Theta_{dif}=\sum_{n=1}^{N_{K}}\left\lbrace\stretchint{5ex}_{\bs\bs\Gamma_{n}}\left[-u_{n}\left(\mathbf{q}_{n}\cdot \mathbf{n}\right)+u_{n}\left(\mathbf{q}_{n}^{*}\cdot\mathbf{n}\right)+u_{n}^{*}\left(\mathbf{q}_{n}\cdot\mathbf{n}\right)\right]\mbox{d}\Gamma_{n}\right\rbrace,\label{eq:theta dif ini}
\end{equation}
represents contributions of the diffusive fluxes from the boundaries. The coefficients $c_{m}=c\binom{p}{m-1}$, $\kappa_{m}=\kappa\binom{p}{m-1}$, $A_{s}$ denotes the area of the reference equilateral triangle, and $\Gamma_{n}$ defines all the edges of an element $\Omega_{n}$. In the previous equation, the positive scalar value $b$ multiplies every term of $\Theta_{dif}$ and is hence factored out. This is a slight difference from~\cite{williams_energy_2013}, where $b$ does not multiply the penalty term. This choice does not impact our result as it is just a nondimensionalization of the penalty term. To ensure energy stability, we must have the right hand side of equation~\eqref{eq:equation to get ESFR stability} to be less than equal to zero, where $\left(-b\Vert \mathbf{Q}\Vert_{p,\kappa}^{2}\right)$ is ensured to be negative granted that the coefficient $b$ is positive. Hence we must demonstrate that $\Theta_{dif}$ is non-positive for appropriate choices of numerical fluxes and interface solutions.

\begin{figure}[H]
\centering
\includegraphics[width=12cm]{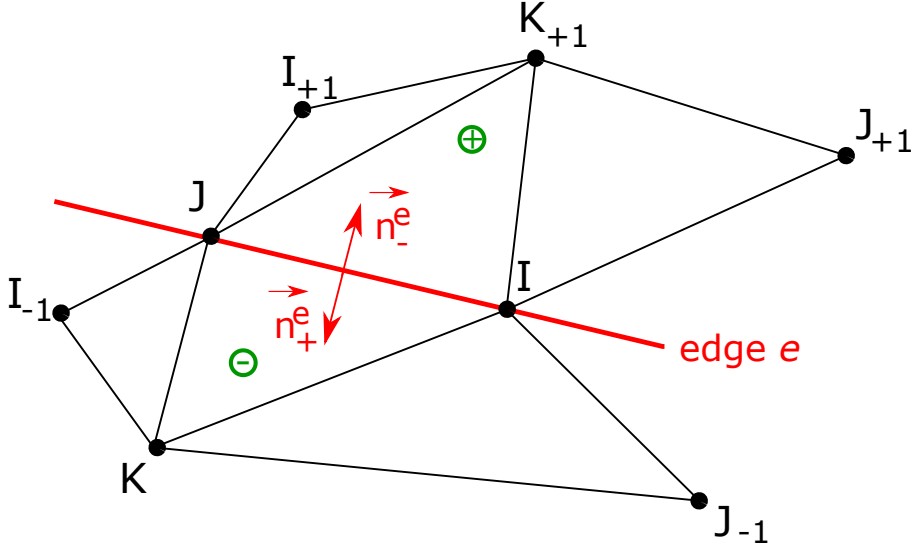}
\caption{Part of an arbitrary physical domain.}
\label{fig:Part domain}
\end{figure}

For the rest of our proof, we consider the edge $e$ represented in Figure~\ref{fig:Part domain}. The triangle $IJK$ is the interior element of edge $e$ and triangle $IJK_{+1}$, the exterior element of edge $e$.

\begin{lem}\label{lem:theta dif}
Employing the IP scheme and considering periodic boundary conditions, the following equation holds
\normalfont
\begin{equation}\label{eq:theta dif lem}
\displaystyle{
\Theta_{dif}=\sum_{e=1}^{N_{e}}\int_{\Gamma_{e}}\left(-\{\{\mathbf{q}-\mathbf{\nabla} u\}\}\right)\cdot \mathbf{\llbracket u \rrbracket}_{e} -\tau\mathbf{\llbracket u\rrbracket}_{e}\cdot\mathbf{\llbracket u\rrbracket}_{e}\,\mbox{d}\Gamma_{e}.
}
\end{equation}
\end{lem}

\begin{proof}
Instead of considering an element based formulation, we modify $\Theta_{dif}$ in~\eqref{eq:theta dif ini} to an edge based formulation,
\begin{equation}\label{eq:theta dif expand}
\Theta_{dif}=\displaystyle{
\sum_{e=1}^{N_{e}}\int_{\Gamma_{e}}\left[-u_{n}\left(\mathbf{q}_{n}\cdot\mathbf{n}\right)+u_{n+1}\left(\mathbf{q}_{n+1}\cdot\mathbf{n}\right)+\mathbf{q}_{e}^{*}\cdot\mathbf{n}\left(u_{n}-u_{n+1}\right)+u_{e}^{*}\left(\mathbf{q}_{n}-\mathbf{q}_{n+1}\right)\cdot\mathbf{n}\right]\mbox{d}\Gamma_{e},
}
\end{equation}
where $\mathbf{n}=\mathbf{n}^{e}_{-}$ as shown in Figure~\ref{fig:Part domain}.

We then use the IP numerical fluxes for both $u^*$ and $\mathbf{q}^*$ as defined in~\eqref{eq:numerical fluxes} to expand~\eqref{eq:theta dif expand}
\begin{equation}
\begin{array}{ccll}
\Theta_{dif}&=&
\mathlarger{\sum}_{e=1}^{N_{e}}\stretchint{5ex}_{\bs \bs\Gamma_{e}}&\left(-u_{n} \mathbf{q}_{n}+u_{n+1}\mathbf{q}_{n+1}\right)\cdot\mathbf{n}+\left(\{\{\mathbf{\nabla} u \}\} -\tau\mathbf{\llbracket u \rrbracket}_{e}\right)\cdot\mathbf{\llbracket u\rrbracket}_{e}+\{\{u\}\}\llbracket q\rrbracket \mbox{d}\Gamma_{e},\\
&=&\mathlarger{\sum}_{e=1}^{N_{e}}\stretchint{5ex}_{\bs \bs \Gamma_{e}}&\left[-u_{n}\mathbf{q}_{n}\cdot\mathbf{n}\left(1-\frac{1}{2}\right)\right.\\
&&&+u_{n+1}\mathbf{q}_{n+1}\cdot\mathbf{n}\left(1-\frac{1}{2}\right)\vspace{0.2cm}\\
&&&+\frac{1}{2}\left(-u_{n}\mathbf{q}_{n+1}+u_{n+1}\mathbf{q}_{n}\right)\cdot\mathbf{n}\vspace{0.2cm}\\
&&&\left.+\left(\{\{\nabla u \}\} -\tau\mathbf{\llbracket u \rrbracket}_{e}\right)\cdot\mathbf{\llbracket u\rrbracket}_{e} \right]\mbox{d}\Gamma_{e},\\
&=&\mathlarger{\sum}_{e=1}^{N_{e}}\stretchint{5ex}_{\bs \bs \Gamma_{e}}&\dfrac{1}{2}\left[\left(u_{n+1}-u_{n}\right)\left(\mathbf{q}_{n}+\mathbf{q}_{n+1}\right)\cdot\mathbf{n}\right]+\left(\{\{\mathbf{\nabla} u \}\} -\tau\mathbf{\llbracket u \rrbracket}_{e}\right)\cdot\mathbf{\llbracket u\rrbracket}_{e} \mbox{d}\Gamma_{e}\\
&=&\mathlarger{\sum}_{e=1}^{N_{e}}\stretchint{5ex}_{\bs \bs \Gamma_{e}}& -\{\{\mathbf{q}\}\}\cdot\mathbf{\llbracket u\rrbracket}_{e}+\left(\{\{\nabla u \}\} -\tau\mathbf{\llbracket u \rrbracket}_{e}\right)\cdot\mathbf{\llbracket u\rrbracket}_{e} \mbox{d}\Gamma_{e}.\\
\end{array}
\end{equation}
We finally obtain the desired equation
\begin{equation}
\displaystyle{
\Theta_{dif}=\sum_{e=1}^{N_{e}}\int_{\Gamma_{e}}\left(-\{\{\mathbf{q}-\mathbf{\nabla} u\}\}\right)\cdot \mathbf{\llbracket u \rrbracket}_{e} -\tau\mathbf{\llbracket u\rrbracket}_{e}\cdot\mathbf{\llbracket u\rrbracket}_{e}\,\mbox{d}\Gamma_{e}.
}
\end{equation}
\end{proof}
\begin{rmk}
Notice that if we had taken the LDG scheme then $\mathbf{\nabla}u$ would have been replaced by $\mathbf{q}$ and we would have had the stability of the scheme for $\tau\geq 0$ as shown by~\cite{williams_energy_2013}.
\end{rmk}

However, for the IP and BR2 formulations, additional effort is required to demonstrate conditions on stability. We decompose the result of Lemma~\ref{lem:theta dif} into two components,
\begin{equation}\label{eq:decomposition Theta}
\displaystyle{
\Theta_{dif}=\sum_{e=1}^{N_{e}}\left[\underbrace{\int_{\Gamma_{e}}\left(-\{\{\mathbf{q}-\mathbf{\nabla} u\}\}\right)\cdot \mathbf{\llbracket u \rrbracket}_{e}\,\mbox{d}\Gamma_{e}}_{\Theta_{e,1}} -\underbrace{\int_{\Gamma_{e}}\tau\llbracket u\rrbracket^{2}_{e}}_{\Theta_{e,2}}\,\mbox{d}\Gamma_{e}\right],
}
\end{equation}
where $\llbracket u \rrbracket_{e}=\mathbf{\llbracket u \rrbracket}_{e}\cdot \mathbf{n}$. We now investigate separately $\Theta_{e,1}$ and $\Theta_{e,2}$.

\begin{lem}\label{lem:theta e1 expansion}
Using the auxiliary equation, defined in~\eqref{eq:aux compu corrected eq 2D} and the Gauss-Legendre quadratures, $\Theta_{e,1}$ can be computed exactly,
\begin{equation}\label{eq:theta 1 similar}
\begin{array}{lll}
\Theta_{e,1}&=&\left|J^{e}\right|\dfrac{F_{s,-}^{e}}{4}\mathlarger{\sum}_{i=1}^{N_{fp}}\left[\llbracket u\rrbracket_{ei}^{2}\omega_{i}\psi_{ei,-}\left(\mathbf{r}_{i}\right)+\llbracket u\rrbracket_{ei}\omega_{i}\left(\mathlarger{\sum}_{\substack{j=1\\ j\neq i}}^{N_{fp}}\llbracket u \rrbracket_{ej}\psi_{ej,-}\left(\mathbf{r}_{i}\right)\right)\right]\vspace*{0.2cm}\\
&&+\left|J^{e}\right|\dfrac{F_{s,+}^{e}}{4}\mathlarger{\sum}_{i=1}^{N_{fp}}\left[\llbracket u\rrbracket_{ei}^{2}\omega_{i}\psi_{ei,+}\left(\mathbf{r}_{i}\right)+\llbracket u\rrbracket_{ei}\omega_{i}\left(\mathlarger{\sum}_{\substack{j=1\\ j\neq i}}^{N_{fp}}\llbracket u \rrbracket_{ej}\psi_{ej,+}\left(\mathbf{r}_{i}\right)\right)\right]\vspace*{0.2cm}\\
&&+\left|J^{e}\right|\mathlarger{\sum}_{i=1}^{N_{fp}}\left[\llbracket u\rrbracket_{ei}\omega_{i}\left(\mathlarger{\sum}_{\substack{f=1\\f\neq e}}^{N_{fp}}\dfrac{F_{s,-}^{f}\left(\mathbf{n}\cdot\mathbf{n}_{f,-}\right)}{4}\mathlarger{\sum}_{j=1}^{N_{fp}}\llbracket u \rrbracket_{fj,-}\psi_{fj,-}\left(\mathbf{r}_{i}\right)\right)\right]\vspace*{0.2cm}\\
&&+\left|J^{e}\right|\mathlarger{\sum}_{i=1}^{N_{fp}}\left[\llbracket u\rrbracket_{ei}\omega_{i}\left(\mathlarger{\sum}_{\substack{f=1\\f\neq e}}^{N_{fp}}\dfrac{F_{s,+}^{f}\left(\mathbf{n}\cdot\mathbf{n}_{f,+}\right)}{4}\mathlarger{\sum}_{j=1}^{N_{fp}}\llbracket u \rrbracket_{fj,+}\psi_{fj,+}\left(\mathbf{r}_{i}\right)\right)\right],
\end{array}
\end{equation}
where $\mathbf{r}_{i}$ and $\omega_{i}$ are the Gauss-Legendre nodes and weights and $\left|J^{e}\right|$ is the length of edge $e$. $F_{s,-}^{f}$ (resp. $F_{s,+}^{f}$) is the ratio of the length of an edge $f$ over the determinant of the Jacobian of the interior (resp. exterior) element $\left(\dfrac{\left|J^{f}\right|}{\left|J_{e,\pm}\right|}\right)$.
\end{lem}

\begin{proof}
We use equation~\eqref{eq:aux compu corrected eq 2D} along with the transformation equalities from the physical space to the computational space, equations~\eqref{eq:nabla2D} and~\eqref{eq:q2D}. Instead of using element indices $n$, $n+1$, we use the subscript $+$ and $-$ as shown in Figure~\ref{fig:Part domain}. In the term $\Theta_{e,1}$ in equation~\eqref{eq:decomposition Theta}, we replace $\mathbf{\llbracket u \rrbracket}_{e}$ by $\llbracket u \rrbracket_{e}\mathbf{n}$,
\begin{equation}
\begin{array}{lll}
\Theta_{e,1}&=&\stretchint{5ex}_{\bs\bs \Gamma_{e}}-\dfrac{\llbracket u \rrbracket_{e}}{2}\left[\dfrac{1}{\left|J_{e,-}\right|}\mathbf{J}_{e,-}^{-T}\left(\hat{\mathbf{q}}_{e,-}-\hat{\mathbf{\nabla}}\hat{u}_{e,-}\right)+\dfrac{1}{\left|J_{e,+}\right|}\mathbf{J}_{e,+}^{-T}\left(\hat{\mathbf{q}}_{e,+}-\hat{\mathbf{\nabla}}\hat{u}_{e,+}\right)\right]\cdot\mathbf{n}\,\mbox{d}\Gamma_{e} \vspace*{0.3cm}\\
&=&\stretchint{5ex}_{\bs\bs \Gamma_{e}}-\dfrac{\llbracket u \rrbracket_{e}}{2}\left[\dfrac{1}{\left|J_{e,-}\right|}\mathbf{J}_{e,-}^{-T}\left(\mathlarger{\sum}_{f={1}}^{3}\mathlarger{\sum}_{j=1}^{N_{fp}}\hat{u}_{fj,-}^{c}\psi_{fj,-}\mathbf{\hat{n}}_{fj,-}\right)\right]\cdot\mathbf{n}\,\mbox{d}\Gamma_{e}\vspace*{0.3cm}\\
&+&\stretchint{5ex}_{\bs\bs \Gamma_{e}}-\dfrac{\llbracket u \rrbracket_{e}}{2}\left[\dfrac{1}{\left|J_{e,+}\right|}\mathbf{J}_{e,+}^{-T}\left(\mathlarger{\sum}_{f={1}}^{3}\mathlarger{\sum}_{j=1}^{N_{fp}}\hat{u}_{fj,+}^{c}\psi_{fj,+}\hat{\mathbf{n}}_{fj,+}\right)\right]\cdot\mathbf{n}\,\mbox{d}\Gamma_{e},
\end{array}
\end{equation}

where $\hat{u}_{fj}^{c}=\left(\hat{u}^{*}-\hat{u}\right)_{fj}$, is the transformed solution correction. Using~\eqref{eq:u2D} and~\eqref{eq:n2D}, 
\begin{equation}
\begin{array}{lll}
\Theta_{e,1}&=&\stretchint{5ex}_{\bs \bs \Gamma_{e}}-\dfrac{\llbracket u\rrbracket_{e}}{2}\left[\dfrac{1}{\left|J_{e,-}\right|}\left(\mathlarger{\sum}_{f={1}}^{3}\left| J^{f,-}\right|\mathlarger{\sum}_{j=1}^{N_{fp}}u_{fj,-}^{c}\psi_{fj,-}\mathbf{n}_{fj,-}\right)\right]\cdot\mathbf{n}\,\mbox{d}\Gamma_{e}\\
&&+\stretchint{5ex}_{\bs \bs \Gamma_{e}}-\dfrac{\llbracket u\rrbracket_{e}}{2}\left[\dfrac{1}{\left|J_{e,+}\right|}\left(\mathlarger{\sum}_{f=1}^{3}\left| J^{f,+}\right|\mathlarger{\sum}_{j=1}^{N_{fp}}u_{fj,+}^{c}\psi_{fj,+}\mathbf{n}_{fj,+}\right)\right]\cdot\mathbf{n}\,\mbox{d}\Gamma_{e}.
\end{array}
\end{equation}

Through a substitution of the numerical flux into the solution correction, we obtain $u_{fj}^{c}=\{\{ u\}\}_{fj}-u_{fj}=-\frac{1}{2}\llbracket u\rrbracket_{fj}$. Note that while $\llbracket u\rrbracket_{e}$ is a function of $\left(r,s\right)$, $\llbracket u\rrbracket_{fj}$ is a constant. We also define the parameter $F_{s}$ which is the ratio of the length of an edge over the determinant of the Jacobian of an element. Since we employ a straight sided triangle, we can drop the subscript $j$ for the normal of an edge. In the following equation, we expand on the summation over the three faces and separate the jump belonging to edge $e$, $\llbracket u\rrbracket_{ej}$ from the other jumps $\llbracket u\rrbracket_{fj}$,

\begin{equation}
\begin{array}{lll}
\Theta_{e,1}&=&\stretchint{5ex}_{\bs \bs \Gamma_{e}}\dfrac{\llbracket u\rrbracket_{e}}{4}\left[F_{s,-}^{e}\mathlarger{\sum}_{j=1}^{N_{fp}}\llbracket u \rrbracket_{ej}\psi_{ej,-}\right]\mbox{d}\Gamma_{e}\vspace*{0.3cm}\\
&&+\stretchint{5ex}_{\bs \bs \Gamma_{e}}\dfrac{\llbracket u\rrbracket_{e}}{4}\left[F_{s,+}^{e}\mathlarger{\sum}_{j=1}^{N_{fp}}\llbracket u \rrbracket_{ej}\psi_{ej,+} \right]\mbox{d}\Gamma_{e}\vspace*{0.3cm}\\
&&+\stretchint{5ex}_{\bs \bs \Gamma_{e}}\dfrac{\llbracket u\rrbracket_{e}}{4}\left[\mathlarger{\sum}_{\substack{f=1\\f\neq e}}F_{s,-}^{f}\mathlarger{\sum}_{j=1}^{N_{fp}}\llbracket u \rrbracket_{fj}\psi_{fj,-} \mathbf{n}_{f}\right]\cdot \mathbf{n}\,\mbox{d}\Gamma_{e}\vspace*{0.3cm}\\
&&+\stretchint{5ex}_{\bs \bs \Gamma_{e}}\dfrac{\llbracket u\rrbracket_{e}}{4}\left[\mathlarger{\sum}_{\substack{f=1\\f\neq e}}F_{s,+}^{f}\mathlarger{\sum}_{j=1}^{N_{fp}}\llbracket u \rrbracket_{fj}\psi_{fj,+} \mathbf{n}_{f}\right]\cdot \mathbf{n}\,\mbox{d}\Gamma_{e},
\end{array}
\end{equation}
where $\Theta_{e,1}$ can now be expressed as the sum of four terms. The next step is to transform each integral into the computational space, where each term integrates the product $\llbracket u \rrbracket_{e} \psi_{fj}$ over the computational edge $\Gamma_{s,e}$. Since $u_{n}\in P_{p}\left(\Omega_{s}\right)$, we have also $u_{n}\in R_{p}\left(\Gamma_{s}\right)$. Similarly, $\mathbf{\nabla}\cdot \mathbf{g}_{fj}=\psi_{fj}\in P_{p}\left(\Omega_{s}\right)$ and hence $\psi_{fj}\in R_{p}\left(\Gamma_{s}\right)$. Therefore the product $\llbracket u \rrbracket_{e} \psi_{fj}$ is a polynomial of degree less than or equal to $2p$. Using Gauss-Legendre quadratures, we can compute exactly the integrals of $\Theta_{e,1}$. We introduce $\mathbf{r}_{i}$ and $\omega_{i}$ as the Gauss-Legendre nodes and weights,

\begin{equation}
\begin{array}{lll}
\Theta_{e,1}&=&\left|J^{e}\right|\dfrac{F_{s,-}^{e}}{4}\mathlarger{\sum}_{i=1}^{N_{fp}}\left[\llbracket u\rrbracket_{ei}^{2}\omega_{i}\psi_{ei,-}\left(\mathbf{r}_{i}\right)+\llbracket u\rrbracket_{ei}\omega_{i}\left(\mathlarger{\sum}_{\substack{j=1\\ j\neq i}}^{N_{fp}}\llbracket u \rrbracket_{ej}\psi_{ej,-}\left(\mathbf{r}_{i}\right)\right)\right]\vspace*{0.2cm}\\
&&+\left|J^{e}\right|\dfrac{F_{s,+}^{e}}{4}\mathlarger{\sum}_{i=1}^{N_{fp}}\left[\llbracket u\rrbracket_{ei}^{2}\omega_{i}\psi_{ei,+}\left(\mathbf{r}_{i}\right)+\llbracket u\rrbracket_{ei}\omega_{i}\left(\mathlarger{\sum}_{\substack{j=1\\ j\neq i}}^{N_{fp}}\llbracket u \rrbracket_{ej}\psi_{ej,+}\left(\mathbf{r}_{i}\right)\right)\right]\vspace*{0.2cm}\\
&&+\left|J^{e}\right|\mathlarger{\sum}_{i=1}^{N_{fp}}\left[\llbracket u\rrbracket_{ei}\omega_{i}\left(\mathlarger{\sum}_{\substack{f=1\\f\neq e}}^{N_{fp}}\dfrac{F_{s,-}^{f}\left(\mathbf{n}\cdot\mathbf{n}_{f,-}\right)}{4}\mathlarger{\sum}_{j=1}^{N_{fp}}\llbracket u \rrbracket_{fj,-}\psi_{fj,-}\left(\mathbf{r}_{i}\right)\right)\right]\vspace*{0.2cm}\\
&&+\left|J^{e}\right|\mathlarger{\sum}_{i=1}^{N_{fp}}\left[\llbracket u\rrbracket_{ei}\omega_{i}\left(\mathlarger{\sum}_{\substack{f=1\\f\neq e}}^{N_{fp}}\dfrac{F_{s,+}^{f}\left(\mathbf{n}\cdot\mathbf{n}_{f,+}\right)}{4}\mathlarger{\sum}_{j=1}^{N_{fp}}\llbracket u \rrbracket_{fj,+}\psi_{fj,+}\left(\mathbf{r}_{i}\right)\right)\right],
\end{array}
\end{equation}
where the multiplicative scalar $\left|J^{e}\right|$ in front of each term comes from the transformation from the physical to the computational domain.
\end{proof}

Having completed the expansion of the first term, $\Theta_{e,1}$, we now turn our attention to the second term $\Theta_{e,2}$ from~\eqref{eq:decomposition Theta}. 

\begin{lem}\label{lem:theta e2 GL quad}
Using Gauss-Legendre quadratures, we can compute $\Theta_{e,2}$ exactly,
\begin{equation}\label{eq:theta e2 GL quad}
\Theta_{e,2}=-\left|J^{e}\right|\mathlarger{\sum}_{i=1}^{N_{fp}}\tau_{ei}\omega_{i}\llbracket u\rrbracket_{ei}^{2}.
\end{equation}
where $\left(\omega_{i}\right)_{i\in\llbracket 1,N_{fp}\rrbracket}$ and $\left(\mathbf{r}_{i}\right)_{i\in\llbracket 1,N_{fp}\rrbracket}$ are the Gauss-Legendre quadrature weights and points.
\end{lem}
\begin{proof}
Since $\llbracket u\rrbracket_{e} ^{2}$ is a polynomial of degree $2p$ on $\Gamma_{s}$, we can evaluate it exactly with Gauss-Legendre quadratures. Thus the second term $\Theta_{e,2}$ can be written as,
\begin{equation}
\begin{array}{lll}
\Theta_{e,2}&=&-\stretchint{4ex}_{\bs\bs\Gamma_{e}}\tau \llbracket u\rrbracket ^{2} \mbox{d}\Gamma_{e},\\
&=&-\left|J^{e}\right|\mathlarger{\sum}_{i=1}^{N_{fp}}\tau_{ei}\omega_{i}\llbracket u\rrbracket_{ei}^{2}.
\end{array}
\end{equation}
\end{proof}

\begin{thm}\label{thm:tau IP}
Employing the IP scheme for the diffusion equation with affine triangular meshes, periodic boundary conditions and using the ESFR methods; for all edges $e$ and for all flux points $i$, $\tau_{ei}$ greater than $\tau_{ei}^{*}$ implies the energy stability of the solution, with
\begin{equation}\label{eq:criterion IP 2D triangle}
\begin{array}{lll}
\tau_{ei}^{*}&=&\dfrac{1}{4}\min\limits_{\kappa}\mathlarger{\sum}_{k}\Bigg[F_{s,k}^{e}\left(\psi_{ei,k}\left(\mathbf{r}_{i}^{e}\right)-\left|\psi_{ei,k}\left(\mathbf{r}_{i}^{e}\right)\right|\right)\\
&&+\mathlarger{\sum}_{f=1}^{3}\mathlarger{\sum}_{j=1}^{N_{fp}}\left(\dfrac{F_{s,k}^{f}\left|\mathbf{n}\cdot\mathbf{n}_{f}\right|}{2}\left(\left|\psi_{fj,k}\left(\mathbf{r}_{i}^{e}\right)\right|+\dfrac{\omega_{j}}{\omega_{i}}\left|\psi_{ei,k}\left(\mathbf{r}_{j}^{f}\right)\right|\right)\right)\Bigg],
\end{array}
\end{equation}
where $k=\left\lbrace -,+\right\rbrace$. $k=-$ signifies triangle $IJK$, and $k=+$ denotes the adjacent triangle $IJK_{+1}$.
\end{thm}
\begin{proof}
First, we employ the triangular inequality $\left(ab\leq \frac{1}{2}\left(a^{2}+b^{2}\right)\right)$ on all products of $\llbracket u \rrbracket _{ij}\llbracket u \rrbracket_{kl}$ of the result of Lemma~\ref{lem:theta e1 expansion},
\begin{equation}\label{eq:Theta e 2D ini}
\begin{array}{lll}
\Theta_{e,1}&\leq&\left|J^{e}\right|\dfrac{F_{s,-}^{e}}{4}\mathlarger{\sum}_{i=1}^{N_{fp}}\left[\llbracket u\rrbracket_{ei}^{2}\omega_{i}\psi_{ei,-}\left(\mathbf{r}_{i}\right)+\dfrac{1}{2}\left(\mathlarger{\sum}_{\substack{j=1\\ j\neq i}}^{N_{fp}}\omega_{i}\left|\psi_{ej,-}\left(\mathbf{r}_{i}\right)\right|\left(\llbracket u \rrbracket_{ei}^{2}+\llbracket u \rrbracket_{ej}^{2}\right)\right)\right]\vspace{0.2cm}\\
&&+\left|J^{e}\right|\dfrac{F_{s,+}^{e}}{4}\mathlarger{\sum}_{i=1}^{N_{fp}}\left[\llbracket u\rrbracket_{ei}^{2}\omega_{i}\psi_{ei,+}\left(\mathbf{r}_{i}\right)+\dfrac{1}{2}\left(\mathlarger{\sum}_{\substack{j=1\\ j\neq i}}^{N_{fp}}\omega_{i}\left|\psi_{ej,+}\left(\mathbf{r}_{i}\right)\right|\left(\llbracket u \rrbracket_{ei}^{2}+\llbracket u \rrbracket_{ej}^{2}\right)\right)\right]\vspace*{0.2cm}\\
&&+\left|J^{e}\right|\mathlarger{\sum}_{i=1}^{N_{fp}}\left[\llbracket u\rrbracket_{ei}^{2}\omega_{i}\mathlarger{\sum}_{\substack{f=1\\ f\neq e}}^{3}\mathlarger{\sum}_{j=1}^{N_{fp}}\dfrac{F_{s,-}^{f}\left|\mathbf{n}\cdot\mathbf{n}_{f,-}\right|}{8}\left|\psi_{fj,-}\left(\mathbf{r}_{i}\right)\right|\right]\vspace*{0.2cm}\\
&&+\left|J^{e}\right|\mathlarger{\sum}_{i=1}^{N_{fp}}\left[\llbracket u\rrbracket_{ei}^{2}\omega_{i}\mathlarger{\sum}_{\substack{f=1\\ f\neq e}}^{3}\mathlarger{\sum}_{j=1}^{N_{fp}}\dfrac{F_{s,+}^{f}\left|\mathbf{n}\cdot\mathbf{n}_{f,+}\right|}{8}\left|\psi_{fj,+}\left(\mathbf{r}_{i}\right)\right|\right]\vspace*{0.2cm}\\
&&+\left|J^{e}\right|\mathlarger{\sum}_{i=1}^{N_{fp}}\left[\omega_{i}\mathlarger{\sum}_{\substack{f=1\\ f\neq e}}^{3}\mathlarger{\sum}_{j=1}^{N_{fp}}\dfrac{F_{s,-}^{f}\left|\mathbf{n}\cdot\mathbf{n}_{f,-}\right|}{8}\llbracket u\rrbracket_{fj,-}^{2}\left|\psi_{fj,-}\left(\mathbf{r}_{i}\right)\right|\right]\vspace*{0.2cm}\\
&&+\left|J^{e}\right|\mathlarger{\sum}_{i=1}^{N_{fp}}\left[\omega_{i}\mathlarger{\sum}_{\substack{f=1\\ f\neq e}}^{3}\mathlarger{\sum}_{j=1}^{N_{fp}}\dfrac{F_{s,+}^{f}\left|\mathbf{n}\cdot\mathbf{n}_{f,+}\right|}{8}\llbracket u\rrbracket_{fj,+}^{2}\left|\psi_{fj,+}\left(\mathbf{r}_{i}\right)\right|\right].
\end{array}
\end{equation}

The terms $\mathlarger{\sum}_{i=1}^{N_{fp}}\mathlarger{\sum}_{\substack{j=1\\ j\neq i}}^{N_{fp}}\omega_{i}\left|\psi_{ej,-}\left(\mathbf{r}_{i}\right)\right|\llbracket u \rrbracket_{ej}^{2}$ and $\mathlarger{\sum}_{i=1}^{N_{fp}}\mathlarger{\sum}_{\substack{j=1\\ j\neq i}}^{N_{fp}}\omega_{i}\left|\psi_{ej,+}\left(\mathbf{r}_{i}\right)\right|\llbracket u \rrbracket_{ej}^{2}$ of the first two lines of the previous equation need further derivations.

\begin{equation}\label{eq:triang derivations ini}
\begin{array}{lll}
A&=&\mathlarger{\sum}_{i=1}^{N_{fp}}\left[\mathlarger{\sum}_{\substack{j=1\\ j\neq i}}^{N_{fp}}\omega_{i}\left|\psi_{ej,-}\left(\mathbf{r}_{i}\right)\right|\llbracket u \rrbracket_{ej}^{2}\right]\vspace{0.2cm}\\
&=&\mathlarger{\sum}_{i=1}^{N_{fp}}\left[\mathlarger{\sum}_{j=1}^{N_{fp}}\omega_{i}\left|\psi_{ej,-}\left(\mathbf{r}_{i}\right)\right|\llbracket u \rrbracket_{ej}^{2}\right]-\mathlarger{\sum}_{i=1}^{N_{fp}}\omega_{i}\left|\psi_{ei,-}\left(\mathbf{r}_{i}\right)\right|\llbracket u \rrbracket_{ei}^{2}\vspace{0.2cm}\\
&=&\mathlarger{\sum}_{j=1}^{N_{fp}}\left[\llbracket u \rrbracket_{ej}^{2}\mathlarger{\sum}_{i=1}^{N_{fp}}\omega_{i}\left|\psi_{ej,-}\left(\mathbf{r}_{i}\right)\right|\right]-\mathlarger{\sum}_{i=1}^{N_{fp}}\omega_{i}\left|\psi_{ei,-}\left(\mathbf{r}_{i}\right)\right|\llbracket u \rrbracket_{ei}^{2}.
\end{array}
\end{equation}
On the first term of $A$, renaming indices $j$ into $i$ and vice versa, we obtain
\begin{equation}\label{eq:triang derivations fin}
\begin{array}{lll}
A&=&\mathlarger{\sum}_{i=1}^{N_{fp}}\left[\llbracket u \rrbracket_{ei}^{2}\mathlarger{\sum}_{j=1}^{N_{fp}}\omega_{j}\left|\psi_{ei,-}\left(\mathbf{r}_{j}\right)\right|\right]-\mathlarger{\sum}_{i=1}^{N_{fp}}\omega_{i}\left|\psi_{ej,-}\left(\mathbf{r}_{i}\right)\right|\llbracket u \rrbracket_{ej}^{2}\vspace{0.2cm}\\
&=&\mathlarger{\sum}_{i=1}^{N_{fp}}\left[\llbracket u \rrbracket_{ei}^{2}\left(\mathlarger{\sum}_{j=1}^{N_{fp}}\left(\omega_{j}\left|\psi_{ei,-}\left(\mathbf{r}_{j}\right)\right|\right)-\omega_{i}\left|\psi_{ej,-}\left(\mathbf{r}_{i}\right)\right|\right)\right]\vspace{0.2cm}\\
&=&\mathlarger{\sum}_{i=1}^{N_{fp}}\llbracket u \rrbracket_{ei}^{2}\left(\mathlarger{\sum}_{\substack{j=1\\ j\neq i}}^{N_{fp}}\omega_{j}\left|\psi_{ei,-}\left(\mathbf{r}_{j}\right)\right|\right).
\end{array}
\end{equation}

Similar derivations are conducted for $\mathlarger{\sum}_{i=1}^{N_{fp}}\mathlarger{\sum}_{\substack{j=1\\ j\neq i}}^{N_{fp}}\omega_{i}\left|\psi_{ej,+}\left(\mathbf{r}_{i}\right)\right|\llbracket u \rrbracket_{ej}^{2}$. Including these terms into equation~\eqref{eq:Theta e 2D ini} yields,
\begin{equation}\label{eq:Theta e 2D}
\begin{array}{lll}
\Theta_{e,1}&\leq&\left|J^{e}\right|\dfrac{F_{s,-}^{e}}{4}\mathlarger{\sum}_{i=1}^{N_{fp}}\left[\llbracket u\rrbracket_{ei}^{2}\left(\omega_{i}\psi_{ei,-}\left(\mathbf{r}_{i}\right)+\dfrac{1}{2}\mathlarger{\sum}_{\substack{j=1\\ j\neq i}}^{N_{fp}}\omega_{i}\left|\psi_{ej,-}\left(\mathbf{r}_{i}\right)\right|+\omega_{j}\left|\psi_{ei,-}\left(\mathbf{r}_{j}\right)\right|\right)\right]\vspace{0.2cm}\\
&&+\left|J^{e}\right|\dfrac{F_{s,+}^{e}}{4}\mathlarger{\sum}_{i=1}^{N_{fp}}\left[\llbracket u\rrbracket_{ei}^{2}\left(\omega_{i}\psi_{ei,+}\left(\mathbf{r}_{i}\right)+\dfrac{1}{2}\mathlarger{\sum}_{\substack{j=1\\ j\neq i}}^{N_{fp}}\omega_{i}\left|\psi_{ej,+}\left(\mathbf{r}_{i}\right)\right|+\omega_{j}\left|\psi_{ei,+}\left(\mathbf{r}_{j}\right)\right|\right)\right]\vspace*{0.2cm}\\
&&+\left|J^{e}\right|\mathlarger{\sum}_{i=1}^{N_{fp}}\left[\llbracket u\rrbracket_{ei}^{2}\omega_{i}\mathlarger{\sum}_{\substack{f=1\\ f\neq e}}^{3}\mathlarger{\sum}_{j=1}^{N_{fp}}\dfrac{F_{s,-}^{f}\left|\mathbf{n}\cdot\mathbf{n}_{f,-}\right|}{8}\left|\psi_{fj,-}\left(\mathbf{r}_{i}\right)\right|\right]\vspace*{0.2cm}\\
&&+\left|J^{e}\right|\mathlarger{\sum}_{i=1}^{N_{fp}}\left[\llbracket u\rrbracket_{ei}^{2}\omega_{i}\mathlarger{\sum}_{\substack{f=1\\ f\neq e}}^{3}\mathlarger{\sum}_{j=1}^{N_{fp}}\dfrac{F_{s,+}^{f}\left|\mathbf{n}\cdot\mathbf{n}_{f,+}\right|}{8}\left|\psi_{fj,+}\left(\mathbf{r}_{i}\right)\right|\right]\vspace*{0.2cm}\\
&&+\left|J^{e}\right|\mathlarger{\sum}_{i=1}^{N_{fp}}\left[\omega_{i}\mathlarger{\sum}_{\substack{f=1\\ f\neq e}}^{3}\mathlarger{\sum}_{j=1}^{N_{fp}}\dfrac{F_{s,-}^{f}\left|\mathbf{n}\cdot\mathbf{n}_{f,-}\right|}{8}\llbracket u\rrbracket_{fj,-}^{2}\left|\psi_{fj,-}\left(\mathbf{r}_{i}\right)\right|\right]\vspace*{0.2cm}\\
&&+\left|J^{e}\right|\mathlarger{\sum}_{i=1}^{N_{fp}}\left[\omega_{i}\mathlarger{\sum}_{\substack{f=1\\ f\neq e}}^{3}\mathlarger{\sum}_{j=1}^{N_{fp}}\dfrac{F_{s,+}^{f}\left|\mathbf{n}\cdot\mathbf{n}_{f,+}\right|}{8}\llbracket u\rrbracket_{fj,+}^{2}\left|\psi_{fj,+}\left(\mathbf{r}_{i}\right)\right|\right].
\end{array}
\end{equation}

Notice that $\Theta_{e,1}$ is only associated to edge $e_{IJ}$ of $\Theta_{dif}$. However it contains the jump squared of the solution on the four faces of the two neighboring triangles that share edge, $e_{IJ}$.
\begin{figure}[H]
\centering
\includegraphics[width=10cm]{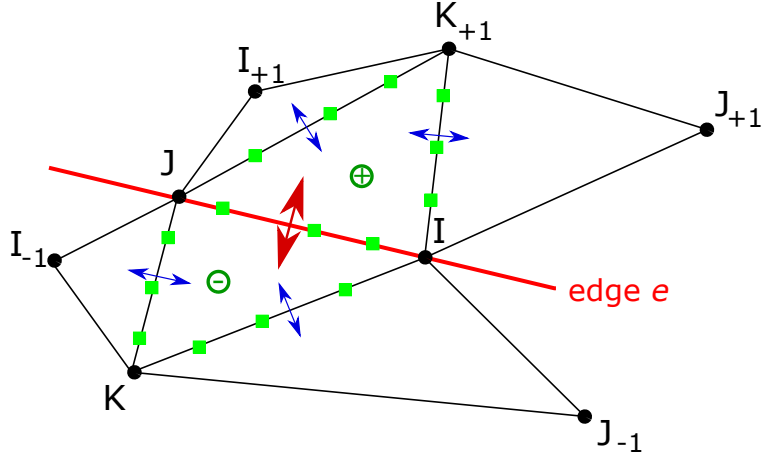}
\caption{Representation of the influence of $\Theta_{e,1}$ on the other edges ; the red arrow represents the jump of the solution on the edge $e$ while the blue arrows are the jumps on the neighboring edges; the green squares represent the flux points; in this example $p=2$.}
\label{fig:jump edges 2D}
\end{figure}

At this juncture we have only computed the value of $\Theta_{e,1}$ over the edge $e_{IJ}$. To simplify the final expression further, we require all the contributions to the jump over edge $e$. From Figure~\ref{fig:jump edges 2D} and equation~\eqref{eq:Theta e 2D}, we identify four additional contributions to the jump $\llbracket u\rrbracket^{2}_{e}$. These contributions are a result of the expansion of equation~\eqref{eq:Theta e 2D} of other edges that contain a contribution towards the edge $e_{IJ}$.
\begin{itemize}
\item $\Theta_{e_{IK}}$ has a contribution of $\left|J^{e}\right|\mathlarger{\sum}_{i=1}^{N_{fp}}\llbracket u\rrbracket_{ei}^{2}\left(\dfrac{F_{s,-}^{IK}\left|\mathbf{n}_{IJ}\cdot \mathbf{n}_{IK}\right|}{8}\mathlarger{\sum_{j=1}}^{N_{fp}}\omega_{j}\left|\psi_{ei,-}\left(\mathbf{r}_{j}^{IK}\right)\right|\right)$. \vspace{0.2cm}
\item $\Theta_{e_{KJ}}$ has a contribution of $\left|J^{e}\right|\mathlarger{\sum}_{i=1}^{N_{fp}}\llbracket u\rrbracket_{ei}^{2}\left(\dfrac{F_{s,-}^{KJ}\left|\mathbf{n}_{KJ}\cdot \mathbf{n}_{IK}\right|}{8}\mathlarger{\sum_{j=1}}^{N_{fp}}\omega_{j}\left|\psi_{ei,-}\left(\mathbf{r}_{j}^{KJ}\right)\right|\right)$.
\item $\Theta_{e_{JK_{+1}}}$ has a contribution of $\left|J^{e}\right|\mathlarger{\sum}_{i=1}^{N_{fp}}\llbracket u\rrbracket_{ei}^{2}\left(\dfrac{F_{s,+}^{JK_{+1}}\left|\mathbf{n}_{JK_{+1}}\cdot \mathbf{n}_{JK_{+1}}\right|}{8}\mathlarger{\sum_{j=1}}^{N_{fp}}\omega_{j}\left|\psi_{ei,+}\left(\mathbf{r}_{j}^{JK_{+1}}\right)\right|\right)$.
\item $\Theta_{e_{IK_{+1}}}$ has a contribution of $\left|J^{e}\right|\mathlarger{\sum}_{i=1}^{N_{fp}}\llbracket u\rrbracket_{ei}^{2}\left(\dfrac{F_{s,+}^{IK_{+1}}\left|\mathbf{n}_{IK_{+1}}\cdot \mathbf{n}_{IK_{+1}}\right|}{8}\mathlarger{\sum_{j=1}}^{N_{fp}}\omega_{j}\left|\psi_{ei,+}\left(\mathbf{r}_{j}^{IK_{+1}}\right)\right|\right)$.
\end{itemize}
Note that while in~\eqref{eq:Theta e 2D}, the correction fields were computed at the flux points of the edge itself, the correction fields of the above terms are calculated at the flux points of the surrounding edges of $e_{IJ}$. We indicate the flux point on the corresponding edges through the superscripts, i.e. $IK$, on the node location, $\mathbf{r}$. 

If we sum all the edges, we finally attain
\begin{equation}\label{eq:Theta dif 1}
\begin{array}{llll}
&\Theta_{dif,1}=\mathlarger{\sum}_{e=1}^{N_{e}}\Theta_{e,1}\\
&\hspace*{1.3cm}\leq\dfrac{1}{4}\mathlarger{\mathlarger{\sum}}_{e=1}^{N_{e}}\Bigg[\left|J^{e}\right|\mathlarger{\mathlarger{\sum}}_{i=1}^{N_{fp}}\Bigg[\llbracket u\rrbracket_{ei}^{2}\Bigg(F_{s,-}^{e}\omega_{i}\underbrace{\psi_{ei,-}\left(\mathbf{r}_{i}^{e}\right)}_{\text{Term A,-}}+F_{s,+}^{e}\omega_{i}\psi_{ei,+}\left(\mathbf{r}_{i}^{e}\right)\\
&+\dfrac{F_{s,-}^{e}}{2}\mathlarger{\sum}_{\substack{j=1\\j\neq i}}^{N_{fp}}\left(\omega_{i}\underbrace{\left|\psi_{ej,-}\left(\mathbf{r}_{i}^{e}\right)\right|}_{\text{Term B,-}}+\omega_{j}\underbrace{\left|\psi_{ei,-}\left(\mathbf{r}_{j}^{e}\right)\right|}_{\text{Term C,-}}\right)+\dfrac{F_{s,+}^{e}}{2}\mathlarger{\sum}_{\substack{j=1\\j\neq i}}^{N_{fp}}\left(\omega_{i}\left|\psi_{ej,+}\left(\mathbf{r}_{i}^{e}\right)\right|+\omega_{j}\left|\psi_{ei,+}\left(\mathbf{r}_{j}^{e}\right)\right|\right)\\
&+\omega_{i}\left[\mathlarger{\sum}_{\substack{f=1\\f\neq e}}^{3}\mathlarger{\sum}_{j=1}^{N_{fp}}\left(\dfrac{F_{s,-}^{f}\left|\mathbf{n}\cdot\mathbf{n}_{f,-}\right|}{2}\underbrace{\left|\psi_{fj,-}\left(\mathbf{r}_{i}^{e}\right)\right|}_{\text{Term D,-}}\right)+\mathlarger{\sum}_{\substack{f=1\\f\neq e}}^{3}\mathlarger{\sum}_{j=1}^{N_{fp}}\left(\dfrac{F_{s,+}^{f}\left|\mathbf{n}\cdot\mathbf{n}_{f,+}\right|}{2}\left|\psi_{fj,+}\left(\mathbf{r}_{i}^{e}\right)\right|\right)\right]\vspace*{0.3cm}\\
&+\mathlarger{\sum}_{\substack{f=1\\ f\neq e}}^{3}\left[\dfrac{F_{s,-}^{f}\left|\mathbf{n}\cdot\mathbf{n}_{f,-}\right|}{2}\mathlarger{\sum}_{j=1}^{N_{fp}}\left(\omega_{j}\underbrace{\left|\psi_{ei,-}\left(\mathbf{r}_{j}^{f}\right)\right|}_{\text{Term E,-}}\right)\right]+\mathlarger{\sum}_{\substack{f=1\\f\neq e}}^{3}\left[\dfrac{F_{s,+}^{f}\left|\mathbf{n}\cdot\mathbf{n}_{f,+}\right|}{2}\mathlarger{\sum}_{j=1}^{N_{fp}}\left(\omega_{j}\left|\psi_{ei,+}\left(\mathbf{r}_{j}^{f}\right)\right|\right)\right]\Bigg)\Bigg]\Bigg].
\end{array}
\end{equation}

For the jump of the flux point $i$ on the edge $e$, $\llbracket u\rrbracket_{ei}^{2}$, we have the influence of all the correction fields surrounding $e$ evaluated at the flux point $\mathbf{r}_{i}^{e}$ (Terms B and D). In addition, we have the influence of the correction field $\psi_{ei}$ evaluated at all the flux points (Terms A, C and E). Although this last equation may seem complicated, it is quite logical. These individual contributions are best understood through a graphical means as depicted in Figure~\ref{fig:influence correction field}. For simplicity, we just represent the terms arising from triangle $IJK$, since we have a symmetrical influence from triangle $IJK_{+1}$.
\begin{figure}[H]
\centering
\includegraphics[width=10cm]{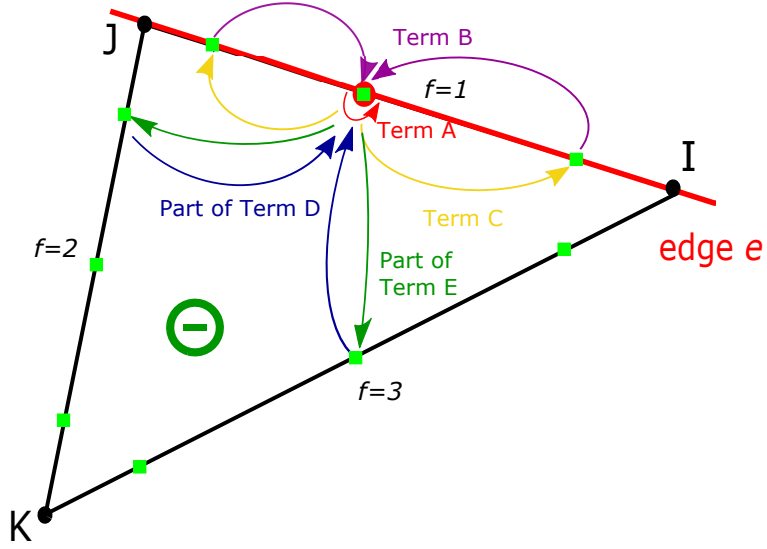}
\caption{Influence of the correction fields, for the left side of the edge (Triangle $IJK$), on the flux point $\left(f,j\right)=\left(1,2\right)$ for $p=2$.}
\label{fig:influence correction field}
\end{figure}

We now return to equation~\eqref{eq:decomposition Theta} and provide a complete expansion for the diffusion terms toward stability. We combine equations~\eqref{eq:theta e2 GL quad} and~\eqref{eq:Theta dif 1}, and introduce a simplification in our notation where the parameter $k=\left\lbrace -,+\right\rbrace$; where $k=-$ signifies triangle $IJK$, while $k=+$ denotes triangle $IJK_{+1}$. The first term in equation~\eqref{eq:Theta dif 1} is simplified into,

\begin{equation}\label{eq:final expression theta dif,1}
\begin{array}{lll}
\Theta_{dif,1}&\leq&\mathlarger{\mathlarger{\sum}}_{e=1}^{N_{e}}\Bigg[\left|J^{e}\right|\mathlarger{\mathlarger{\sum}}_{i=1}^{N_{fp}}\Bigg[\dfrac{\llbracket u\rrbracket_{ei}^{2}}{4}\mathlarger{\sum}_{k}\Bigg[F_{s,k}^{e}\omega_{i}\psi_{ei,k}\left(\mathbf{r}_{i}^{e}\right)+\dfrac{F_{s,k}^{e}}{2}\mathlarger{\sum}_{\substack{j=1\\j\neq i}}^{N_{fp}}\left(\omega_{i}\left|\psi_{ej,k}\left(\mathbf{r}_{i}^{e}\right)\right|+\omega_{j}\left|\psi_{ei,k}\left(\mathbf{r}_{j}^{e}\right)\right|\right)\\
&&+\mathlarger{\sum}_{\substack{f=1\\f\neq e}}^{3}\mathlarger{\sum}_{j=1}^{N_{fp}}\left(\dfrac{F_{s,k}^{f}\left|\mathbf{n}\cdot\mathbf{n}_{f,k}\right|}{2}\left(\omega_{i}\left|\psi_{fj,k}\left(\mathbf{r}_{i}^{e}\right)\right|+\omega_{j}\left|\psi_{ei,k}\left(\mathbf{r}_{j}^{f}\right)\right|\right)\right)\Bigg]\Bigg]\Bigg],
\end{array}
\end{equation}

and using Lemma~\ref{lem:theta e2 GL quad}, we compute $\Theta_{dif}$,
\begin{equation}
\begin{array}{lll}
\Theta_{dif}&=&\mathlarger{\sum}_{e=1}^{N_{e}}\left[\Theta_{e,1}+\Theta_{e,2}\right]\\
&\leq&\mathlarger{\mathlarger{\sum}}_{e=1}^{N_{e}}\left|J^{e}\right|\mathlarger{\mathlarger{\sum}}_{i=1}^{N_{fp}}\Bigg[\llbracket u\rrbracket_{ei}^{2}\Bigg[-\tau_{ei}\omega_{i}+\dfrac{1}{4}\mathlarger{\sum}_{k}\Bigg[F_{s,k}^{e}\omega_{i}\psi_{ei,k}\left(\mathbf{r}_{i}^{e}\right)\\
&&+\dfrac{F_{s,k}^{e}}{2}\mathlarger{\sum}_{\substack{j=1\\j\neq i}}^{N_{fp}}\left(\omega_{i}\left|\psi_{ej,k}\left(\mathbf{r}_{i}^{e}\right)\right|+\omega_{j}\left|\psi_{ei,k}\left(\mathbf{r}_{j}^{e}\right)\right|\right)\\
&&+\mathlarger{\sum}_{\substack{f=1\\f\neq e}}^{3}\mathlarger{\sum}_{j=1}^{N_{fp}}\left(\dfrac{F_{s,k}^{f}\left|\mathbf{n}\cdot\mathbf{n}_{f,k}\right|}{2}\left(\omega_{i}\left|\psi_{fj,k}\left(\mathbf{r}_{i}^{e}\right)\right|+\omega_{j}\left|\psi_{ei,k}\left(\mathbf{r}_{j}^{f}\right)\right|\right)\right)\Bigg]\Bigg]\Bigg].
\end{array}
\end{equation}

To evaluate the limiting value, $\tau^*$ to ensure stability, we remove the exclusion of $f\neq e$ and similarly for the second to last term, where we remove the exclusion of $j\neq i$. This results in an additional term $\left|\psi_{ei,k}\left(\mathbf{r}_{i}^{e}\right)\right|$ that is subtracted from the first term. We thus ensure energy stability for the diffusion equation in two dimensions for triangles. The IP scheme requires $\tau_{e,i}\geq\tau_{e,i}^{*}\,\,\forall e\,\forall i$ where $\tau^{*}$ is defined as
\begin{equation}\label{eq:IP stability without minimization}
\begin{array}{lll}
\tau_{ei}^{*}&=&\dfrac{1}{4}\mathlarger{\sum}_{k}\Bigg[F_{s,k}^{e}\left(\psi_{ei,k}\left(\mathbf{r}_{i}^{e}\right)-\left|\psi_{ei,k}\left(\mathbf{r}_{i}^{e}\right)\right|\right)\\
&&+\mathlarger{\sum}_{f=1}^{3}\mathlarger{\sum}_{j=1}^{N_{fp}}\left(\dfrac{F_{s,k}^{f}\left|\mathbf{n}\cdot\mathbf{n}_{f}\right|}{2}\left(\left|\psi_{fj,k}\left(\mathbf{r}_{i}^{e}\right)\right|+\dfrac{\omega_{j}}{\omega_{i}}\left|\psi_{ei,k}\left(\mathbf{r}_{j}^{f}\right)\right|\right)\right)\Bigg].
\end{array}
\end{equation}

From Postulate~\ref{pos:independent} the ESFR high-order method with either the IP or BR2 schemes for the diffusion problem is independent of $\kappa$. As a consequence, the energy stability is also independent of $\kappa$. Thus we finally obtain,,
\begin{equation}
\begin{array}{lll}
\tau_{ei}^{*}&=&\dfrac{1}{4}\min\limits_{\kappa}\mathlarger{\sum}_{k}\Bigg[F_{s,k}^{e}\left(\psi_{ei,k}\left(\mathbf{r}_{i}^{e}\right)-\left|\psi_{ei,k}\left(\mathbf{r}_{i}^{e}\right)\right|\right)\\
&&+\mathlarger{\sum}_{f=1}^{3}\mathlarger{\sum}_{j=1}^{N_{fp}}\left(\dfrac{F_{s,k}^{f}\left|\mathbf{n}\cdot\mathbf{n}_{f}\right|}{2}\left(\left|\psi_{fj,k}\left(\mathbf{r}_{i}^{e}\right)\right|+\dfrac{\omega_{j}}{\omega_{i}}\left|\psi_{ei,k}\left(\mathbf{r}_{j}^{f}\right)\right|\right)\right)\Bigg].
\end{array}
\end{equation}
\end{proof}

\begin{rmk}
We could simplify~\eqref{eq:criterion IP 2D triangle} further by noting that $\left|\mathbf{n}\cdot\mathbf{n}_{f}\right|\leq 1$, and that the quantity $F_{s,max}=\max\left(F_{s}\right)$. However such simplification would result in a less accurate condition and hence a larger value of $\tau^{*}$. We refrain from this simplification as increasing $\tau$ results in an increase in $\Delta t_{max}$ \cite{sam_article1D}.
\end{rmk}

\begin{rmk}
In the following graphs, we represent the maximum value of the array of $\tau_{theory}^{*}$ evaluated via equation~\eqref{eq:IP stability without minimization}. These numerical simulations show that the minimum of $\tau_{ei}^{*}$ is obtained when $\kappa$ approaches $\kappa_{+}$. In the rest of the article, we will compute $\tau_{ei}^{*}$ in equation~\eqref{eq:criterion IP 2D triangle} with $\kappa=\kappa_{+}$ for all ESFR schemes.
\begin{figure}[H] 
\centering

\begin{subfigure}{0.45\textwidth}
\centering
\includegraphics[width=\linewidth]{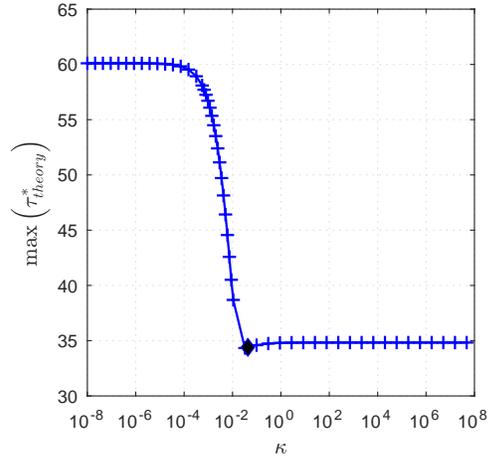}
\caption{$\max\left(\tau_{theory}^{*}\right)$ for $p=2$.} \label{fig:indep kappa p2}
\end{subfigure}\hspace*{\fill}
\begin{subfigure}{0.45\textwidth}
\centering
\includegraphics[width=\linewidth]{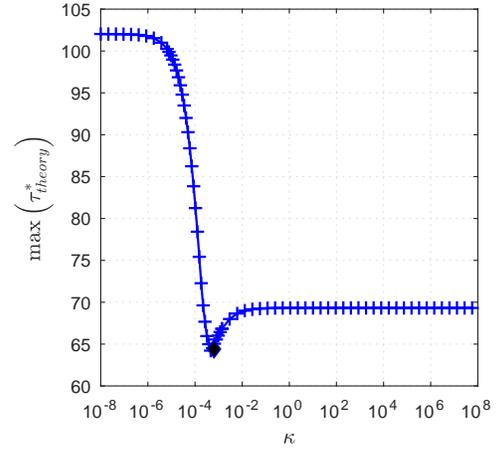}
\caption{$\max\left(\tau_{theory}^{*}\right)$ for $p=3$.} \label{indep kappa p3}
\end{subfigure}

\medskip
\begin{subfigure}{0.45\textwidth}
\centering
\includegraphics[width=\linewidth]{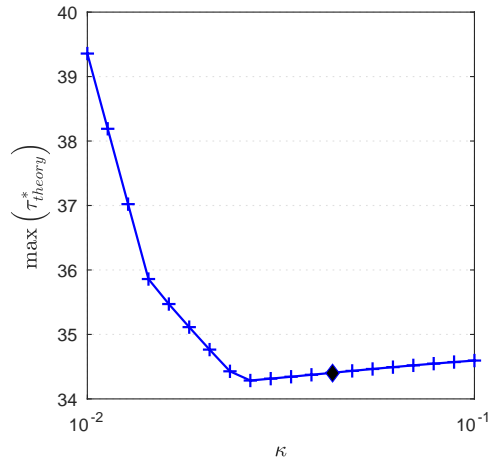}
\caption{$\max\left(\tau_{theory}^{*}\right)$ for $p=2$ zoom in.} \label{indep kappa p2 zoom}
\end{subfigure}\hspace*{\fill}
\begin{subfigure}{0.45\textwidth}
\centering
\includegraphics[width=\linewidth]{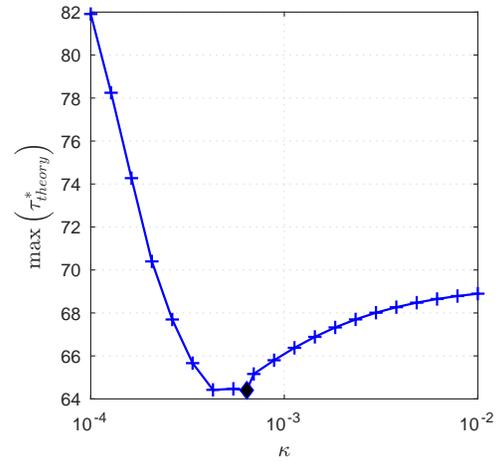}
\caption{$\max\left(\tau_{theory}^{*}\right)$ for $p=3$ zoom in.} \label{indep kappa p3 zoom}
\end{subfigure}
\caption{Minimization of $\max\left(\tau_{theory}^{*}\right)$ along $\kappa$ for different $p$ where the diamond corresponds to $\kappa=\kappa_{+}$.}
\label{fig:plot tau minimization}
\end{figure}
These numerical simulations are obtained for the mesh represented in Figure~\ref{fig:regular mesh} with 128 elements.
\end{rmk}

\subsection{Numerical results}\label{sec:Numerical results IP 2D}

The purpose of the following numerical simulation is to find the minimal numerical penalty term $\tau_{numerical}^{*}$  which guarantees stability for a defined problem. This numerical penalty term will be compared with~\eqref{eq:criterion IP 2D triangle}.

The problem solved is of the following: Find $u\left(x,y,t\right)$ such that
\begin{equation}\label{PDE}
\left\lbrace
\begin{array}{llll}
\dfrac{\partial u}{\partial t}&=&b\Delta u,&\,\text{for}\, x\times y \in\left[-1;1\right]^{2}\, \text{and}\, t\in\left[0,2\right],\vspace{0.2cm}\\
u\left(x,y,0\right)&=&\sin\left(\pi x\right)\sin\left(\pi y\right).
\end{array}
\right.
\end{equation}
We impose periodic boundary conditions and the exact solution for this system is $u_{exact}=e^{-2b\pi^{2}t}\sin\left(\pi x\right)\sin\left(\pi y\right)$.

We take $b=0.1$ and use the fourth order five stage Runge-Kutta as the time integrator \cite{carpenter_fourth-order_1994}. The flux points are taken as the Gauss-Legendre nodes while the solution points are taken as the $\alpha$-optimised nodes \cite{NDG}. The mesh generated is regular as shown in Figure~\ref{fig:regular mesh}.
\begin{figure}[H]
\centering
\includegraphics[width=3.2in]{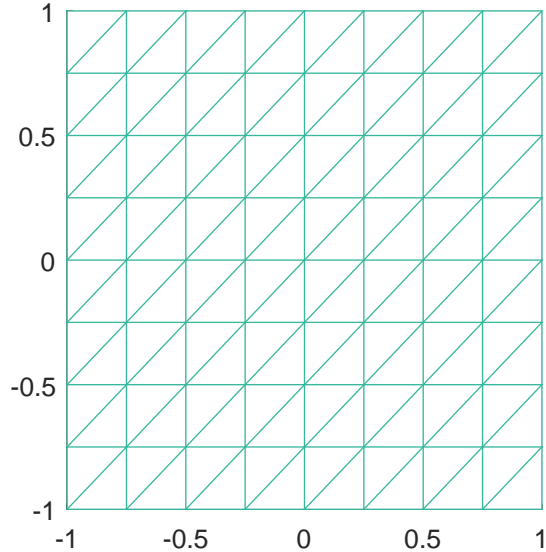}
\caption{Representation of a regular mesh $8\times 8\times 2$.}
\label{fig:regular mesh}
\end{figure}

The time step is taken as advised by Hesthaven and Warburton \cite{NDG},
\begin{equation}
\Delta t=CFL\left(\dfrac{2}{3}\min \Delta r_{i}^{2}\right)\min_{\Omega}\left(\dfrac{r_{D}}{\left| b\right|}\right),
\end{equation}
where $CFL=10^{-2}$ in our cases, $\Delta r_{i}$ is a measure of the distance between the solution points and $r_{D}$ is the ratio between the perimeter of the triangle with its area.

While in our proof, $\tau_{theory}^{*}$ is an array, we consider here, for simplicity, that $\tau$ is a constant for all the edges and flux points. We apply the same procedure as \cite{sam_article1D}: choose a criterion of stability on the upper bound of the solution, $\left|u\left(x,t\right)\right|\leq u_{max}$. As, for all $t$, $\left|u_{exact}\left(x,t\right)\right|\leq 1$, choosing $u_{max}=2$ is sufficient and allow some instabilities arising from the scheme.  Then select a sufficiently low value for $\tau_{0}$ as a starting point to ensure an unstable solution and increase it by $\mbox{d}\tau$ of 0.1 until the solution is stable. Taking the final time equal to 2 ensures that a few thousands iterations are run. The results are presented for four methods: $c_{DG}/\kappa_{DG}$, $c_{DG}/\kappa_{+}$, $c_{+}/\kappa_{DG}$ and $c_{+}/\kappa_{+}$ where $c_{+}=\kappa_{+}$ is taken to be equal to the last column of Table~\ref{tab:c+ triangle}.

\begin{table}[H]
\centering
\setlength\tabcolsep{4pt}
\begin{minipage}{0.45\textwidth}
\centering
\resizebox{\textwidth}{!}{
\begin{tabular}{|c||c|c||c|c|}
\hline
$p$ & \multicolumn{2}{c||}{2} & \multicolumn{2}{c|}{3}\\
\hline
\bf \backslashbox{$c$}{$\kappa$} & $\kappa_{DG}$ & $\kappa_{+}$ & $\kappa_{DG}$ & $\kappa_{+}$\\
\hline
$\max\left(\tau_{theory}^{*}\right)$ & 34.4 & 34.4 & 64.4 & 64.4\\
\hline
\hline
$c_{DG}$ & 19.3 & 19.3 & 37.0 & 37.0\\
\hline
$c_{+}$ & 15.3 & 15.3 & 36.4 & 36.4\\
\hline
\end{tabular}
}
\caption{$\tau_{numerical}^{*}$ for the IP scheme for p=2 and p=3 for a $8\times 8\times 2$ mesh.}
\label{tab:result IP 2D 8 elements}
\end{minipage}
\hfill
\begin{minipage}{0.45\textwidth}
\centering
\resizebox{\textwidth}{!}{
\begin{tabular}{|c||c|c||c|c|}
\hline
$p$ & \multicolumn{2}{c||}{2} & \multicolumn{2}{c|}{3}\\
\hline
\bf \backslashbox{$c$}{$\kappa$} & $\kappa_{DG}$ & $\kappa_{+}$ & $\kappa_{DG}$ & $\kappa_{+}$\\
\hline
$\max\left(\tau_{theory}^{*}\right)$ & 64.8 & 64.8 & 129.9 & 129.9\\
\hline
\hline
$c_{DG}$ & 41.7 & 41.7 & 76.1 & 76.1\\
\hline
$c_{+}$ & 39.1 & 39.2 & 75.7 & 75.7\\
\hline
\end{tabular}
}
\caption{$\tau_{numerical}^{*}$ for the IP scheme for p=2 and p=3 for a $16\times 16\times 2$ mesh.}
\label{tab:result IP 2D 16 elem}
\end{minipage}
\end{table}

For both tables \ref{tab:result IP 2D 8 elements} and \ref{tab:result IP 2D 16 elem}, the maximum of the criterion found in~\eqref{eq:criterion IP 2D triangle} is greater than $\tau_{numerical}^{*}$. These results do not validate~\eqref{eq:criterion IP 2D triangle} for every edge and every flux point as $\tau_{numerical}$ is taken to be a constant. However since $\max\left(\tau_{theory}^{*}\right)\geq \tau_{numerical}^{*}$, this lead us to believe that the criterion~\eqref{eq:criterion IP 2D triangle} is valid.
\section{BR2 stability condition}\label{sec:BR2 section}

\subsection{Theoretical result}

For the BR2 scheme, the steps are identical to that of the IP scheme with the exception of $\mathbf{q}^*$. Therefore Lemma~\ref{lem:theta dif} is modified into
\begin{equation}\label{eq:theta dif BR2}
\Theta_{dif}=\mathlarger{\sum}_{e=1}^{N_{e}}\left[\underbrace{\stretchint{5ex}_{\bs\bs\Gamma_{e}}\left(-\{\{\mathbf{q}-\mathbf{\nabla} u\}\}\right)\cdot \mathbf{\llbracket u\rrbracket}_{e}\,\mbox{d}\Gamma_{e}}_{\Theta_{e,1}} +\underbrace{\stretchint{5ex}_{\bs\bs\Gamma_{e}}s\, \{\{\mathbf{r^{e}}\left(\mathbf{\llbracket u \rrbracket}\right)\}\}\cdot\mathbf{\llbracket u \rrbracket}_{e}}_{\Theta_{e,2}}\,\mbox{d}\Gamma_{e}\right].
\end{equation}

The first term $\Theta_{e,1}$ is the same as the IP scheme and hence $\Theta_{dif,1}=\sum_{e=1}^{N_{e}}\Theta_{e,1}$ follows the same derivations up to~\eqref{eq:final expression theta dif,1}; however, the second term requires additional derivations. Similar to the 1D case \cite{sam_article1D}, we expand $\{\{\mathbf{r^{e}}\left(\mathbf{\llbracket u \rrbracket}\right)\}\}$ in terms of $\mathbf{\llbracket u \rrbracket}_{e}$.  We base our analysis on the article of Huynh \cite{huynh_high-order_2011}, where he presented an equivalence between the DG and the FR formulations for the advection equation. This equivalence has been presented for all high dimensions and for curvilinear elements~\cite{zwanenburg_equivalence_2016}; it is, once again, presented in the following for completeness of the paper. The advection equation is similar to \eqref{eq:main compu eq 2D} with $\mathbf{f}=-u\mathbf{a}$, where $\mathbf{a}$ is the velocity. The FR approach, for the advection problem, is similar to~\eqref{eq:main compu corrected eq 2D},
\begin{equation}\label{eq:adv FR 2D}
\dfrac{\partial \hat{u}_{n}}{\partial t}=\hat{\mathbf{\nabla}}\cdot\hat{\mathbf{f}}_{n}\left(\hat{u}_{n}\right)+\hat{\mathbf{\nabla}}\cdot\sum_{f=1}^{3}\sum_{j=1}^{N{fp}}\left[\left(\hat{\mathbf{f}}_{n,fj}^{*}-\hat{\mathbf{f}}_{n,fj}\right)\cdot\hat{\mathbf{n}}_{fj}\right]\mathbf{h}_{fj}\left(\mathbf{r}\right).
\end{equation}
However the DG formulation, in the strong form, as given in~\cite{NDG}, for an element $\Omega_{n}$ is
\begin{equation}\label{eq:adv DG 2D}
\stretchint{5ex}_{\bs\bs\Omega_{s}}\dfrac{\partial \hat{u}_{n}}{\partial t}\Phi\,\mbox{d}\Omega_{s}=\stretchint{5ex}_{\bs\bs \Omega_{s}}\hat{\mathbf{\nabla}}\cdot\hat{\mathbf{f}}_{n}\left(u_{n}\right)\Phi\,\mbox{d}\Omega_{s}+\stretchint{5ex}_{\bs\bs\Gamma_{s}}\left[\left(\hat{\mathbf{f}}^{*}_{n}-\hat{\mathbf{f}}_{n}\right)\cdot\hat{\mathbf{n}}\right]\Phi\,\mbox{d}\Gamma_{s},
\end{equation}
with $\Phi$ as the test function. To retrieve the FR formulation from equation~\eqref{eq:adv DG 2D}, we must define, for a face $f$, the correction field $\delta_{f}$ such that
\begin{equation}\label{eq:delta_def 1}
\int_{\Omega_{s}}\delta_{f}\Phi\,\mbox{d}\Omega_{s}=\int_{\Gamma_{s,f}}\left[\left(\hat{\mathbf{f}}_{n,f}^{*}-\hat{\mathbf{f}}_{n,f}\right)\cdot\hat{\mathbf{n}}\right]\Phi\,\mbox{d}\Gamma_{s,f},
\end{equation}
where $\delta_{f}$ is in fact a lifting operator associated to the face $f$, $\Gamma_{s,f}$, of $\Omega_{s}$.
We evaluate exactly the integral over $\Gamma_{s,f}$ employing Gauss-Legendre quadratures, (nodes $\mathbf{r}_{j}^{f}$, weight $\omega_{j}$),
\begin{equation}
\int_{\Omega_{s}}\delta_{f}\Phi\,\mbox{d}\Omega_{s}=\mathlarger{\sum}_{j=1}^{N_{fp}}\left.\left[\left(\hat{\mathbf{f}}_{f}^{*}-\hat{\mathbf{f}}_{n,fj}\right)\cdot\hat{\mathbf{n}}_{n,fj}\right]\right|_{\mathbf{r}_{j}^{f}}\omega_{j}\Phi\left(\mathbf{r}_{j}^{f}\right).
\end{equation}
Upon removing the corrective flux, we must define $\delta_{fj}^{*}$ such that
\begin{equation}\label{eq:delta_def 2}
\int_{\Omega_{s}}\delta_{fj}^{*}\Phi\,\mbox{d}\Omega_{s}=\omega_{j}\Phi\left(\mathbf{r}_{j}^{f}\right).
\end{equation}
From \eqref{eq:adv DG 2D}, let us decompose $\Gamma_{
s}$ into $\sum_{f=1}^{3}\Gamma_{s,f}$,
\begin{equation}
\begin{array}{rll}
\stretchint{5ex}_{\bs\bs\Omega_{s}}\dfrac{\partial \hat{u}_{n}}{\partial t}\Phi\,\mbox{d}\mbox{r}&=&\stretchint{5ex}_{\bs\bs \Omega_{s}}\hat{\mathbf{\nabla}}\cdot\hat{\mathbf{f}}_{n}\left(\hat{u}_{n}\right)\Phi\,\mbox{d}\Omega_{s}+\mathlarger{\sum}_{f=1}^{3}\stretchint{5ex}_{\bs\bs\Gamma_{s,f}}\left[\left(\hat{\mathbf{f}}_{n}^{*}-\hat{\mathbf{f}}_{n}\right)\cdot\hat{\mathbf{n}}\right]\Phi\,\mbox{d}\Gamma_{s,f},\\
&=&\stretchint{5ex}_{\bs\bs\Omega_{s}}\mathbf{\nabla}\cdot\hat{\mathbf{f}}_{n}\left(\hat{u}_{n}\right)\Phi\,\mbox{d}\Omega_{s}+\mathlarger{\sum}_{f=1}^{3}\mathlarger{\sum}_{j=1}^{N_{fp}}\left[\left(\hat{\mathbf{f}}_{n,fj}^{*}-\hat{\mathbf{f}}_{n,fj}\right)\cdot\hat{\mathbf{n}}_{fj}\right]\stretchint{5ex}_{\bs\bs\Omega_{s}}\delta_{fj}^{*}\Phi\,\mbox{d}\Omega_{s}\\
0&=&\stretchint{7ex}_{\bs\bs\Omega_{s}}\left[\left(-\dfrac{\partial \hat{u}_{n}}{\partial t}+\hat{\mathbf{\nabla}}\cdot\hat{\mathbf{f}}_{n}\left(\hat{u}_{n}\right)+\mathlarger{\sum}_{f=1}^{3}\mathlarger{\sum}_{j=1}^{N_{fp}}\left[\left(\hat{\mathbf{f}}_{n,fj}^{*}-\hat{\mathbf{f}}_{n,fj}\right)\cdot\hat{\mathbf{n}}_{fj}\right]\delta_{fj}^{*}\right)\Phi \right]\,\mbox{d}\Omega_{s}.
\end{array}
\end{equation}

Using the lifting operator enables the test function $\Phi$ to be factored out and removed from the integral over the different terms of the PDE, resulting into a differential formulation: the FR approach,
\begin{equation}\label{eq:FR with delta adv}
\dfrac{\partial \hat{u}_{n}}{\partial t}=\mathbf{\nabla}\cdot\hat{\mathbf{f}}_{n}\left(\hat{u}_{n}\right)+\mathlarger{\sum}_{f=1}^{3}\mathlarger{\sum}_{j=1}^{N_{fp}}\left[\left(\hat{\mathbf{f}}_{n,fj}^{*}-\hat{\mathbf{f}}_{n,fj}\right)\cdot\hat{\mathbf{n}}_{fj}\right]\delta_{fj}^{*}.
\end{equation}

\begin{lem}\label{lem:equivalence DG FR}
Employing the ESFR correction fields in equation~\eqref{eq:ESFR property} to define $\delta_{fj}^{*}$, there is a unique value of $c$ such that equation~\eqref{eq:delta_def 2} is valid and this value is $c=0$.
\end{lem}
\begin{proof}
Replacing $L_{i}$ by $\Phi$, a test function, in equation~\eqref{eq:ESFR property}, we obtain the following,
\begin{equation}
\begin{array}{lll}
&&\stretchint{5ex}_{\bs \bs \Omega_{s}}\mathbf{h}_{fj}\cdot\hat{\mathbf{\nabla}} \Phi \mbox{d}\Omega_{s}=c\mathlarger{\sum}_{m=1}^{p+1}\binom{p}{m-1}\left(D^{\left(m,p\right)}\Phi\right)\left(D^{\left(m,p\right)}\phi_{fj}\right)\\
&\Leftrightarrow&\stretchint{5ex}_{\bs \bs \Omega_{s}}\hat{\mathbf{\nabla}}\left(\mathbf{h}_{fj}\Phi\right)\mbox{d}\Omega_{s}-\stretchint{5ex}_{\bs \bs \Omega_{s}}\Phi\hat{\mathbf{\nabla}}\cdot \mathbf{h}_{fj} \mbox{d}\Omega_{s}=c\mathlarger{\sum}_{m=1}^{p+1}\binom{p}{m-1}\left(D^{\left(m,p\right)}\Phi\right)\left(D^{\left(m,p\right)}\phi_{fj}\right)\\
&\Leftrightarrow&\stretchint{5ex}_{\bs \bs \Omega_{s}}\phi_{fj} \Phi\,\mbox{d}\Omega_{s}=\stretchint{5ex}_{\bs \bs \Gamma_{s}} \Phi\left(\mathbf{h}_{fj}\cdot \mathbf{\hat{n}}\right)\,\mbox{d}\Gamma_{s}-c\mathlarger{\sum}_{m=1}^{p+1}\binom{p}{m-1}\binom{p}{m-1}\left(D^{\left(m,p\right)}\Phi\right)\left(D^{\left(m,p\right)}\phi_{fj}\right).
\end{array}
\end{equation}
Since $\mathbf{h}_{fj}\cdot \mathbf{\hat{n}}\in R_{p}\left(\Gamma_{s}\right)$, we have $\Phi\left(\mathbf{h}_{fj}\cdot \mathbf{\hat{n}}\right)$ a polynomial of degree $2p$ on the edge. We use Gauss-Legendre quadrature ($N_{fp}$ points, same quadrature as~\eqref{eq:delta_def 2}) to compute the integral of the right-hand side. Moreover $\mathbf{h}_{fj}$ verifies~\eqref{eq:h dot n}, hence we obtain
\begin{equation}
\int_{\Omega_{s}}\phi_{fj}\Phi\,\mbox{d}\Omega_{s}=\omega_{j}\Phi\left(\mathbf{r}_{j}^{f}\right)+c\mathlarger{\sum}_{m=1}^{p+1}\binom{p}{m-1}\left(D^{\left(m,p\right)}\Phi\right)\left(D^{\left(m,p\right)}\phi_{fj}\right).
\end{equation}

Therefore,
\begin{equation}
\begin{array}{lll}
&&\stretchint{5ex}_{\bs\bs\Omega_{s}}\phi_{fj}\Phi\,\mbox{d}x=\omega_{j}\Phi\left(\mathbf{r}_{j}^{f}\right)\\
&\Leftrightarrow&c\mathlarger{\sum}_{m=1}^{p+1}\binom{p}{m-1}\left(D^{\left(m,p\right)}\Phi\right)\left(D^{\left(m,p\right)}\phi_{fj}\right)=0
\end{array}
\end{equation}
Taking, for instance, $\Phi=\phi_{fj}$, the previous equation is valid if $c=0$.
\end{proof}

\begin{rmk}
With $c=0$, the FR formulation is equivalent to the DG formulation. Hence the associated correction function is denoted $\phi_{fj}^{DG}$.
\end{rmk}

We now apply the analogy between the correction field, $\phi_{fj}^{DG}$ and the lifting operator $\mathbf{r^{e}}$. The support of $\mathbf{r^{e}}$ is the union of the element of the triangles forming the edge $e$. On each one of the elements, $\mathbf{r^{e}}$ is a polynomial of degree $p$.

We define the space $\Omega_{e}=\Omega_{-}\bigcup\Omega_{+}$. Where $\Omega_{-}$ (resp. $\Omega_{+}$) is the interior (resp. exterior) element of edge $e$. Referring to Figure~\ref{fig:Part domain}, $\Omega_{-}=\Omega_{IJK}=\Omega_{n}$. From the affine mapping defined in~\eqref{eq:affine mapping}, we define the quantities,
\begin{eqnarray}
\mathbf{v}_{n}^{a}&=&-\dfrac{1}{2}\left(\mathbf{v}_{1,n}-\mathbf{v}_{2,n}\right)\vspace{0.2cm},\\
\mathbf{v}_{n}^{b}&=&-\dfrac{\sqrt{3}}{6}\left(\mathbf{v}_{1,n}\mathbf{v}_{2,n}-2\mathbf{v}_{3,n}\right)\vspace{0.2cm},\\
\mathbf{\bar{v}}_{n}&=&\dfrac{1}{3}\left(\mathbf{v}_{1,n}+\mathbf{v}_{2,n}+\mathbf{v}_{3,n}\right)\vspace{0.2cm},\\
\left|\mathbf{V}_{n}\right|&=&Det\begin{pmatrix}
v_{n,x}^{a}&v_{n,x}^{b}\\
v_{n,y}^{a}&v_{n,y}^{b}\\
\end{pmatrix},
\end{eqnarray}
where $v_{n,x}^{a}$ is the $x-$component of $\mathbf{v}_{n}^{a}$. We then define the surjection
\begin{equation}\label{eq:affine mapping re}
\begin{array}{lccl}
\mathcal{M}^{-1}_{e}\colon &\Omega_{e}&\to&\Omega_{s}\\
&\left(x,y\right)&\mapsto&\dfrac{1}{\left|\mathbf{V}_{-}\right|}\begin{pmatrix}
\left(x-\bar{v}_{-,x}\right)v_{-,y}^{b}-\left(y-\bar{v}_{-,y}\right)v_{-,x}^{b}\\
-\left(x-\bar{v}_{-,x}\right)v_{-,x}^{a}+\left(y-\bar{v}_{-,y}\right)v_{-,x}^{a}
\end{pmatrix}\left.\chi\right|_{-}\left(x,y\right)\vspace{0.2cm}\\
&&+&\dfrac{1}{\left|\mathbf{V}_{+}\right|}\begin{pmatrix}
\left(x-\bar{v}_{+,x}\right)v_{+,y}^{b}-\left(y-\bar{v}_{+,y}\right)v_{+,x}^{b}\\
-\left(x-\bar{v}_{+,x}\right)v_{+,x}^{a}+\left(y-\bar{v}_{+,y}\right)v_{+,x}^{a}
\end{pmatrix}\left.\chi\right|_{+}\left(x,y\right)
\end{array}
\end{equation}
where $\left.\chi\right|_{\Omega_{i}}\left(x,y\right)$ is equal to 1 if $\left(x,y\right)\in\Omega_{i}$, or 0 if $\left(x,y\right)\not\in\Omega_{i}$.
\begin{thm}\label{thm:formula re}
The lifting operator, $\mathbf{r^{e}}$, employed in the BR2 scheme,  is a linear combination of the correction field associated to the DG method and is equivalent to the following formula,
\begin{equation}\label{eq:formula re}
\mathbf{r^{e}}\left(\mathbf{\llbracket u \rrbracket}\right)\left(\mathbf{x}\right)=-\left(\mathlarger{\sum}_{j=1}^{N_{fp}}\mathbf{\llbracket u \rrbracket}_{ej}\left[\dfrac{F_{s,-}^{e}}{2}\phi_{ej,-}^{DG}\left(\mathcal{M}_{e}^{-1}\left(\mathbf{x}\right)\right)\left.\chi\right|_{-}\left(\mathbf{x}\right)+\dfrac{F_{s,+}^{e}}{2}\phi_{ej,+}^{DG}\left(\mathcal{M}_{e}^{-1}\left(\mathbf{x}\right)\right)\left.\chi\right|_{+}\left(\mathbf{x}\right)\right]\right),
\end{equation}
where subscript $-$ (resp. $+$) denotes the interior (resp. exterior) element.
\end{thm}

\begin{proof}
From the definition of $\mathbf{r^{e}}$ in~\eqref{eq:re int def}, we have
\begin{equation}
\begin{array}{lll}
\stretchint{5ex}_{\bs\bs\Omega}\mathbf{r^{e}}\left(\mathbf{\llbracket u \rrbracket}\right)\cdot\bm{\Phi}\,\mbox{d}\Omega&=&\stretchint{5ex}_{\bs\bs\Omega_{e}}\mathbf{r^{e}}\left(\mathbf{\llbracket u \rrbracket}\right)\cdot\bm{\Phi}\,\mbox{d}\Omega_{e}\vspace{0.2cm}\\
&=&-\stretchint{5ex}_{\bs \bs \Gamma_{e}}\mathbf{\llbracket u \rrbracket}\cdot\left\lbrace\left\lbrace\bm{\Phi}\right\rbrace\right\rbrace\,\mbox{d}\Gamma_{e}\vspace{0.2cm}\\
&=&-\dfrac{1}{2}\left[\stretchint{5ex}_{\bs \bs \Gamma_{e}}\mathbf{\llbracket u \rrbracket}\cdot\left.\bm{\Phi}\right|_{-}\,\mbox{d}\Gamma_{e}+\stretchint{5ex}_{\bs \bs \Gamma_{e}}\mathbf{\llbracket u \rrbracket}\cdot\left.\bm{\Phi}\right|_{+}\,\mbox{d}\Gamma_{e}\right]\vspace{0.2cm}\\
&=&-\dfrac{\left|J^{e}\right|}{2}\left[\stretchint{5ex}_{\bs \bs \Gamma_{s,e}}\mathbf{\llbracket u \rrbracket}\cdot\left.\bm{\Phi}\right|_{-}\,\mbox{d}\Gamma_{s,e}+\stretchint{5ex}_{\bs \bs \Gamma_{s,e}}\mathbf{\llbracket u \rrbracket}\cdot\left.\bm{\Phi}\right|_{+}\,\mbox{d}\Gamma_{s,e}\right].
\end{array}
\end{equation}

Both of these terms are integrals of a polynomial of degree less or equal to $2p$. Using Gauss-Legendre quadratures ($\left(\mathbf{r}_{i}\right)_{i\in\llbracket 1,N_{fp}\rrbracket}$ represents the nodes and $\left(\omega_{i}\right)_{i\in\llbracket 1,N_{fp}\rrbracket}$ represents the weights) we have,
\begin{equation}
\begin{array}{lll}
\stretchint{5ex}_{\bs\bs\Omega_{e}}\mathbf{r^{e}}\left(\mathbf{\llbracket u \rrbracket}\right)\cdot\bm{\Phi}\,\mbox{d}\Omega_{e}&=&-\dfrac{1}{2}\left|J^{e}\right|\left[\mathlarger{\sum}_{j=1}^{N_{fp}}\mathbf{\llbracket u \rrbracket}_{ej}\cdot\left(\left.\bm{\Phi}\right|_{-}\left(\mathbf{r}_{j}^{e}\right)\omega_{j}+\left.\bm{\Phi}\right|_{+}\left(\mathbf{r}_{j}^{e}\right)\omega_{j}\right)\right].
\end{array}
\end{equation}
Applying the result of Lemma~\ref{lem:equivalence DG FR}, we obtain
\begin{equation}
\begin{array}{lll}
\stretchint{5ex}_{\bs\bs\Omega_{e}}\mathbf{r^{e}}\left(\mathbf{\llbracket u \rrbracket}\right)\cdot\bm{\Phi}\,\mbox{d}\Omega_{e}&=&-\dfrac{1}{2}\left|J^{e}\right|\left[\mathlarger{\sum}_{j=1}^{N_{fp}}\mathbf{\llbracket u\rrbracket}_{ej}\cdot\left(\stretchint{4ex}_{\bs\bs\Omega_{s}}\phi_{ej,-}^{DG}\left.\bm{\Phi}\right|_{-}\,\mbox{d}\Omega_{s}+\stretchint{4ex}_{\bs\bs\Omega_{s}}\phi_{ej,+}^{DG}\left.\bm{\Phi}\right|_{+}\,\mbox{d}\Omega_{s}\right)\right]\\
&=&-\mathlarger{\sum}_{j=1}^{N_{fp}}\mathbf{\llbracket u \rrbracket}_{ej}\cdot\left(\stretchint{5ex}_{\bs\bs\Omega_{-}}\left(\dfrac{F_{s,-}^{e}}{2}\phi_{ej,-}^{DG}\left.\bm{\Phi}\right|_{-}\right)\,\mbox{d}\Omega_{-}+\stretchint{5ex}_{\bs\bs\Omega_{+}}\left(\dfrac{F_{s,+}^{e}}{2}\phi_{ej,+}^{DG}\left.\bm{\Phi}\right|_{+}\right)\,\mbox{d}\Omega_{+}\right)\\
&=&-\stretchint{7ex}_{\bs\bs\Omega_{e}}\left[\mathlarger{\sum}_{j=1}^{N_{fp}}\left(\dfrac{F_{s,-}^{e}}{2}\phi_{ej,-}^{DG}\mathbf{\llbracket u \rrbracket}_{ej}\cdot\bm{\Phi}\left.\chi\right|_{-}+\dfrac{F_{s,+}^{e}}{2}\phi_{ej,+}^{DG}\mathbf{\llbracket u \rrbracket}_{ej}\cdot\bm{\Phi}\left.\chi\right|_{+}\right)\right]\,\mbox{d}\Omega_{e}\\
&=&-\stretchint{7ex}_{\bs\bs\Omega_{e}}\left(\mathlarger{\sum}_{j=1}^{N_{fp}}\mathbf{\llbracket u \rrbracket}_{ej}\left[\dfrac{F_{s,-}^{e}}{2}\phi_{ej,-}^{DG}\left(\mathcal{M}_{e}^{-1}\left(\mathbf{x}\right)\right)\left.\chi\right|_{-}\left(\mathbf{x}\right)+\dfrac{F_{s,+}^{e}}{2}\phi_{ej,+}^{DG}\left(\mathcal{M}_{e}^{-1}\left(\mathbf{x}\right)\right)\left.\chi\right|_{+}\left(\mathbf{x}\right)\right]\right)\cdot\bm{\Phi}\,\mbox{d}\Omega_{e}
\end{array}
\end{equation}
Shifting the right-hand side of the previous equation into the left, we can gather the various terms under the integrals and factor out the test function $\bm{\Phi}$. We finally obtain expression~\eqref{eq:formula re}.
\end{proof}

\begin{cor}\label{cor:formula theta e2}
Employing Gauss-Legendre quadratures, the term $\Theta_{e,2}$ defined in equation~\eqref{eq:theta dif BR2} can be computed exactly as,
\begin{equation}\label{eq: formula theta e2}
\Theta_{e,2}=-\left|J^{e}\right|\mathlarger{\sum}_{i=1}^{N_{fp}}\llbracket u\rrbracket_{ei}s\omega_{i}\left[\mathlarger{\sum}_{j=1}^{N_{fp}}\llbracket u\rrbracket_{ej}\left(\dfrac{F_{s,-}^{e}}{4}\psi_{ej,-}^{DG}\left(\mathbf{r}_{i}^{e}\right)+\dfrac{F_{s,+}^{e}}{4}\psi_{ej,+}^{DG}\left(\mathbf{r}_{i}^{e}\right)\right)\right].
\end{equation}
\end{cor}

\begin{proof}
We apply the result of Theorem~\ref{thm:formula re} on equation~\eqref{eq:theta dif BR2} and then transform the integral from the physical domain to the computational domain,
\begin{equation}
\begin{array}{lll}
\Theta_{e,2}&=&\stretchint{5ex}_{\bs \bs \Gamma_{e}}s\{\{\mathbf{r^{e}}\left(\mathbf{\llbracket u\rrbracket}\right)\}\}\cdot \mathbf{\llbracket u\rrbracket} \mbox{d}\Gamma_{e}\\
&=&-\left|J^{e}\right|\Bigg[\stretchint{5ex}_{\bs \bs \Gamma_{s,e}}\left(s\mathbf{\llbracket u \rrbracket}\cdot\left(\mathlarger{\sum}_{j=1}^{N_{fp}}\mathbf{\llbracket u \rrbracket}_{ej}\dfrac{F_{s,-}^{e}}{4}\phi_{ej,-}^{DG}\left(\mathbf{r}\right)\right)\right)\,\mbox{d}\Gamma_{s,e}\\
&&+\stretchint{5ex}_{\bs\bs\Gamma_{s,e}}\left(s\mathbf{\llbracket u \rrbracket}\cdot\left(\mathlarger{\sum}_{j=1}^{N_{fp}}\mathbf{\llbracket u \rrbracket}_{ej}\dfrac{F_{s,+}^{e}}{4}\phi_{ej,+}^{DG}\left(\mathbf{r}\right)\right)\right)\,\mbox{d}\Gamma_{s,e}\Bigg].
\end{array}
\end{equation}

Both $\mathbf{\llbracket u \rrbracket}$ and $\phi_{ej}^{DG}$ are polynomials of degree less or equal to $p$. We compute exactly the integrals using Gauss-Legendre quadratures to obtain,

\begin{equation}
\begin{array}{lll}
\Theta_{e,2}&=&-\left|J^{e}\right|\Bigg[\mathlarger{\sum}_{i=1}^{N_{fp}}\left(s\llbracket u \rrbracket_{ei}\omega_{i}\left(\mathlarger{\sum}_{j=1}^{N_{fp}}\llbracket u \rrbracket_{ej}\dfrac{F_{s,-}^{e}}{4}\phi_{ej,-}^{DG}\left(\mathbf{r}_{i}^{e})\right)\right)\right)\\
&&+\mathlarger{\sum}_{i=1}^{N_{fp}}\left(s\llbracket u \rrbracket_{ei}\omega_{i}\left(\mathlarger{\sum}_{j=1}^{N_{fp}}\llbracket u \rrbracket_{ej}\dfrac{F_{s,+}^{e}}{4}\phi_{ej,+}^{DG}\left(\mathbf{r}_{i}^{e})\right)\right)\right)\Bigg]\\
&=&-\left|J^{e}\right|\mathlarger{\sum}_{i=1}^{N_{fp}}\llbracket u\rrbracket_{ei}s\omega_{i}\left[\mathlarger{\sum}_{j=1}^{N_{fp}}\llbracket u\rrbracket_{ej}\left(\dfrac{F_{s,-}^{e}}{4}\psi_{ej,-}^{DG}\left(\mathbf{r}_{i}^{e}\right)+\dfrac{F_{s,+}^{e}}{4}\psi_{ej,+}^{DG}\left(\mathbf{r}_{i}^{e}\right)\right)\right].
\end{array}
\end{equation}
\end{proof}

\begin{thm}\label{thm:tau BR2}
Employing the BR2 scheme for the diffusion equation with the ESFR methods, for all edges $e$ and for all flux points $i$, $s_{ei}$ greater than $s_{ei}^{*}$ implies the energy stability of the solution, with
\begin{equation}\label{eq:criterion BR2 triangle}
\begin{array}{lll}
s_{ei}^{*}&=&\min\limits_{\kappa}\mathlarger{\sum}_{k}\left(\dfrac{\psi_{ei,k}\left(\mathbf{r}_{i}^{e}\right)-\left|\psi_{ei,k}\left(\mathbf{r}_{i}^{e}\right)\right|+\mathlarger{\sum}_{f=1}^{3}\dfrac{F_{s,k}^{f}\left|\mathbf{n}\cdot\mathbf{n}^{f}\right|}{2F_{s,k}^{e}}\left(\mathlarger{\sum}_{j=1}^{N_{fp}}\left(\left|\psi_{fj,k}\left(\mathbf{r}_{i}^{e}\right)\right|+\dfrac{\omega_{j}}{\omega_{i}}\left|\psi_{ei,k}\left(\mathbf{r}_{j}^{f}\right)\right|\right)\right)}{\psi_{ei,k}^{DG}\left(\mathbf{r}_{i}^{e}\right)+\left|\psi_{ei,k}^{DG}\left(\mathbf{r}_{i}^{e}\right)\right|-\dfrac{1}{2}\mathlarger{\sum}_{j=1}^{N_{fp}}\left(\left|\psi_{ej,k}^{DG}\left(\mathbf{r}_{i}^{e}\right)\right|+\dfrac{\omega_{j}}{\omega_{i}}\left|\psi_{ei,k}^{DG}\left(\mathbf{r}_{j}^{e}\right)\right|\right)}\right).
\end{array}
\end{equation}
where $k=\left\lbrace -,+\right\rbrace$, $k=-$ signifies interior to triangle $IJK$, and $k=+$ denotes triangle $IJK_{+1}$.
\end{thm}

\begin{proof}
The expression of $\Theta_{e2}$ in Corollary~\ref{cor:formula theta e2} is similar to the second part of the first term in~\eqref{eq:theta 1 similar}. Similarly, we apply the triangular inequality,
\begin{equation}
\begin{array}{lll}
\Theta_{e,2}&\leq&\left|J^{e}\right|\dfrac{F_{s,-}^{e}}{4}\mathlarger{\sum}_{i=1}^{N_{fp}}\left[\llbracket u\rrbracket_{ei}^{2}\left(-s_{ei}\omega_{i}\psi_{ei,-}^{DG}\left(\mathbf{r}_{i}^{e}\right)+\dfrac{\left|s_{ei}\right|}{2}\mathlarger{\sum}_{\substack{j=1\\j\neq i}}^{N_{fp}}\left(\omega_{i}\left|\psi_{ej,-}^{DG}\left(\mathbf{r}_{i}^{e}\right)\right|+\omega_{j}\left|\psi_{ei,-}^{DG}\left(\mathbf{r}_{j}^{e}\right)\right|\right)\right)\right]\\
&&+\left|J^{e}\right|\dfrac{F_{s,+}^{e}}{4}\mathlarger{\sum}_{i=1}^{N_{fp}}\left[\llbracket u\rrbracket_{ei}^{2}\left(-s_{ei}\omega_{i}\psi_{ei,+}^{DG}\left(\mathbf{r}_{i}^{e}\right)+\dfrac{\left|s_{ei}\right|}{2}\mathlarger{\sum}_{\substack{j=1\\j\neq i}}^{N_{fp}}\left(\omega_{i}\left|\psi_{ej,+}^{DG}\left(\mathbf{r}_{i}^{e}\right)\right|+\omega_{j}\left|\psi_{ei,+}^{DG}\left(\mathbf{r}_{j}^{e}\right)\right|\right)\right)\right],
\end{array}
\end{equation}
where the last term of the two lines are obtained via similar derivations performed in the IP section (from equation~\eqref{eq:triang derivations ini} up to~\eqref{eq:triang derivations fin}).
We then combine the previous inequality with inequality~\eqref{eq:final expression theta dif,1}. We can further simplify the expression by removing the exclusion of $j\neq i$ ( and retrieve an additional term$-\left|s_{ei}\right|\left|\psi_{ei,k}^{DG}\left(\mathbf{r}_{i}^{e}\right)\right|$) and introduce the parameter $k=\left\lbrace-,+\right\rbrace$ to signify the summation of the terms across the edge to yield,
\begin{equation}
\begin{array}{lll}
\Theta_{dif}&\leq&\dfrac{\left|J^{e}\right|}{4}\mathlarger{\mathlarger{\sum}}_{e=1}^{N_{e}}\mathlarger{\mathlarger{\sum}}_{i=1}^{N_{fp}}\Bigg[\llbracket u\rrbracket_{ei}^{2}\Bigg(\mathlarger{\sum}_{k}\bigg[F_{s,k}^{e}\omega_{i}\left(\psi_{ei,k}\left(\mathbf{r}_{i}^{e}\right)-\left|\psi_{ei,k}\left(\mathbf{r}_{i}^{e}\right)\right|-s_{ei}\psi_{ei,k}^{DG}\left(\mathbf{r}_{i}^{e}\right)-\left|s_{ei}\right|\left|\psi_{ei,k}^{DG}\left(\mathbf{r}_{i}^{e}\right)\right|\right)\\
&&+\dfrac{F_{s,k}^{e}\left|s_{ei}\right|}{2}\mathlarger{\sum}_{j=1}^{N_{fp}}\left(\omega_{i}\left|\psi_{ej,k}^{DG}\left(\mathbf{r}_{i}^{e}\right)\right|+\omega_{j}\left|\psi_{ei,k}^{DG}\left(\mathbf{r}_{j}^{e}\right)\right|\right)\\
&&+\mathlarger{\sum}_{f=1}^{3}\mathlarger{\sum}_{j=1}^{N_{fp}}\left(\dfrac{F_{s,k}^{f}\left|\mathbf{n}\cdot\mathbf{n}_{f}\right|}{2}\left(\omega_{i}\left|\psi_{fj,k}\left(\mathbf{r}_{i}^{e}\right)\right|+\left|\psi_{ei,k}\left(\mathbf{r}_{j}^{f}\right)\right|\right)\right)\bigg]\Bigg)\Bigg].
\end{array}
\end{equation}

One minor problem from the previous equation is that we are unable to factor out the parameter $s$ as some of the terms are multiplied by $\left|s\right|$. This can be resolved by making the assumption that the parameter $s$ is non-negative, $s\geq 0$. This assumption does not pose a problem for the stability proof as: if $s$ was negative then by taking it to be positive only adds more dissipation and ensures the stability of the scheme. 

In order to have $\Theta_{dif}\leq 0$, we require $s_{ei}\geq s_{ei}^{*} \, \forall e\in\llbracket 1,N_{e}\rrbracket\,\forall i\in\llbracket 1,N_{fp}\rrbracket$, where $s_{ei}^{*}$ is defined as
\begin{equation}
\begin{array}{lll}
s_{ei}^{*}&=&\mathlarger{\sum}_{k}\left(\dfrac{\psi_{ei,k}\left(\mathbf{r}_{i}^{e}\right)-\left|\psi_{ei,k}\left(\mathbf{r}_{i}^{e}\right)\right|+\mathlarger{\sum}_{f=1}^{3}\dfrac{F_{s,k}^{f}\left|\mathbf{n}\cdot\mathbf{n}^{f}\right|}{2F_{s,k}^{e}}\mathlarger{\sum}_{j=1}^{N_{fp}}\left(\left|\psi_{fj,k}\left(\mathbf{r}_{i}^{e}\right)\right|+\dfrac{\omega_{j}}{\omega_{i}}\left|\psi_{ei,k}\left(\mathbf{r}_{j}^{f}\right)\right|\right)}{\psi_{ei,k}^{DG}\left(\mathbf{r}_{i}^{e}\right)+\left|\psi_{ei,k}^{DG}\left(\mathbf{r}_{i}^{e}\right)\right|-\dfrac{1}{2}\mathlarger{\sum}_{j=1}^{N_{fp}}\left(\left|\psi_{ej,k}^{DG}\left(\mathbf{r}_{i}^{e}\right)\right|+\dfrac{\omega_{j}}{\omega_{i}}\left|\psi_{ei,k}^{DG}\left(\mathbf{r}_{j}^{e}\right)\right|\right)}\right).
\end{array}
\end{equation}

We showed through Postulate~\ref{pos:independent} that the problem is independent of $\kappa$. As a consequence, the energy stability is also independent of $\kappa$. Minimizing the previous equation results in criterion~\eqref{eq:criterion BR2 triangle}.

\end{proof}

\begin{rmk}
The minimization of $s^{*}$ with respect to $\kappa$ yields similar results to Figure~\ref{fig:plot tau minimization} and are hence omitted. In the rest of the article, $s^{*}_{theory}$ is computed with equation~\eqref{eq:criterion BR2 triangle} with $\kappa_{+}$.
\end{rmk}

\subsection{Numerical results}\label{sec:Numerical results BR2 2D}

In order to validate the previous result, we conduct numerical simulations to find the minimum penalty term $s_{numerical}^{*}$ that ensures stability. We consider the same problem and parameters as in Section~\ref{sec:Numerical results IP 2D}.

\begin{table}[H]
\centering
\setlength\tabcolsep{4pt}
\begin{minipage}{0.45\textwidth}
\centering
\resizebox{\textwidth}{!}{
\begin{tabular}{|c||c|c||c|c|}
\hline
$p$ & \multicolumn{2}{c||}{2} & \multicolumn{2}{c|}{3}\\
\hline
\bf \backslashbox{$c$}{$\kappa$} & $\kappa_{DG}$ & $\kappa_{+}$ & $\kappa_{DG}$ & $\kappa_{+}$\\
\hline
$\max\left(\tau_{theory}^{*}\right)$ & 1.82 & 1.82 & 2.51 & 2.51\\
\hline
\hline
$c_{DG}$ & 0.67 & 0.67 & 0.76 & 0.76\\
\hline
$c_{+}$ & 0.53 & 0.53 & 0.73 & 0.73\\
\hline
\end{tabular}
}
\caption{$s_{numerical}^{*}$ for the BR2 scheme for $p=2$ and $p=3$ for a $8\times 8\times 2$ mesh.}
\label{tab:result BR2 8 elem 2D trian}
\end{minipage}%
\hfill
\begin{minipage}{0.45\textwidth}
\centering
\resizebox{\textwidth}{!}{
\begin{tabular}{|c||c|c||c|c|}
\hline
$p$ & \multicolumn{2}{c||}{2} & \multicolumn{2}{c|}{3}\\
\hline
\bf \backslashbox{$c$}{$\kappa$} & $\kappa_{DG}$ & $\kappa_{+}$ & $\kappa_{DG}$ & $\kappa_{+}$\\
\hline
$\max\left(s_{theory}^{*}\right)$ & 1.82 & 1.82 & 2.51 & 2.51\\
\hline
\hline
$c_{DG}$ & 0.73 & 0.73 & 0.78 & 0.78\\
\hline
$c_{+}$ & 0.68 & 0.68 & 0.76 & 0.76\\
\hline
\end{tabular}
}
\caption{$s_{numerical}^{*}$ for the BR2 scheme for $p=2$ and $p=3$ for a $16\times 16\times 2$ mesh.}
\label{tab:result BR2 16 elem triangular 2D}
\end{minipage}
\end{table}

Both tables \ref{tab:result BR2 8 elem 2D trian} and \ref{tab:result BR2 16 elem triangular 2D} guarantee that $\max\left(s_{theory}\right)\geq s_{numerical}^{*}$. Similar to the IP results, these tables do not validate the criterion of~\eqref{eq:criterion BR2 triangle} for every edge and flux point but the results do not contradict our criterion.
\section{Von Neumann analysis}\label{sec:VN section}
This section presents a von Neumann analysis to study the maximal time step, $\Delta t_{max}$ of the different schemes.
\subsection{Maximal time step}
ESFR schemes offer a range of methods where the values of both $c$ and $\kappa$ dictate the amount of filtering or a relaxation of the highest modes of the DG correction fields~\cite{zwanenburg_equivalence_2016}. The purpose of this section is to present the methods with the highest time step. Castonguay et al.~\cite{castonguay_new_2012} performed the analysis for the advection equation and studied the influence of $c$ (Table~\ref{tab:c+ triangle}).  From section~\ref{sec:independent kappa}, we demonstrated that the problem is independent of $\kappa$. In this section, we will confirm this theoretical result and study the influence of $c$ and the penalty term ($\tau$ and $s$) of the numerical schemes. Let us consider a 2D periodic pattern, controled by the angle $\gamma$,  which forms the domain $\Omega$.

\begin{figure}[H]
\centering
\includegraphics[width=5in]{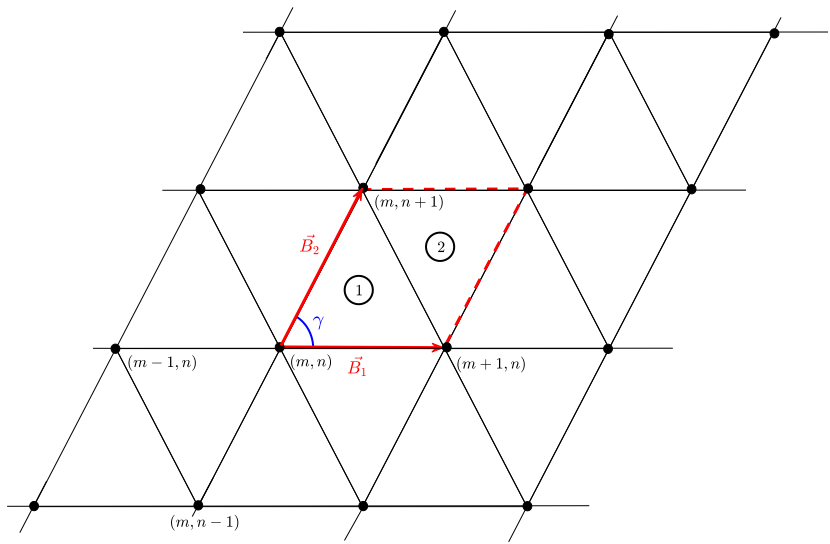}
\caption{Pattern used for the von Neumann analysis.}
\label{fig:pattern VN}
\end{figure}

Referring to Figure~\ref{fig:pattern VN}, the pattern is formed by two triangles, themselves parametrized by two vectors, $\vec{B_{1}}$ and $\vec{B_{2}}$. In the study considered, we chose $\vec{B_{1}}=\begin{pmatrix}
\Delta_{B}\\0
\end{pmatrix}$ and $\vec{B_{2}}=\begin{pmatrix}
\Delta_{B}\cos\left(\gamma\right)\\ \Delta_{B}\sin\left(\gamma\right)
\end{pmatrix}$. As shown by Castonguay \cite{castonguay_new_2012}, the shape of the pattern has an influence on $\Delta t_{max}$ hence two values of $\gamma$ will be studied ($60\degree$ and $90\degree$). We then create the column vector solution of the pattern element $\bar{u}_{mn}=\begin{pmatrix}
\bar{u}_{1}\\\bar{u}_{2}
\end{pmatrix}$, where $\bar{u}_{1}$ contains the discrete value on the $N_{p}$ nodes of the solution on the first triangle and $\bar{u}_{2}$ contains the discrete values of the solution on the second triangle, represented in Figure~\ref{fig:pattern VN}. Hence $\bar{u}_{mn}$ is a $2N_{p}\times 1$ vector. The solution points choosen for the analysis are the $\alpha$-optimized points.

We then nondimensionalize the partial differential equation~\eqref{eq:diffusion 2D} by introducing $\hat{x}=\frac{x}{\Delta_{B}}$, $\hat{y}=\frac{y}{\Delta_{B}}$, $\hat{t}=\frac{t\, b}{\left(\Delta_{B}\right)^{2}}$, 
\begin{equation}\label{eq:Nondim diffusion equation triangle}
\dfrac{\partial u}{\partial \hat{t}}=\left(\dfrac{\partial^{2}u}{\partial \hat{x}^{2}}+\dfrac{\partial^{2}u}{\partial \hat{y}^{2}}\right).
\end{equation}
Equation \eqref{eq:diffusion 2D} is recovered by taking $b=1$ and a unitary element size $\Delta_{B}=1$. We now rename the variable $\hat{x}$ and $\hat{y}$ by $x$ and $y$. We apply the FR procedure with the IP and BR2 numerical fluxes. Both schemes have compact stencils, where the solution requires information only from its closest neighbours. We write~\eqref{eq:Nondim diffusion equation triangle} in discrete form as
\begin{equation}
\dfrac{\mbox{d} \bar{u}_{mn}}{\mbox{d}t}=\left(\mathbf{A}\bar{u}_{m,n}+\mathbf{B}\bar{u}_{m,n+1}+\mathbf{C}\bar{u}_{m,n-1}+\mathbf{D}\bar{u}_{m-1,n}+\mathbf{E}\bar{u}_{m+1,n}\right),
\end{equation}
where, $\mathbf{A}$ corresponds to a matrix of size  $2N_{p}\times 2N_{p}$ taking $\bar{u}_{m,n}$ as argument to compute the Laplacian of $u_{n}$, and the same can be said of matrices $\mathbf{B}$, $\mathbf{C}$, $\mathbf{D}$ and $\mathbf{E}$ by taking the respective neighbours of the pattern $\left(m,n\right)$ as argument. Lowercase letter with a bar, $\bar{a}$ indicates a column vector containing the discrete values of a scalar quantity. Assuming a Bloch-wave solution~\cite{AINSWORTH2004106}, we have 
\begin{equation}
\bar{u}_{m,n}=e^{\iu\left|k\right|\left(x_{mn}\cos\left(\theta\right)+y_{m,n}\sin\left(\theta\right)\right)-\omega t}\bar{v},
\end{equation}
where $\vec{k}=\left|k\right|\begin{pmatrix}
\cos\left(\theta\right)\\ \sin\left(\theta\right)
\end{pmatrix}$ represents the prescribed wave vector, both $\left|k\right|$ and $\theta$ vary between $\left[0,2\pi\right]$, $\omega^{\delta}$ is the discrete frequency and $\bar{v}$ is a vector independent of the elements. The solution is periodic and hence we obtain,
\begin{equation}
\begin{array}{lll}
\bar{u}_{m,n+1}&=&e^{\iu\left|k\right|\left(\left(x_{mn}+\Delta_{B}\cos\left(\gamma\right)\right)\cos\left(\theta\right)+\left(y_{m,n}+\Delta_{B}\sin\left(\gamma\right)\right)\sin\left(\theta\right)\right)-\omega t}\bar{v}\\
&=&\bar{u}_{m,n}e^{\iu\left|k\right|\Delta_{B}\left(\cos\left(\gamma-\theta\right)\right)},
\end{array}
\end{equation}
and the quantities $\bar{u}_{m,n-1}$, $\bar{u}_{m-1,n}$ and $\bar{u}_{m+1,n}$ are calculated similarly. We finally obtain
\begin{equation}
\dfrac{\mbox{d} \bar{u}_{mn}}{\mbox{d}t}=\mathbf{S}\left(\left|k\right|,\theta\right)\bar{u}_{m,n},
\end{equation}
where $\mathbf{S}=\left(\mathbf{A}+\mathbf{B}e^{\iu\left|k\right|\Delta_{B}\cos\left(\gamma-\theta\right)}+\mathbf{C}e^{-\iu\left|k\right|\Delta_{B}\cos\left(\gamma-\theta\right)}+\mathbf{D}e^{-\iu\left|k\right|\Delta_{B}\cos\left(\theta\right)}+\mathbf{E}e^{\iu\left|k\right|\Delta_{B}\cos\left(\theta\right)}\right)$. 
The quantities $\omega^{\delta}$ and $\bar{v}$ can be computed from the eigenvalues and eigenvectors of $\mathbf{S}$.

We employ the fourth order five stage Runge-Kutta (RK54) method~\cite{carpenter_fourth-order_1994}. The column-vector of the solution at time $n+1$, $\bar{u}_{mn}^{t_{n+1}}$, can be expressed as,
\begin{equation}
\bar{u}_{mn}^{t_{n+1}}=\mathbf{M}\left(\left|k\right|,\theta,\Delta t\right)\bar{u}_{m,n}^{t_{n}},
\end{equation}
where, $\mathbf{M}$ is defined as
\begin{equation}
\mathbf{M}\left(k,\Delta t\right)=1+\Delta t\mathbf{S}\left(k\right)+\dfrac{1}{2!}\left(\Delta t\mathbf{S}\left(k\right)\right)^{2}+\dfrac{1}{3!}\left(\Delta t\mathbf{S}\left(k\right)\right)^{3}+\dfrac{1}{4!}\left(\Delta t\mathbf{S}\left(k\right)\right)^{4}+\dfrac{1}{200}\left(\Delta t\mathbf{S}\left(k\right)\right)^{5}.
\end{equation}

The matrix $\mathbf{M}$ depends on $\left|k\right|$, $\theta$ and $\Delta t$. Whereas $\left|k\right|$ and $\theta$ vary between $\left[0,2\pi\right]$, $\Delta t$ is choosen to ensure the stability of the scheme, i.e the moduli of the spectral radius of $\mathbf{M}$ must be less than 1. The algorithm resumes to: start at an initial $\Delta t_{0}$ sufficiently high to produce an unstable solution, then scan over the range of $\left|k\right|$ and $\theta$ to compute the highest eigenvalue: $\left|\lambda\right|_{max}$, decrease $\Delta t$ until $\left|\lambda\right|_{max}\leq1$. We apply the analysis for both the IP and BR2 schemes, where the value of $c$ is taken to be equal to $c_{DG}$ and $c_{+}$ (Table~\ref{tab:c+ triangle}), while two values of the penalty term are taken (either equal or 1.5 times the criterion found:~\eqref{eq:criterion IP 2D triangle} for the IP scheme and~\eqref{eq:criterion BR2 triangle} for the BR2 scheme).

\begin{table}[H]
\centering
\setlength\tabcolsep{4pt}
\begin{minipage}{0.45\textwidth}
\centering
\resizebox{\textwidth}{!}{
\begin{tabular}{|c|c||c|c||c|c|}
\hline
\multicolumn{2}{|c||}{$p$} &\multicolumn{2}{c||}{2} &\multicolumn{2}{c|}{3} \\
\hline
$c$ &$\kappa$ &$\tau_{theory}$ &$1.5\tau_{theory}$ &$\tau_{theory}$ &$1.5\tau_{theory}$ \\
\hline
\multirow{2}{*}{$c_{DG}$} &$\kappa_{DG}$&1.85e-02 &1.12e-02 &6.23e-03 &3.85e-03 \\
\cline{2-6}
&$\kappa_{+}$&1.85e-02 &1.12e-02 &6.23e-03 &3.85e-03 \\
\hline
\multirow{2}{*}{$c_{+}$} &$\kappa_{DG}$&2.74e-02 &1.76e-02 &8.22e-03 &5.20e-03 \\
\cline{2-6}
&$\kappa_{+}$&2.74e-02 &1.76e-02 &8.22e-03 &5.20e-03 \\
\hline
\end{tabular}
}
\caption{$\Delta t_{max}$ for triangles for the IP scheme for $\gamma=60\degree$.}
\label{tab:dt max triangle IP 60}

\end{minipage}%
\hfill
\begin{minipage}{0.45\textwidth}
\centering
\resizebox{\textwidth}{!}{
\begin{tabular}{|c|c||c|c||c|c|}
\hline
\multicolumn{2}{|c||}{$p$} &\multicolumn{2}{c||}{2} &\multicolumn{2}{c|}{3} \\
\hline
$c$ &$\kappa$ &$\tau_{theory}$ &$1.5\tau_{theory}$ &$\tau_{theory}$ &$1.5\tau_{theory}$ \\
\hline
\multirow{2}{*}{$c_{DG}$} &$\kappa_{DG}$&1.77e-02 &1.08e-02 &5.82e-03 &3.59e-03 \\
\cline{2-6}
&$\kappa_{+}$&1.77e-02 &1.08e-02 &5.82e-03 &3.59e-03 \\
\hline
\multirow{2}{*}{$c_{+}$} &$\kappa_{DG}$&2.64e-02 &1.69e-02 &7.83e-03 &4.93e-03 \\
\cline{2-6}
&$\kappa_{+}$&2.64e-02 &1.69e-02 &7.83e-03 &4.93e-03 \\
\hline
\end{tabular}
}
\caption{$\Delta t_{max}$ for triangles for the IP scheme for $\gamma=90\degree$.}
\label{tab:dt max triangle IP 90} 
\end{minipage}
\end{table}

\begin{table}[H]
\centering
\setlength\tabcolsep{4pt}
\begin{minipage}{0.45\textwidth}
\centering
\resizebox{\textwidth}{!}{
\begin{tabular}{|c|c||c|c||c|c|}
\hline
\multicolumn{2}{|c||}{$p$} &\multicolumn{2}{c||}{2} &\multicolumn{2}{c|}{3} \\
\hline
$c$ &$\kappa$ &$s_{theory}$ &$1.5 \,s_{theory}$ &$s_{theory}$ &$1.5 \,s_{theory}$ \\
\hline
\multirow{2}{*}{$c_{DG}$} &$\kappa_{DG}$&1.08e-02 &6.62e-03 &3.31e-03 &2.06e-03 \\
\cline{2-6}
&$\kappa_{+}$&1.08e-02 &6.62e-03 &3.31e-03 &2.06e-03 \\
\hline
\multirow{2}{*}{$c_{+}$} &$\kappa_{DG}$&1.56e-02 &1.02e-02 &4.13e-03 &2.66e-03 \\
\cline{2-6}
&$\kappa_{+}$&1.56e-02 &1.02e-02 &4.13e-03 &2.66e-03 \\
\hline
\end{tabular}
}
\caption{$\Delta t_{max}$ for triangles for the BR2 scheme for $\gamma=60\degree$.}
\label{tab:dt max triangle BR2 60}

\end{minipage}%
\hfill
\begin{minipage}{0.45\textwidth}
\centering
\resizebox{\textwidth}{!}{
\begin{tabular}{|c|c||c|c||c|c|}
\hline
\multicolumn{2}{|c||}{$p$} &\multicolumn{2}{c||}{2} &\multicolumn{2}{c|}{3} \\
\hline
$c$ &$\kappa$ &$s_{theory}$ &$1.5 \,s_{theory}$ &$s_{theory}$ &$1.5 \,s_{theory}$ \\
\hline
\multirow{2}{*}{$c_{DG}$} &$\kappa_{DG}$&1.00e-02 &6.09e-03 &2.81e-03 &1.72e-03 \\
\cline{2-6}
&$\kappa_{+}$&1.00e-02 &6.09e-03 &2.81e-03 &1.72e-03 \\
\hline
\multirow{2}{*}{$c_{+}$} &$\kappa_{DG}$&1.39e-02 &8.99e-03 &3.58e-03 &2.27e-03 \\
\cline{2-6}
&$\kappa_{+}$&1.39e-02 &8.99e-03 &3.58e-03 &2.27e-03 \\
\hline
\end{tabular}
}
\caption{$\Delta t_{max}$ for triangles for the BR2 scheme for $\gamma=90\degree$.}
\label{tab:dt max triangle BR2 90} 
\end{minipage}
\end{table}

Through all these tables, for both the IP and BR2 numerical fluxes, we observe, that the lower the value of the penalty term the higher the maximal time step. Conversely, increasing $c$ leads to higher maximal time steps. As expected, $\kappa$ has no influence on $\Delta t_{max}$. Comparing Tables~\ref{tab:dt max triangle IP 60} and \ref{tab:dt max triangle IP 90} for the IP scheme, we observe, the maximal time step is always higher for $\gamma=60\degree$ than $90\degree$. Indeed $\gamma=60\degree$ results in a domain where each element is an equilateral triangle. Therefore this domain is more regular than $\gamma=90\degree$ and it results in an increase of the maximal time step. Equivalent trends are observed for the BR2 scheme as shown in Tables \ref{tab:dt max triangle BR2 60} and \ref{tab:dt max triangle BR2 90}.

Comparing the BR2 and the IP schemes, we observe that the latter has the higher time step for every case. A likely explanation is that the criterion for the IP scheme~\eqref{eq:criterion IP 2D triangle} is sharper than the one for the BR2 scheme~\eqref{eq:criterion BR2 triangle}.

As expected , $\Delta t_{max}$ does not depend on $\kappa$.

\begin{figure}[H] 
\centering

\begin{subfigure}{0.49\textwidth}
\centering
\includegraphics[width=\linewidth]{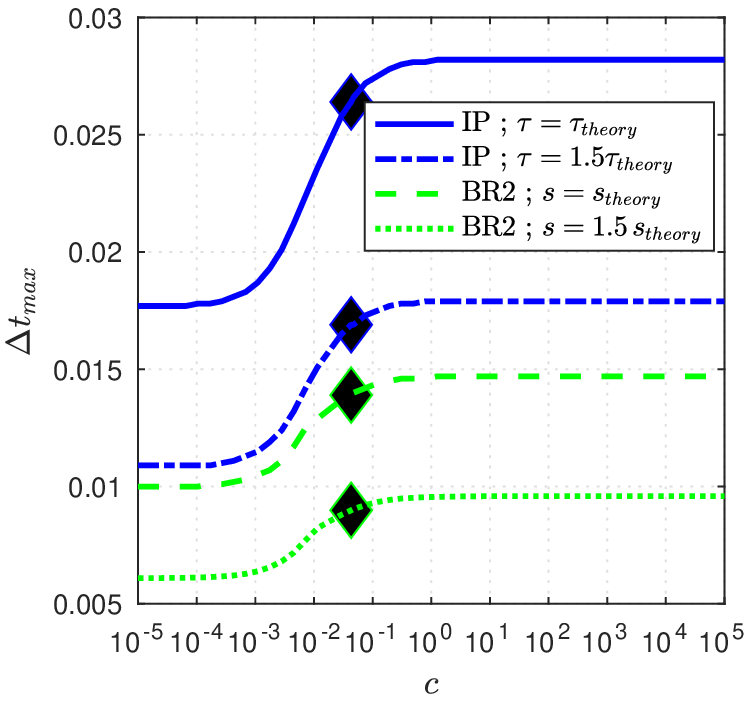}
\caption{$\Delta t_{max}$ for $p=2$.} \label{fig:dtmax DG p2 triangle}
\end{subfigure}\hspace*{\fill}
\begin{subfigure}{0.49\textwidth}
\centering
\includegraphics[width=\linewidth]{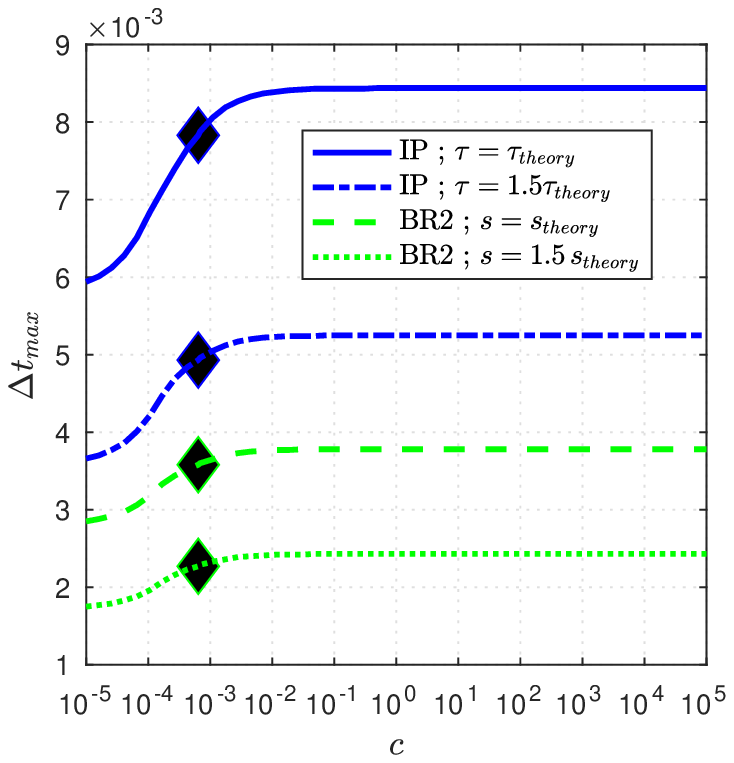}
\caption{$\Delta t_{max}$ for $p=3$.} \label{fig:dtmax DG p3 triangle}
\end{subfigure}
\caption{$\Delta t_{max}$ along $c$ for $\gamma=90\degree$ for $p$, the diamond marker represents $c=c_{+}$ ; using log scale the DG case couldn't be represented.}
\label{fig:plot dt max triangle}
\end{figure}

Figure~\ref{fig:plot dt max triangle} lays emphasis to the IP scheme providing a higher $\Delta t_{max}$ than the BR2 scheme. Moreover, we observe that by taking $c$ closer to $c_{+}$ provides for a higher time step. Therefore it would seem that $c = c_{+}$ for an advection-diffusion problem would lead to the highest $\Delta t_{max}$ possible.

The insights of the features of the ESFR method given by the von Neumann analysis is only true for regular meshes. However real CFD applications often use irregular grids due to geometric complexity and flow anisotropy. Thus we recommend against any generalization from this section.
\section{$L_{2}$ errors and order of accuracy}\label{sec:L2 errors section}

The criterion for both the IP and BR2 schemes have been mostly confirmed with the analysis in Sections \ref{sec:Numerical results IP 2D} and \ref{sec:Numerical results BR2 2D}. Our previous von Neumann analysis has enlightened methods which allow for a time step much higher than the classical DG method. In this section, we present the numerical $L_{2}$ errors and the order of convergence to verify that we get the expected OOA of $p+1$ for the different methods. The problem considered is the same as the one before. We consider a regular mesh (Figure~\ref{fig:regular mesh}) of size $N_{x}\times N_{x}\times 2$ where $N_{x}$ has been taken equal to 16, 32 and 64. We use two values for $\tau$ for the IP schemes: $\tau_{theory}$ given by~\eqref{eq:criterion IP 2D triangle} and $1.5 \tau_{theory}$ and two values for the BR2 schemes $s_{theory}$ given by~\eqref{eq:criterion BR2 triangle} and $1.5 s_{theory}$. The $L_{2}$-error is computed as 
\begin{equation}
L_{2}-\mbox{error}=\sqrt{\dfrac{\mathlarger{\sum}_{n=1}^{2N_{x}^{2}}\mathlarger{\sum}_{i=1}^{N_{p}}\left(u_{n,i}-u_{exact,n,i}\right)^{2}}{ 2N_{x}^{2} N_{p}}},
\end{equation}
where $u_{n,i}$ is the numerical solution evaluated on element $n$ at the solution point $i$ and $u_{exact,n,i}$ is the exact solution on the same element and solution point. The final time $t_{fin}$ was taken to be equal to 1.

The simulations were conducted for four methods: $c_{DG}/\kappa_{DG}$, $c_{DG}/\kappa_{+}$, $c_{+}/\kappa_{DG}$ and $c_{+}/\kappa_{+}$ where $c_{+}=\kappa_{+}$ is taken equal to the second column of Table~\ref{tab:c+ triangle}. We provided the maximal time step for $N_{x}=16$ elements. Maximal time steps were evaluated through an iterative approach while ensuring the solution remains bounded at $t=2$.

\begin{table}[H]
\centering
\resizebox{\textwidth}{!}{
\begin{tabular}{|c|c||c|c|c|ccc|c||c|c|c|ccc|c|}
\hline 
& &\multicolumn{7}{l||}{$\tau_{theory}$} &\multicolumn{7}{l|}{$1.5\tau_{theory}$} \\
\hline
$c$ &$\kappa$ &$N_{x}$=16 &$N_{x}$=32&$N_{x}=$64 &\multicolumn{3}{c|}{OOA} &$\Delta t_{max}$&$N_{x}$=16 &$N_{x}$=32&$N_{x}$=64 &\multicolumn{3}{c|}{OOA} &$\Delta t_{max}$ \\
\hline
\multirow{2}{*}{$c_{DG}$} &$\kappa_{DG}$&1.16e-04 &1.45e-05 &1.81e-06 &-&3.00 &3.00 &2.70e-03 &8.91e-05 &1.12e-05 &1.40e-06 &-&3.00 &3.00 &1.70e-03 \\
\cline{2-16}
&$\kappa_{+}$&1.16e-04 &1.45e-05 &1.81e-06 &-&3.00 &3.00 &2.70e-03 &8.91e-05 &1.12e-05 &1.40e-06 &-&3.00 &3.00 &1.70e-03 \\
\hline
\multirow{2}{*}{$c_{+}$} &$\kappa_{DG}$&1.26e-04 &1.48e-05 &1.82e-06 &-&3.09 &3.02 &4.10e-03 &9.63e-05 &1.14e-05 &1.40e-06 &-&3.08 &3.02 &2.60e-03 \\
\cline{2-16}
&$\kappa_{+}$&1.26e-04 &1.48e-05 &1.82e-06 &-&3.09 &3.02 &4.10e-03 &9.63e-05 &1.14e-05 &1.40e-06 &-&3.08 &3.02 &2.60e-03\\
\hline
\end{tabular}
}
\caption{$L_{2}$ errors using the IP scheme for $p=2$ for triangles.}
\label{L2 errors IP p2 tri}
\end{table}

\begin{table}[H]
\centering
\resizebox{\textwidth}{!}{
\begin{tabular}{|c|c||c|c|c|ccc|c||c|c|c|ccc|c|}
\hline 
& &\multicolumn{7}{l||}{$s_{theory}$} &\multicolumn{7}{l|}{$1.5s_{theory}$} \\
\hline
$c$ &$\kappa$ &$N_{x}$=16 &$N_{x}$=32&$N_{x}=$64 &\multicolumn{3}{c|}{OOA} &$\Delta t_{max}$&$N_{x}$=16 &$N_{x}$=32&$N_{x}$=64 &\multicolumn{3}{c|}{OOA} &$\Delta t_{max}$ \\
\hline
\multirow{2}{*}{$c_{DG}$} &$\kappa_{DG}$&9.57e-05 &1.20e-05 &1.50e-06 &-&2.99 &3.00 &1.50e-03 &8.22e-05 &1.03e-05 &1.28e-06 &-&3.00 &3.00 &9.50e-04 \\
\cline{2-16}
&$\kappa_{+}$&9.57e-05 &1.20e-05 &1.50e-06 &-&2.99 &3.00 &1.50e-03 &8.22e-05 &1.03e-05 &1.28e-06 &-&3.00 &3.00 &9.50e-04 \\
\hline
\multirow{2}{*}{$c_{+}$} &$\kappa_{DG}$&1.03e-04 &1.22e-05 &1.51e-06 &-&3.07 &3.02 &2.10e-03 &8.77e-05 &1.04e-05 &1.29e-06 &-&3.07 &3.02 &1.40e-03 \\
\cline{2-16}
&$\kappa_{+}$&1.03e-04 &1.22e-05 &1.51e-06 &-&3.07 &3.02 &2.10e-03 &8.77e-05 &1.04e-05 &1.29e-06 &-&3.07 &3.02 &1.40e-03\\
\hline
\end{tabular}
}
\caption{$L_{2}$ errors using the BR2 scheme for $p=2$ for triangles.}
\label{L2 errors BR2 p2 tri}
\end{table}

\begin{table}[H]
\centering
\resizebox{\textwidth}{!}{
\begin{tabular}{|c|c||c|c|c|ccc|c||c|c|c|ccc|c|}
\hline 
& &\multicolumn{7}{l||}{$\tau_{theory}$} &\multicolumn{7}{l|}{$1.5\tau_{theory}$} \\
\hline
$c$ &$\kappa$ &$N_{x}$=16 &$N_{x}$=32&$N_{x}=$64 &\multicolumn{3}{c|}{OOA} &$\Delta t_{max}$&$N_{x}$=16 &$N_{x}$=32&$N_{x}$=64 &\multicolumn{3}{c|}{OOA} &$\Delta t_{max}$ \\
\hline
\multirow{2}{*}{$c_{DG}$} &$\kappa_{DG}$&3.94e-06 &2.43e-07 &1.51e-08 &-&4.02 &4.01 &9.10e-04 &3.41e-06 &2.14e-07 &1.34e-08 &-&4.00 &4.00 &5.60e-04 \\
\cline{2-16}
&$\kappa_{+}$&3.94e-06 &2.43e-07 &1.51e-08 &-&4.02 &4.01 &9.10e-04 &3.41e-06 &2.14e-07 &1.34e-08 &-&4.00 &4.00 &5.60e-04 \\
\hline
\multirow{2}{*}{$c_{+}$} &$\kappa_{DG}$&3.92e-06 &2.42e-07 &1.51e-08 &-&4.02 &4.00 &1.20e-03 &3.45e-06 &2.14e-07 &1.34e-08 &-&4.01 &4.00 &7.70e-04 \\
\cline{2-16}
&$\kappa_{+}$&3.92e-06 &2.42e-07 &1.51e-08 &-&4.02 &4.00 &1.20e-03 &3.45e-06 &2.14e-07 &1.34e-08 &-&4.01 &4.00 &7.70e-04 \\
\hline
\end{tabular}
}
\caption{$L_{2}$ errors using the IP scheme for $p=3$ for triangles.}
\label{L2 errors IP p3 tri}
\end{table}

\begin{table}[H]
\centering
\resizebox{\textwidth}{!}{
\begin{tabular}{|c|c||c|c|c|ccc|c||c|c|c|ccc|c|}
\hline 
& &\multicolumn{7}{l||}{$s_{theory}$} &\multicolumn{7}{l|}{$1.5s_{theory}$} \\
\hline
$c$ &$\kappa$ &$N_{x}$=16 &$N_{x}$=32&$N_{x}=$64 &\multicolumn{3}{c|}{OOA} &$\Delta t_{max}$&$N_{x}$=16 &$N_{x}$=32&$N_{x}$=64 &\multicolumn{3}{c|}{OOA} &$\Delta t_{max}$ \\
\hline
\multirow{2}{*}{$c_{DG}$} &$\kappa_{DG}$&3.65e-06 &2.29e-07 &1.43e-08 &-&3.99 &4.00 &4.40e-04 &3.39e-06 &2.14e-07 &1.34e-08 &-&3.98 &4.00 &2.70e-04 \\
\cline{2-16}
&$\kappa_{+}$&3.65e-06 &2.29e-07 &1.43e-08 &-&3.99 &4.00 &4.40e-04 &3.39e-06 &2.14e-07 &1.34e-08 &-&3.98 &4.00 &2.70e-04 \\
\hline
\multirow{2}{*}{$c_{+}$} &$\kappa_{DG}$&3.71e-06 &2.30e-07 &1.44e-08 &-&4.01 &4.00 &5.60e-04 &3.47e-06 &2.16e-07 &1.35e-08 &-&4.01 &4.00 &3.50e-04 \\
\cline{2-16}
&$\kappa_{+}$&3.71e-06 &2.30e-07 &1.44e-08 &-&4.01 &4.00 &5.60e-04 &3.47e-06 &2.16e-07 &1.35e-08 &-&4.01 &4.00 &3.50e-04 \\
\hline
\end{tabular}
}
\caption{$L_{2}$ errors using the BR2 scheme for $p=3$ for triangles.}
\label{L2 errors BR2 p3 tri}
\end{table}

The maximal time step provided by the above tables can be compared with the von Neumann analysis: $\Delta t_{max}=\dfrac{\Delta \hat{t}_{max}}{b}\left(\dfrac{2}{N_{x}}\right)^{2}$. The maximum relative error is at $2.4\%$, which concurs with the maximal time steps obtained in this section. 

For both the IP and BR2 numerical fluxes and for both $p=2$ and $p=3$, we obtain the expected order of accuracy: $p+1$. Similarly, we observe that the IP scheme provides a higher time step than the BR2 method but the error from the IP scheme is also higher.

\section{Conclusion}
This article provides a theoretical proof of energy stability for the diffusion case for triangles using the IP and BR2 schemes. Bounds for the penalty term $\tau$ for the IP scheme and $s$ for the BR2 scheme to ensure stability for various ESFR schemes were obtained. These theoretical proofs were validated through numerical simulations and orders of accuracy were provided. It was established that for both the IP and BR2 numerical fluxes, the stability of the ESFR scheme is independent of the auxiliary correction field. A von-Neumann analysis was conducted to present methods which procure a higher maximal time step than the classical DG method. The trade-off is that the $L_{2}$-error increases for these methods. While the BR2 scheme procures the least amount of error, the IP scheme has the highest time step. Further analysis will be conducted to extend this proof for tetrahedra elements.
\section*{Acknowledgements}
We would like to acknowledge the final support of Natural Sciences and Engineering Research Council of Canada Discovery Grant Program and McGill University. We would also like to thank Philip Zwanenburg for helpful feedback.
\appendix
\section*{Appendix}
\section{Theoretical proof of the independecy of $\kappa$ for $p=1$}\label{sec:appendix independency kappa}
This section proposes a theoretical proof of Postulate~\ref{pos:independency of kappa} and hence of Postulate~\ref{pos:independent} for $p=1$. 
\begin{thm}
Let the solution of the diffusion equation be approximated by a polynomial of degree $p=1$ on the reference triangle. Let $\psi_{ei}$ be the correction field associated to face $e$ at the flux point $i$ parametrized by $\kappa$ and $\mathbf{\phi}_{fj}$ the correction field, associated to face $f$ at the flux point $j$, parametrized by $c$. Let $\left(\mathbf{r}_{f}^{j}\right)_{j\in\llbracket1,N_{fp}\rrbracket}$ be the Gauss-Legendre flux points on face $f$. Then

$\forall c\in\left[0,\infty\right[,$
\begin{equation}
R_{ei}\left(\mathbf{r}\right)=\left(-\hat{\nabla}\psi_{ei}\left(\mathbf{r}\right)\cdot\hat{\mathbf{n}}_{ei}+\mathlarger{\sum}_{f=1}^{3}\mathlarger{\sum}_{j=1}^{N_{fp}}\psi_{ei}\left(\mathbf{r}_{f}^{j}\right)\left(\hat{\mathbf{n}}_{ei}\cdot\mathbf{\hat{n}}_{fj}\right)\phi_{fj}\left(\mathbf{r}\right)\right),
\end{equation}
is independent of the parameter $\kappa$.
\end{thm}

\begin{proof}
The normal $\mathbf{\hat{n}}_{ei}$ is independent of $\kappa$. As a result, proving this Theorem is equivalent to showing that,
\begin{equation}\label{eq:vector show it is independent of kappa}
\mathbf{R}_{ei}\left(\mathbf{r}\right)=\left(-\hat{\nabla}\psi_{ei}\left(\mathbf{r}\right)+\mathlarger{\sum}_{f=1}^{3}\mathlarger{\sum}_{j=1}^{N_{fp}}\psi_{ei}\left(\mathbf{r}_{f}^{j}\right)\phi_{fj}\left(\mathbf{r}\right)\mathbf{\hat{n}}_{fj}\right),
\end{equation}
is independent of $\kappa$.

\begin{figure}[H]
\centering
\includegraphics[width=3.2in]{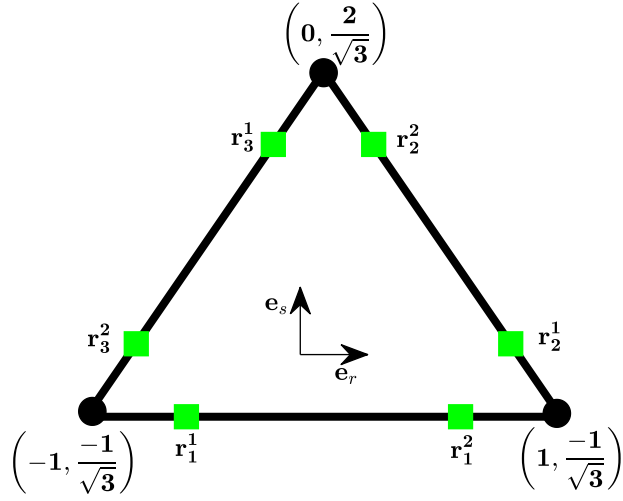}
\caption{Reference element for $p=1$, the green squares represent the FP of the Gauss-Legendre quadrature}
\label{fig:ref p1}
\end{figure}

The reference element is equilateral, thus the flux points are symmetric, Figure~\ref{fig:ref p1}. Moreover, the weights of the Gauss-Legendre quadrature are equal to 1 $\left(\omega_{1}=\omega_{2}=1\right)$.

\begin{table}[H]
\centering
\begin{tabular}{|c||c|c|c|c|c|c|}
\hline
$\left(r,s\right)$&$\mathbf{r}_{1}^{1}$&$\mathbf{r}_{1}^{2}$&$\mathbf{r}_{2}^{1}$&$\mathbf{r}_{2}^{2}$&$\mathbf{r}_{3}^{1}$&$\mathbf{r}_{3}^{2}$\\
\hline
$r$&$-\dfrac{1}{\sqrt{3}}$&$-r_{1}^{1}$&$\dfrac{1}{2}+\dfrac{\sqrt{3}}{6}$&$\dfrac{1}{2}-\dfrac{\sqrt{3}}{6}$&$-r_{2}^{2}$&$-r_{2}^{1}$\\
\hline
$s$&$r_{1}^{1}$&$r_{1}^{1}$&$-r_{2}^{2}$&$r_{2}^{1}$&$s_{2}^{2}$&$s_{2}^{1}$\\
\hline
\end{tabular}
\caption{Numerical values of coordinates of the flux points.}
\label{tab:FP coordinates p1}
\end{table}

As no mathematical formula, known to the authors, enables to simplify equation~\eqref{eq:vector show it is independent of kappa}, we expand the correction fields $\psi_{ei}$ and $\phi_{fj}$ through equation~\eqref{eq: correction fonction decomposition},
\begin{equation}
\mathbf{R}_{ei}\left(\mathbf{r}\right)=\left(\mathlarger{\sum}_{k=1}^{3}\sigma_{ei,k}^{\kappa}\left[-\mathbf{\hat{\nabla}}L_{k}\left(\mathbf{r}\right)+\mathlarger{\sum}_{f=1}^{3}\mathlarger{\sum}_{j=1}^{2}\mathlarger{\sum}_{n=1}^{3}\left(L_{k}\left(r_{f}^{j}\right)\sigma_{fj,n}^{c}L_{n}\left(\mathbf{r}\right)\mathbf{\hat{n}}_{fj}\right)\right]\right).
\end{equation}

Before expanding further, we write the analytical formula for the Dubiner basis $\left(L_{i}\right)_{i\in\llbracket 1,3\rrbracket}$~\cite{castonguay_new_2012},
\begin{equation}
\begin{array}{lll}
L_{1}\left(r,s\right)&=&\dfrac{1}{3^{1/4}}\\
L_{2}\left(r,s\right)&=&\dfrac{\sqrt{6}}{3^{1/4}}s\\
L_{3}\left(r,s\right)&=&\dfrac{\sqrt{6}}{3^{1/4}}r
\end{array}
\end{equation}

Then, we evaluate the coefficients $\left(\sigma_{ei,k}^{\kappa}\right)_{k\in\llbracket 1,3\rrbracket}$ via equation~\eqref{eq:ESFR formula kappa},

\begin{equation}
\mathbf{A}\mathbf{\sigma}^{\kappa}_{ei}=\mathbf{b}_{ei},
\end{equation}
where $\mathbf{A}$ has non-zero values only along its diagonal ($A_{11}=1$, $A_{22}=1+\kappa\left(D^{\left(2,1\right)}L_{2}\right)^{2}$, $A_{33}=1+\kappa\left(D^{\left(1,2\right)}L_{3}\right)^{2}$) and $b_{ei,j}=L_{j}\left(\mathbf{r}_{e}^{i}\right)$. To simplify the derivations, we denote $\alpha=\frac{1}{3^{1/4}}$. Hence, we obtain,
\begin{equation}
\begin{bmatrix}
\sigma_{ei,1}^{\kappa}\\
\sigma_{ei,2}^{\kappa}\\
\sigma_{ei,3}^{\kappa}
\end{bmatrix}
=
\begin{bmatrix}
\alpha\vspace*{0.2cm}\\
\dfrac{\alpha\sqrt{6}s_{e}^{i}}{1+6\alpha^{2}\kappa}\vspace*{0.2cm}\\
\dfrac{\alpha\sqrt{6}r_{e}^{i}}{1+6\alpha^{2}\kappa}
\end{bmatrix}.
\end{equation}

Only the last two terms depend on $\kappa$; therefore we only need to prove that the projection of $\mathbf{R}_{ei,1}\left(\mathbf{r}\right)$ on both $\mathbf{e}_{r}$ and $\mathbf{e}_{s}$ is independent of $\kappa$, where
\begin{equation}
\mathbf{R}_{ei,1}\left(\mathbf{r}\right)=\left(\mathlarger{\sum}_{k=2}^{3}\sigma_{ei,k}^{\kappa}\left[-\mathbf{\hat{\nabla}}L_{k}\left(\mathbf{r}\right)+\mathlarger{\sum}_{f=1}^{3}\mathlarger{\sum}_{j=1}^{2}\mathlarger{\sum}_{n=1}^{3}\left(L_{k}\left(\mathbf{r}_{f}^{j}\right)\sigma_{fj,n}^{c}L_{n}\left(\mathbf{r}\right)\mathbf{\hat{n}}_{fj}\right)\right]\right).
\end{equation}

In the following, we will show that these two projections are equal to $0$.

\textbf{Projection on $\mathbf{e}_{r}$}

We can observe from Figure~\ref{fig:ref p1} that the FP are symmetric. As a consequence the terms $\left(L_{k}\left(\mathbf{r}_{f}^{j}\right)\right)_{\left(f,j\right)\in\llbracket 1,3\rrbracket\times\llbracket 1,2\rrbracket}$ only require to be computed at the FP $\left(1,1\right)$, $\left(2,1\right)$ and $\left(2,2\right)$. Moreover $\mathbf{\hat{n}}_{1j}\cdot\mathbf{e}_{r}=0$. These properties yield,
\begin{equation}
\begin{array}{lllllll}
\mathbf{R}_{ei,1}\left(\mathbf{r}\right)\cdot\mathbf{e}_{r}&=&\dfrac{\alpha\sqrt{6}s_{e}^{i}}{1+6\alpha^{2}\kappa}\Bigg(&&&\dfrac{\sqrt{3}}{2}L_{2}\left(s_{2}^{1}\right)&\left(\alpha^{2}+\dfrac{6\alpha^{2}}{1+6\alpha^{2}c}\left(s_{2}^{1} s+r_{2}^{1} r\right)\right)\\
&&&&+&\dfrac{\sqrt{3}}{2}L_{2}\left(s_{2}^{2}\right)&\left(\alpha^{2}+\dfrac{6\alpha^{2}}{1+6\alpha^{2}c}\left(s_{2}^{2} s+r_{2}^{2} r\right)\right)\\
&&&&-&\dfrac{\sqrt{3}}{2}L_{2}\left(s_{2}^{2}\right)&\left(\alpha^{2}+\dfrac{6\alpha^{2}}{1+6\alpha^{2}c}\left(s_{2}^{2} s-r_{2}^{2} r\right)\right)\\
&&&&-&\dfrac{\sqrt{3}}{2}L_{2}\left(s_{2}^{1}\right)&\left(\alpha^{2}+\dfrac{6\alpha^{2}}{1+6\alpha^{2}c}\left(s_{2}^{1} s-r_{2}^{1} r\right)\right)\Bigg)\\
&&+\dfrac{\alpha\sqrt{6}r_{e}^{i}}{1+6\alpha^{2}\kappa}\Bigg(&-\sqrt{6}\alpha&+&\dfrac{\sqrt{3}}{2}L_{3}\left(r_{2}^{1}\right)&\left(\alpha^{2}+\dfrac{6\alpha^{2}}{1+6\alpha^{2}c}\left(s_{2}^{1} s+r_{2}^{1} r\right)\right)\\
&&&&+&\dfrac{\sqrt{3}}{2}L_{3}\left(r_{2}^{2}\right)&\left(\alpha^{2}+\dfrac{6\alpha^{2}}{1+6\alpha^{2}c}\left(s_{2}^{2} s+r_{2}^{2} r\right)\right)\\
&&&&+&\dfrac{\sqrt{3}}{2}L_{3}\left(r_{2}^{2}\right)&\left(\alpha^{2}+\dfrac{6\alpha^{2}}{1+6\alpha^{2}c}\left(s_{2}^{2} s-r_{2}^{2} r\right)\right)\\
&&&&+&\dfrac{\sqrt{3}}{2}L_{3}\left(r_{2}^{1}\right)&\left(\alpha^{2}+\dfrac{6\alpha^{2}}{1+6\alpha^{2}c}\left(s_{2}^{1} s-r_{2}^{1} r\right)\right)\Bigg)\\
&=&\dfrac{\alpha\sqrt{6}s_{e}^{i}}{1+6\alpha^{2}\kappa}\Bigg(&&&\sqrt{3}L_{2}\left(s_{2}^{1}\right)&\left(\dfrac{6\alpha^{2}}{1+6\alpha^{2}c}r_{2}^{1} r\right)\\
&&&&+&\sqrt{3}L_{2}\left(s_{2}^{2}\right)&\left(\dfrac{6\alpha^{2}}{1+6\alpha^{2}c}r_{2}^{2} r\right)\Bigg)\\
&&+\dfrac{\alpha\sqrt{6}r_{e}^{i}}{1+6\alpha^{2}\kappa}\Bigg(&-\sqrt{6}\alpha&+&\sqrt{3}L_{3}\left(r_{2}^{1}\right)&\left(\alpha^{2}+\dfrac{6\alpha^{2}}{1+6\alpha^{2}c}s_{2}^{1} s\right)\\
&&&&+&\sqrt{3}L_{3}\left(r_{2}^{2}\right)&\left(\alpha^{2}+\dfrac{6\alpha^{2}}{1+6\alpha^{2}c}s_{2}^{2} s\right)\Bigg).
\end{array}
\end{equation}

We now replace the values of $s_{f}^{j}$ with the values of $r_{f}^{j}$ according to Table~\ref{tab:FP coordinates p1} and expand the remaining functions $L_{k}$,
\begin{equation}
\begin{array}{lllllll}
\mathbf{R}_{ei,1}\left(\mathbf{r}\right)\cdot\mathbf{e}_{r}&=&\dfrac{\alpha\sqrt{6}s_{e}^{i}}{1+6\alpha^{2}\kappa}\Bigg(&&&-3\sqrt{2}\alpha r_{2}^{2}&\left(\dfrac{6\alpha^{2}}{1+6\alpha^{2}c}r_{2}^{1} r\right)\\
&&&&+&3\sqrt{2}\alpha r_{2}^{1}&\left(\dfrac{6\alpha^{2}}{1+6\alpha^{2}c}r_{2}^{2} r\right)\Bigg)\\
&&+\dfrac{\alpha\sqrt{6}r_{e}^{i}}{1+6\alpha^{2}\kappa}\Bigg(&-\sqrt{6}\alpha&+&3\sqrt{2}\alpha r_{2}^{1}&\left(\alpha^{2}-\dfrac{6\alpha^{2}}{1+6\alpha^{2}c}r_{2}^{2} s\right)\\
&&&&+&3\sqrt{2}\alpha r_{2}^{2}&\left(\alpha^{2}+\dfrac{6\alpha^{2}}{1+6\alpha^{2}c}r_{2}^{1} s\right)\Bigg).
\end{array}
\end{equation}

We finally obtain
\begin{equation}\label{eq:final equation er}
\mathbf{R}_{ei,1}\left(\mathbf{r}\right)\cdot\mathbf{e}_{r}=-\dfrac{6\alpha^{2}r_{e}^{i}}{1+6\alpha^{2}\kappa}\left(-1+\sqrt{3}\alpha^{2}\left(r_{2}^{1}+r_{2}^{2}\right)\right),
\end{equation}
since $\sqrt{3}\alpha^{2}=$ and $\left(r_{2}^{1}+r_{2}^{2}\right)$ are equal to 1, we retrieve the expected result.

The projection on $\mathbf{e}_{s}$ is similar hence only the main steps will be given.

\textbf{Projection on $\mathbf{e}_{s}$}

As $\mathbf{\hat{n}}_{1j}\cdot\mathbf{e}_{s}=-1$, additional terms need to be derived.

\begin{equation}
\begin{array}{lllllll}
\mathbf{R}_{ei,1}\left(\mathbf{r}\right)\cdot\mathbf{e}_{s}&=&\dfrac{\alpha\sqrt{6}s_{e}^{i}}{1+6\alpha^{2}\kappa}\Bigg(&-\sqrt{6}\alpha&-&L_{2}\left(s_{1}^{1}\right)&\left(\alpha^{2}+\dfrac{6\alpha^{2}}{1+6\alpha^{2}c}\left(s_{1}^{1} s+r_{1}^{1} r\right)\right)\\
&&&&-&L_{2}\left(s_{1}^{1}\right)&\left(\alpha^{2}+\dfrac{6\alpha^{2}}{1+6\alpha^{2}c}\left(s_{1}^{1} s-r_{1}^{1} r\right)\right)\\
&&&&+&\dfrac{1}{2}L_{2}\left(s_{2}^{1}\right)&\left(\alpha^{2}+\dfrac{6\alpha^{2}}{1+6\alpha^{2}c}\left(s_{2}^{1} s+r_{2}^{1} r\right)\right)\\
&&&&+&\dfrac{1}{2}L_{2}\left(s_{2}^{2}\right)&\left(\alpha^{2}+\dfrac{6\alpha^{2}}{1+6\alpha^{2}c}\left(s_{2}^{2} s+r_{2}^{2} r\right)\right)\\
&&&&+&\dfrac{1}{2}L_{2}\left(s_{2}^{2}\right)&\left(\alpha^{2}+\dfrac{6\alpha^{2}}{1+6\alpha^{2}c}\left(s_{2}^{2} s-r_{2}^{2} r\right)\right)\\
&&&&+&\dfrac{1}{2}L_{2}\left(s_{2}^{1}\right)&\left(\alpha^{2}+\dfrac{6\alpha^{2}}{1+6\alpha^{2}c}\left(s_{2}^{1} s-r_{2}^{1} r\right)\right)\Bigg)\\
&&+\dfrac{\alpha\sqrt{6}r_{e}^{i}}{1+6\alpha^{2}\kappa}\Bigg(&&-&L_{3}\left(r_{1}^{1}\right)&\left(\alpha^{2}+\dfrac{6\alpha^{2}}{1+6\alpha^{2}c}\left(s_{1}^{1} s+r_{1}^{1} r\right)\right)\\
&&&&+&L_{3}\left(r_{1}^{1}\right)&\left(\alpha^{2}+\dfrac{6\alpha^{2}}{1+6\alpha^{2}c}\left(s_{1}^{1} s-r_{1}^{1} r\right)\right)\\
&&&&+&\dfrac{1}{2}L_{3}\left(r_{2}^{1}\right)&\left(\alpha^{2}+\dfrac{6\alpha^{2}}{1+6\alpha^{2}c}\left(s_{2}^{1} s+r_{2}^{1} r\right)\right)\\
&&&&+&\dfrac{1}{2}L_{3}\left(r_{2}^{2}\right)&\left(\alpha^{2}+\dfrac{6\alpha^{2}}{1+6\alpha^{2}c}\left(s_{2}^{2} s+r_{2}^{2} r\right)\right)\\
&&&&-&\dfrac{1}{2}L_{3}\left(r_{2}^{2}\right)&\left(\alpha^{2}+\dfrac{6\alpha^{2}}{1+6\alpha^{2}c}\left(s_{2}^{2} s-r_{2}^{2} r\right)\right)\\
&&&&-&\dfrac{1}{2}L_{3}\left(r_{2}^{1}\right)&\left(\alpha^{2}+\dfrac{6\alpha^{2}}{1+6\alpha^{2}c}\left(s_{2}^{1} s-r_{2}^{1} r\right)\right)\Bigg).
\end{array}
\end{equation}
After further simplifications we obtain,
\begin{equation}
\begin{array}{lll}
\mathbf{R}_{ei,1}\left(\mathbf{r}\right)\cdot\mathbf{e}_{s}&=&\dfrac{6\alpha^{2}s_{e}^{i}}{1+6\alpha^{2}\kappa}\left((-1+\alpha^{2}\left(-2s_{1}^{1}+s_{2}^{1}+s_{2}^{2}\right)\right)\\
&+&\dfrac{6\alpha^{2}s_{e}^{i}}{1+6\alpha^{2}\kappa}\dfrac{6\alpha^{2}s}{1+6\alpha^{2}c}\left(-2\left(s_{1}^{1}\right)^{2}+\left(s_{2}^{1}\right)^{2}+\left(s_{2}^{2}\right)^{2}\right)\\
&+&\dfrac{6\alpha^{2}r_{e}^{i}}{1+6\alpha^{2}\kappa}\dfrac{6\alpha^{2}r}{1+6\alpha^{2}c}\left(-2\left(r_{1}^{1}\right)^{2}+\left(r_{2}^{1}\right)^{2}+\left(r_{2}^{2}\right)^{2}\right).
\end{array}
\end{equation}
Replacing $r_{f}^{j}$ and $s_{f}^{j}$ with their numerical values in Table~\ref{tab:FP coordinates p1}, the last two lines are equal to 0. Futhermore,
\begin{equation}
\begin{array}{lll}
\alpha^{2}\left(-2s_{1}^{1}+s_{2}^{1}+s_{2}^{2}\right)&=&\dfrac{1}{\sqrt{3}}\left(\dfrac{2}{\sqrt{3}}+\left(\dfrac{\sqrt{3}}{6}-\dfrac{1}{2}\right)+\left(\dfrac{1}{2}+\dfrac{\sqrt{3}}{6}\right)\right)\\
&=&1
\end{array}
\end{equation}

Hence we have $\mathbf{R}_{ei,1}\left(\mathbf{r}\right)\cdot\mathbf{e}_{s}=0$.
\end{proof}

\begin{cor}
The diffusion equation is independent of $\kappa$ when employing the IP or BR2 numerical fluxes with a $p=1$ interpolation.
\end{cor}
\begin{proof}
We let the reader refer to the proof of Postulate~\ref{pos:independent}.
\end{proof}

\section*{References}

\bibliography{bibliographie}

\end{document}